\documentclass{amsart}
\usepackage[T1]{fontenc}
\usepackage{times}
\usepackage{amssymb,verbatim}
\theoremstyle{plain}
\usepackage[T1]{fontenc}
\usepackage[utf8]{inputenc}
\usepackage[english]{babel}
\usepackage{amsthm}
\usepackage{amsmath}
\usepackage{amstext}
\usepackage{float}
\usepackage{tikz}
\usepackage{enumerate}
\usepackage{multicol}
\usetikzlibrary{matrix}
\usepackage[mathscr]{euscript}
\usepackage{commath}
\usepackage[backref=page]{hyperref}
\usepackage{tikz-cd}
\usetikzlibrary{matrix}

\usepackage{algorithm}
\usepackage{algorithmic}

\usepackage{booktabs}
\usepackage{array}
\usetikzlibrary{matrix,arrows.meta,calc}

\newtheorem{thm}{Theorem}[section]
\newtheorem*{conj*}{Conjecture}

\newtheorem{dfn}[thm]{Definition}
\newtheorem{lemma}[thm]{Lemma}
\newtheorem{prop}[thm]{Proposition}
\newtheorem{cor}[thm]{Corollary}
\newtheorem{problem}{Problem}

\newtheorem{THM}{Theorem}

\theoremstyle{remark}
\newtheorem{example}[thm]{Example}
\newtheorem{remark}[thm]{Remark}

\newtheorem*{thmC}{\textbf{Theorem C}}

\newcommand{\Log}{\textsf{Log}}
\newcommand{\mb}{\mathbb}
\newcommand{\mc}{\mathcal}
\newcommand{\C}{\mb C}
\newcommand{\Pj}{\mb P}
\newcommand{\F}{\mc F}
\newcommand{\G}{\mc G}
\newcommand{\D}{\mc D}
\newcommand{\h}{\mc H}
\newcommand{\TF}{{T_\F}}
\newcommand{\TD}{{T_\D}}
\newcommand{\TG}{{T_\G}}
\newcommand{\Th}{{T_\h}}
\newcommand{\TM}{{T_M}}

\DeclareMathOperator{\codim}{codim}
\DeclareMathOperator{\sing}{sing}
\DeclareMathOperator{\rank}{rank}
\DeclareMathOperator{\Aut}{Aut}
\DeclareMathOperator{\Ker}{Ker}
\DeclareMathOperator{\supp}{Supp}

\newenvironment{acknowledgement}{\textbf{Acknowledgments}}

\numberwithin{equation}{section}

\addtocounter{section}{0}             
\numberwithin{equation}{section}       

\title{Splitting aspects of holomorphic distributions with locally free tangent sheaf}

\author[R.C.~da~Costa]{Raphael Constant da Costa}
\address{UERJ, Universidade do Estado do Rio de Janeiro, Rua São Francisco ˜Xavier, 524, Maracanã, 20550-900, Rio de Janeiro, Brazil. }
\email{raphaelconstant@ime.uerj.br}



\begin{document}

\begin{abstract}
In this work, we mainly deal with a two-dimensional singular holomorphic distribution $\D$ defined on $M$, where $M$ represents a complex manifold of dimension $n \geq 3$ or a germ of it, whose tangent sheaf $\TD$ is locally free. As is well known, when $M=\mathbb{P}^n$ or $M=(\mathbb{C}^n,0)$, there is a one-dimensional foliation $\G$ on $M$ tangent to $\D$ and we study whether $\TD$ splits starting from it. In both cases, we provide sufficient conditions on $\G$ so that there is another one-dimensional foliation $\h$ on $M$ tangent to $\D$, such that their respective tangent sheaves satisfy the splitting relation $\TD=\TG \oplus \Th$. We introduce a concept of local division of $\D$ by $\G$, exhibiting a characterization of $\mathcal{S}(\G,\D)$, the set of points $p \in M$ where $\G$ does not locally divide $\D$ at $p$. Furthermore, for $M=\mathbb{P}^n$ we prove that the existence of such $\h$ is equivalent to $\mathcal{S}(\G,\D)=\emptyset$. Additionally, given a codimension one holomorphic foliation $\mathcal{F}$ on $\mathbb{P}^3$ with locally free tangent sheaf, we show that $\TF$ splits provided there exists a nonzero holomorphic vector field on $\mathbb{P}^3$ tangent to $\mathcal{F}$. We obtain division results involving holomorphic differential forms and vector fields, and some of them could serve as alternatives to classical results coming from the De Rham-Saito Division Lemma, while others can be applied in situations not covered by the latter.
\end{abstract}

\maketitle

\setcounter{tocdepth}{1}

\tableofcontents

\section{Introduction}

Let $\D$ be a singular holomorphic distribution of dimension $k$ on $\Pj^n$, $n \geq 3$. Assume that its tangent sheaf $\TD$ is locally free. A natural and important question is to determine whether there exists a splitting relation as
\[
\TD=\bigoplus_{i=1}^k\mathcal{O}_{\Pj^n}(e_i),
\]
for some integers $e_i$. 

This problem becomes particularly interesting in the case where the distribution is integrable, i.e., when it defines a foliation, denoted by $\mathcal{F}$ instead of $\mathcal{D}$. The importance of this matter can be seen in \cite{CaCeGiLi} and \cite{CukiermanPereira}, where it is proved that some foliations with a special type of singular set and split tangent sheaf are stable under small deformations. This makes it possible to obtain irreducible components of the space of foliations on projective spaces, which can also be found in \cite{Lizarbe17} and \cite{Constant}.

For codimension one foliations on $\Pj^3$, in \cite{GiraldoCollantes} it was proved that the tangent sheaf of $\F$ is locally free if and only if $\sing(\F)$ (considered as a scheme) is a curve. Furthermore, $\TF$ splits if and only if $\sing(\F)$ is an \textit{arithmetically Cohen–Macaulay} curve. These results were generalized later in \cite{MaMaRe2015} for higher dimensions.

The following question was posed in \cite{CaCeGiLi}:

\begin{problem}\label{P:loc_split_problem}
Let $\F$ be a codimension one foliation on $\mathbb{P}^3$, where $\TF$ is locally free. Does $\TF$ split?
\end{problem}

In \cite{CalvoArielFederico2019}, the authors provide sufficient conditions for a codimension one foliation on $\Pj^3$ with locally free tangent sheaf to split. They use the relation between the Kupka scheme and the unfoldings ideal associated to a codimension one foliation. They also show that if $\F$ is a codimension one foliation on $\Pj^n$ such that $\sing(\F)$ is smooth of dimension $(n-2)$, then $\TF$ splits (in fact it is a degree $0$ rational foliation), generalizing \cite[Theorem 3.11]{10.1093/imrn/rny251}. In \cite{calvoandrade2024dimension}, the authors show that if $\D$ is a two-dimensional distribution on $\Pj^4$ of degree at most $2$ with locally free tangent sheaf such that $\sing(\D)$ (considered as a scheme) has pure dimension one, then $\TD$ splits.

Regarding Problem~\ref{P:loc_split_problem}, it is well known that there are nonintegrable distributions whose tangent sheaf is locally free and does not split. We can take, for example, a regular distribution of codimension one and degree $0$ on $\Pj^3$, i.e., a holomorphic contact structure on $\Pj^3$. Very recently, it was given in \cite{FaMaJe2023} an answer (in the negative) to Problem~\ref{P:loc_split_problem}.

\subsection{Statement of the main results}

We know that there are foliations $\F$ on projective spaces such that $\TF$ is locally free but $\TF$ does not split, see \cite{FaMaJe2023}. Given a foliation $\F$ on $\mathbb{P}^n$ with locally free tangent sheaf, we would like to find extra conditions for $\TF$ to split. Although part of our original analysis was for codimension $1$ foliations on $\mathbb{P}^3$, it naturally extends to distributions of dimension $2$ on $\mathbb{P}^n$, $n \geq 3$. Most of the time we deal with the following:

\begin{problem}\label{P:paper_problem}
Let $M$ represent $\mathbb{P}^n$ or $(\mathbb{C}^n,0)$. Given a two-dimensional distribution $\D$ on $M$ with locally free tangent sheaf, and a one-dimensional foliation $\G$ on $M$ tangent to $\D$, under what conditions can we guarantee that there is another one-dimensional foliation $\h$ on $M$ tangent to $\D$ so that $\TD=\TG \oplus \Th$?
\end{problem}

\begin{remark}
For the sake of clarity, by a $k$-dimensional distribution $\D$ on $(\mathbb{C}^n,0)$ we mean a germ of a $k$-dimensional distribution $\D$ at $0$. Such germ can be represented by a $k$-dimensional distribution $\D_U$ defined on a sufficiently small polydisc $U$ containing $0$. We write $\TD$ to represent $(T_{\mathcal{D}_U})_0$, the stalk of the sheaf $T_{\mathcal{D}_U}$ at $0$. Thus, in this case $\TD$ represents a $\mathcal{O}_{\C^n,0}$-module rather than a sheaf. We say that the tangent sheaf of $\D$ is locally free if there exists an open set $V \subset U$ such that the tangent sheaf of the restriction of $\D_U$ to $V$ is locally free. Since $T_{\mathcal{D}_U}$ is a coherent sheaf, it follows that the tangent sheaf of $\D$ is locally free if and only if $\TD$ is a free $\mathcal{O}_{\C^n,0}$-module.
\end{remark}

\begin{remark}
If $\mathcal{D}$ is a singular holomorphic distribution on a complex manifold $M$, and $\mathcal{G}$ is a one-dimensional foliation on $M$ tangent to $\mathcal{D}$, then the pair $(\mathcal{G},\mathcal{D})$ is a particular case of a $2$-flag of holomorphic distributions; see \cite{MR4520047} and the references therein.
\end{remark}

We start by approaching Problem~\ref{P:paper_problem} in local terms.

\begin{dfn}
Let $v$ be a germ of a holomorphic vector field defined at $0 \in \C^n$. Assume that $0$ is a singularity of $v$, i.e., $v(0)=0$. We define the linear rank of $v$ (at $0$) as the rank of the linear transformation $Dv(0)$. 
\end{dfn}

Writing $\hat{v}_1$ for the linear part of $v$, observe that the linear rank of $v$ is at most $1$ if and only if $\hat{v}_1=fu$, where $f$ is a homogeneous polynomial of degree one in $\C^n$ and $u$ is a germ of a constant vector field defined at $0 \in \C^n$. Thus, a germ of a vector field with a singularity at $0$ and linear rank at most $1$ is, in some sense, degenerate.

Note that if $\varphi:(\C^n,0) \to (\C^n,0)$ is a germ of biholomorphism and $u=\varphi_* v$, then the linear maps $Dv(0)$ and $Du(0)$ are similar. As a consequence, $v$ and $u$ have the same linear rank. This invariance allows us to define the linear rank of a vector field $v$, defined on a complex manifold, at a singularity of $v$ in a natural way.

Moreover, if $f$ is a germ of holomorphic function defined at $0$ with $f(0) \neq 0$, then the linear ranks of $v$ and $u=fv$ are the same. This allows one to define the linear rank of a one-dimensional foliation $\G$, defined on a complex manifold, at a singularity of $\G$ in the obvious way. Finally, if $\sing(\G) \neq \emptyset$, we also define the linear rank of $\G$ as the least possible value of the rank of $\G$ at $p$ for $p \in \sing(\G)$.

The following results provide answers to Problem~\ref{P:paper_problem}, the first one being in local terms.

\begin{THM}\label{T:A}
Let $\D$ be a germ of a $k$-dimensional distribution at $0 \in \mathbb C^n$ with locally free tangent sheaf, and let $\G$ be a germ of a nonregular one-dimensional foliation at $0$, tangent to $\D$ and with linear rank at least $k$, where $n \geq 3$ and $1 \leq k < n$. Then exactly one of the following assertions holds:
\begin{enumerate}
\item\label{1:T:A} There exist germs of one-dimensional foliations $\h_1,\ldots,\h_{k-1}$ defined at $0$ such that
\[
\TD=\TG \oplus T_{\h_1}\oplus \cdots \oplus T_{\h_{k-1}}.
\] 
\item\label{2:T:A} The germ of distribution $\D$ is regular. In this case, the linear rank of $\G$ is $k$.
\end{enumerate}
Moreover, under the assumptions of the theorem, \eqref{1:T:A} is equivalent to $\sing(\G) \subset \sing(\D)$.
\end{THM}

In the context of Theorem~\ref{T:A}, a careful look at its proof shows that if $\G$ is regular, then the equality between $\mathcal{O}_{\C^n,0}$-modules appearing in part \eqref{1:T:A} still holds. The reason for excluding this case is twofold: We have defined the linear rank only for germs of vector fields with a singularity at $0$; moreover, regarding part \eqref{2:T:A}, $\D$ could be either regular or not, depending on the situation. Thus, the two parts would no longer be mutually exclusive.

\begin{THM}\label{T:B}
Let $\D$ be a two-dimensional distribution on $\Pj^n$, $n \geq 3$, with locally free tangent sheaf. Assume that there exists a one-dimensional foliation $\G$ on $\Pj^n$ tangent to $\D$ and with linear rank at least $2$. Then the following are equivalent:
\begin{enumerate}
\item\label{1:T:B} Every irreducible component of $\sing(\G)$ of dimension $n-2$ intersects $\sing(\D)$.
\item\label{2:T:B} Every irreducible component of $\sing(\G)$ intersects $\sing(\D)$.
\item\label{3:T:B} $\sing(\G) \subset \sing(\D)$.
\item\label{4:T:B} There is a one-dimensional foliation $\h$ on $\mathbb{P}^n$ such that $\TD=\TG \oplus \Th$.
\end{enumerate}
In particular, if $\D$ is a two-dimensional distribution with locally free tangent sheaf on $\mathbb{P}^n, n\geq 4$, such that $\dim(\sing(\D))\geq 2$, and there exists a one-dimensional foliation $\G$ tangent to $\D$ with linear rank at least $2$, then $\TD$ splits.
\end{THM}

Regarding Theorem~\ref{T:B}, there is a stronger version if we assume that the linear rank of $\G$ is at least $3$, see the particular case of Theorem~\ref{T:version_lr_general} below. Another result in the spirit of Theorem~\ref{T:B}, without the assumption that $\TD$ is locally free, can be found in Corollary~\ref{C:T:B_without_loc_free}.

As an application of the results obtained, we have the following:

\begin{THM}\label{T:structure}
Let $\F$ be a codimension one holomorphic foliation on $\mathbb P^3$ with locally free tangent sheaf. If there exists a nonzero holomorphic vector field on $\mathbb P^3$ tangent to $\F$, then $\TF$ splits.
\end{THM}

For a codimension one distribution $\D$ on $\Pj^3$ with locally free tangent sheaf, we also show that $\TD$ splits, provided that $\deg(\D) \geq 1$ and there exists a nonnilpotent holomorphic vector field on $\Pj^3$ tangent to $\D$, see Corollary~\ref{C:result_for-distrib}.

Finally, our computations lead us to some division results involving vector fields and differential forms. See Corollary~\ref{C:crite_div_loc_free} and Propositions~\ref{P:new_division_result} and \ref{P:new_div_dual}. They are easily verifiable in the sense that it is equivalent to finding the rank of a linear transformation.

\subsection{Notation} We will denote by:
\begin{enumerate}[-]
\item $M$: a complex manifold of dimension $n$. In the vast majority of cases, we will assume that $n \geq 3$. In some cases, we will have $M=\mathbb{P}^n$, $n \geq 3$.
\item $(N,p)$: the germ of a complex analytic subset $N \subset M$ at a point $p \in M$. Most of the time we will deal with $(\C^n,0)$.
\item $\Omega^k(U)$ and $\mathcal{X}(U)$: the sets of holomorphic $k$-forms and holomorphic vector fields on $U$, respectively, where $U$ is $(M,p)$ or an open set of $M$.
\item $\mathcal{O}_M$: the sheaf of germs of holomorphic functions on $M$. We also denote by $\mathcal{O}_M^*$ the sheaf of nonvanishing holomorphic functions on $M$.
\item $\mathcal{O}_{M,p}$: The stalk of $\mathcal{O}_M$ at $p \in M$, i.e., the ring of germs of holomorphic functions defined at $p$. We also denote by $\mathcal{O}^*_{M,p}$ the stalk of $\mathcal{O}_M^*$ at $p \in M$.
\item Given $v \in \mathcal{X}(\C^n,0)$, we will write $v=\sum_{i=0}^{\infty} \hat{v}_i$ for the decomposition of $v$ into homogeneous polynomials vector fields, i.e., for each $i=0,1,2,\ldots$, $\hat{v}_i$ is a vector field with coefficient polynomials which are homogeneous of degree $i$. In parallel, we write $\omega=\sum_{i=0}^{\infty}\hat{\omega}_i$ for $\omega \in \Omega^k(\C^n,0)$.
\end{enumerate}

\

\begin{acknowledgement}
The author is supported by CNPq (Grant numbers 402936/2021-3 and 408687/2023-1), an also by the UERJ Prodocência Program. I would like to thank the anonymous referees for many important comments and suggestions, which helped to improve the manuscript. In particular, Proposition~\ref{P:eq_subbundle} originated from a question posed by one of the referees. We have used the software Maple (Version 18.02) on five occasions to solve huge systems of linear equations.
\end{acknowledgement}

\section{Basic definitions}
Let $M$ be a complex manifold of dimension $n$ and $1 \leq k < n$. A singular holomorphic distribution $\D$ of dimension $k$ (or codimension $n-k$) can be defined by a family $\{(U_j,\omega_j)\}_{j \in J}$ where:
\begin{enumerate}
\item 
$\mathcal{U}=\{U_j\}_{j \in J}$ is an open covering of $M$.
\item $\omega_j \not\equiv 0$ is a holomorphic ($n-k$)-form defined on the open set $U_j$, which is locally decomposable off the singular set (LDS); see \cite[Definition~1.2.1]{MR1842027}. By definition, this means that for each point $z \in U_j \setminus \sing(\omega_j)$, where $\sing(\omega_j):=\{z \in U_j: \omega_j(z) =0\}$, there exist a neighborhood $V \subset U_j$ of $z$ and holomorphic $1$-forms $\eta_1,\ldots,\eta_{n-k}$ on $V$ such that 
\begin{equation}\label{E:loc_dec}
{\omega_j}_{|V}=\eta_1 \wedge \cdots \wedge \eta_{n-k}.
\end{equation}
\item If $U_{ij}:=U_i \cap U_j \neq \emptyset$, there is a nonzero holomorphic function defined on $U_{ij}$, $h_{ij} \in \mathcal{O}_M^* (U_{ij})$, satisfying
\begin{equation}\label{E:compatibility}
{\omega_i}_{|U_{ij}}=h_{ij}{\omega_j}_{|U_{ij}}.
\end{equation}
\item
$\codim (\sing (\omega_j)) \geq 2,\forall j \in J$.
\end{enumerate}

The singular set of $\D$ is the analytic subset of $M$ given by 
\[
\sing(\D)=\bigcup_j \{z \in M : \omega_j(z)=0\}.
\]
Note that the definition of $\sing(\D)$ is possible due to \eqref{E:compatibility}. If $\sing(\D) = \emptyset$, we say that the distribution $\D$ is regular.

If $T_p M$ denotes the holomorphic tangent space of $M$ at $p \in M$, it follows that we can define on $M \setminus \sing(\D)$ a distribution of $k$-planes
\[
\D(p)=\{v \in T_p M:i_v \omega_j(p)=0\},
\]
whenever $p \in U_j \setminus \sing(\D)$. In fact, if $\eta_1,\ldots,\eta_{n-k}$ are as in \eqref{E:loc_dec}, we have
\[
\D(p)=\bigcap_{1 \leq l \leq n-k}\Ker(\eta_l (p)).
\]

About decomposition \eqref{E:loc_dec}, if additionally for every $j \in J$ the following Frobenius integrability condition is satisfied on $V$:
\begin{equation}\label{E:integrability}
d\eta_i \wedge \omega_j=0,i=1,\ldots,n-k,
\end{equation}
we say that the distribution $\D$ is integrable, and call it a holomorphic singular foliation of dimension $k$ on $M$. In this case we usually denote such a distribution by $\F$. 

A foliation $\F$ provides a decomposition of $M \setminus \sing(\F)$ into a union of disjoint connected immersed submanifolds $\{\mathcal{L}_\alpha\}_{\alpha \in A}$ of dimension $k$, called the leaves of $\F$, in such a way that if $p \in \mathcal{L}_{\alpha}$, then the tangent space to $\mathcal{L}_{\alpha}$ at $p$ is given by $T_p \mathcal{L}_{\alpha}=\D(p)$.

\begin{remark}\label{R:iso_X_O}
We will frequently use the well-known fact that, given $p \in M$ and a nonvanishing germ of a top-degree form $\Theta \in \Omega^n(M,p)$, there is an isomorphism of $\mathcal{O}_{M,p}$-modules between $\mathcal{X}(M,p)$ and $\Omega^{n-1}(M,p)$ given by $X \mapsto i_X \Theta$. From this, it is also possible to show that if $U \subset M$ is an open set for which there exists a nonvanishing $\Theta \in \Omega^n(U)$, then the map $\mathcal{X}(U) \to \Omega^{n-1}(U)$ given by $X \mapsto i_X \Theta$ is an isomorphism of $\mathcal{O}_M(U)$-modules.
\end{remark}

In the case $k=1$, in view of Remark~\ref{R:iso_X_O}, by shrinking the open sets $U_j$ if necessary, we may assume that $\omega_j \in \Omega^{n-1}(U_j)$ is of the form
\begin{equation}\label{E:one_form_vector}
\omega_j = i_{v_j}\Theta_j, 
\end{equation}
where $v_j \in \mathcal{X}(U_j)$ and $\Theta_j \in \Omega^n(U_j)$ is nonvanishing. 

On $U_{ij}$, we have $\Theta_i = f_{ij}\Theta_j$, where $f_{ij} \in \mathcal{O}_M^*(U_{ij})$. Together with \eqref{E:compatibility}, this shows that a one-dimensional distribution $\F$ can equivalently be described by a collection of vector fields $v_j$ defined on the open sets $U_j$ of an open covering $\mathcal{U}$ of $M$, with $\codim(\sing(v_j)) \geq 2$, and nonvanishing holomorphic functions $g_{ij} \in \mathcal{O}_M^*(U_{ij})$ satisfying
\[
{v_i}_{|U_{ij}} = g_{ij}{v_j}_{|U_{ij}}.
\]

Observe that the relation $g_{ij}=h_{ij} / f_{ij}$ on $U_{ij}$ establishes the connection between the two definitions of a one-dimensional distribution, namely those given by $(n-1)$-forms and by vector fields. From \eqref{E:one_form_vector}, we obtain that
\[
\sing(\D)=\bigcup_j \{z \in M : v_j(z)=0\}.
\]

It is well known that a one-dimensional distribution is always integrable, see also Example~\ref{Ex:dist_1_inte} below. The integral curves of all vector fields $v_j$ are glued together forming the leaves of $\F$.

\begin{remark}\label{R:vec_field_def_fol}
It is clear from the definition that the singular set of a distribution on a complex manifold $M$ always has codimension at least $2$. Throughout this article, whenever we say that a $k$-dimensional distribution on an open set $U \subset M$ is defined by $\omega \in \Omega^{n-k}(U)$, we always assume that $\codim(\sing(\omega)) \geq 2$. In some situations, we will consider a nonzero meromorphic vector field $v$ on $M$, possibly with codimension one zeros, defining a one-dimensional foliation $\mathcal{G}$ as follows: For every $p \in M$, there exists a neighborhood $U_p$ of $p$ such that $v=f \tilde{v}_p$, where $f$ is a meromorphic function and $\tilde{v}_p$ is a holomorphic vector field, both defined on $U_p$, with $\codim(\sing(\tilde{v}_p)) \geq 2$. The one-dimensional foliation $\G$ is given by the collection $\{(U_p,\tilde{v}_p)\}_{p \in M}$. Of course, $v$ defines a one-dimensional (possibly nonregular) foliation $\h$ on $M \setminus (\supp(D_0) \cup  \supp(D_{\infty}))$, where $\supp(D_0)$ and $\supp(D_{\infty})$ denote the support of the divisor of zeros and poles of $v$, respectively. We say that $\G$ is the extension of $\h$ to $M$.
\end{remark}

Denote by $\TM$ the tangent sheaf of $M$. The tangent sheaf of $\D$, denoted by $\TD$, is the subsheaf of $\TM$ whose sections over an open set $U \subset M$ is given by
\[
\TD(U)=\{v \in \TM(U): i_v \omega_j=0 \text{ on }U \cap U_j\}.
\]
Note that $(\TD)_p$, the stalk of $\TD$ at $p \in M$, consists of germs of vector fields at $p$ tangent to $\D$. 

\begin{remark}\label{R:tang_sheaf_properties}
Recall that a theorem due to Oka guarantees that a holomorphic sheaf on $M$ is coherent if it is given locally as the kernel of a morphism between free holomorphic sheaves. After writing the definition of $\TD$ in local coordinates, we see that $\TD$ is given locally by a finite intersection of the kernels of morphisms between free holomorphic sheaves. This implies that $\TD$ is coherent, since it is given locally by a finite intersection of coherent sheaves. From the definition of $\TD$, a straightforward verification shows that the coherent sheaf $\TM/\TD $ is torsion-free. Furthermore, since $T_M$ is locally free and $\TM/\TD $ is torsion-free, we have that $\TD$ is a reflexive coherent sheaf, see \cite[Proposition~1.1]{Hartshorne1980}. In this way, we recover the well-known definition of a singular distribution found, for instance, in \cite{10.1093/imrn/rny251}: A distribution $\D$ on $M$ is given by a coherent subsheaf $\TD$ of $\TM$ such that the quotient $\TM/\TD $ is torsion-free.
\end{remark}

\subsection{Distributions on $\Pj^n$}
For a $k$-dimensional distribution $\D$ on $\mathbb{P}^n$, there is an important discrete invariant called the degree of the distribution, denoted by $\deg(\D)$. By definition, it is the degree of the locus of tangency of $\D$ with a generic linear subspace $\Pj^{n-k} \subset \Pj^n$. The distribution $\D$ can be defined in $\C^{n+1}$ by a homogeneous polynomial ($n-k$)-form
\[
\Omega=\sum_{0\leq i_1 < \cdots < i_{n-k} \leq n}A_{i_1,\ldots,i_{n-k}}(x_0,\ldots,x_n)dx_{i_1} \wedge \cdots \wedge dx_{i_{n-k}}
\]
which is LDS and satisfies $i_R \Omega=0$, where
\[
R=x_0 \frac{\partial }{\partial x_0}+\cdots+x_n \frac{\partial }{\partial x_n}
\]
is the radial vector field, $A_{i_1,\ldots,i_{n-k}}(x_0,\ldots,x_n)$ is a homogeneous polynomial of degree $\deg(\D)+1$ and $\codim (\sing (\Omega)) \geq 2$, see for example \cite{CukiermanPereira}. The integrability of $\D$ is equivalent to the integrability of $\Omega$.

The singular set of $\D$ is the image, under the canonical projection $\Pi:\C^{n+1} \setminus \{0\} \to \mathbb{P}^n$, of the singular set of $\Omega$ in $\C^{n+1}$. Similarly, the $k$-planes of the distribution $\D$ on $\Pj^n$ are the images, under the same map, of the ($k+1$)-planes on $\C^{n+1}$ induced by $\Omega$. If $\D$ is integrable, a similar result holds for its leaves.

When $k=1$, alternatively a one-dimensional foliation $\F$ on $\Pj^n$ can be defined in $\C^{n+1}$ by the vector field
\[
V=V_0(x_0,\ldots,x_n)\frac{\partial}{\partial x_0}+\cdots+V_n(x_0,\ldots,x_n)\frac{\partial}{\partial x_n},
\]
where $V_i(x_0,\ldots,x_n)$, $i=0,\ldots,n$, is a homogeneous polynomial of degree $\deg(\F)$. Although necessary, the condition $\codim(\sing(V)) \geq 2$ is not sufficient to define such a one-dimensional foliation. In fact, $\sing(\F)$ is the image, under $\Pi$, of the set
\[
\{ p \in \mathbb C^{n+1} \setminus \{0\} : V(p) \wedge R(p)=0\}.
\]

\subsection{Distributions on $\mathbb{P}^n$ with split tangent sheaf}

Given a distribution $\D$ of dimension $k$ on $\Pj^n$, $n \geq 3$, we say that the tangent sheaf of $\D$ splits if
\[
\TD=\bigoplus_{i=1}^k\mathcal{O}_{\Pj^n}(e_i),
\]
for some integers $e_i$. This is equivalent to $\D$ being induced by the homogeneous ($n-k$)-form on $\C^{n+1}$
\begin{equation}\label{E:split}
\Omega=i_{X_1} \cdots i_{X_k}i_R (dx_0 \wedge \cdots \wedge dx_n),
\end{equation}
where for $i=1,\ldots, k$, $X_i$ is a homogeneous polynomial vector field on $\C^{n+1}$ of degree $d_i=1-e_i$. Note that $X_i$ defines a one-dimensional foliation $\G_i$ on $\Pj^n$ tangent to $\D$ with $\deg(\G_i)=d_i$, $i=1,\ldots,k$. In particular,
\begin{equation}\label{E:deg_aditive}
\deg(\D)=\deg(\G_1)+\cdots+\deg(\G_k).
\end{equation}
The reader can find more details in \cite{CukiermanPereira}, where it deals with foliations, but the same procedures apply to distributions.

\section{Distributions with locally free tangent sheaf}

In this section, we present some results related to distributions with 
locally free tangent sheaf. We believe that Proposition~\ref{P:known_fact_2} 
is known, but we are not aware of any reference for it, at least in the form it is stated. We will use the following result a few times in the sequel.

\begin{lemma}\label{R:existence_functions}
Let $M$ be a complex manifold of dimension $n \geq 3$, $X_1,\ldots,X_k \in \mathcal{X}(M)$, and let $\Theta \in \Omega^n(M)$. Set $\omega=i_{X_1} \cdots i_{X_k}\Theta$, and assume that $\codim(\sing(\omega)) \geq 2$. If $X \in \mathcal{X}(M)$ is such that $i_X \omega=0$, then there are unique holomorphic functions $f_1,\ldots,f_k$ on $M$ such that $X=f_1 X_1+ \cdots+f_k X_k$.
\end{lemma}
\begin{proof}
First, observe that the condition $\codim(\sing(\omega)) \geq 2$ guarantees that the holomorphic top-degree form $\Theta$ is nowhere vanishing on $M$.

Let $p \in M \setminus \sing(\omega)$. By the definition of $\omega$, it follows that $X_1(p), \ldots, X_k(p)$ are linearly independent in $T_p M$. Note that $\{X_1,\ldots,X_k\}$ can be extended to a local frame of $TM$ over a sufficiently small neighborhood $U$ of $p$. 

Indeed, let $\{v_{k+1},\ldots,v_n\} \subset T_p M$ be such that $\{X_1(p), \ldots, X_k(p), v_{k+1}, \ldots, v_n\}$ is a basis of $T_p M$. Choose an open neighborhood $U \subset M$ of $p$ and $n-k$ vector fields $X_{k+1},\ldots,X_n$ defined on $U$ such that $X_{k+1}(p)=v_{k+1},\ldots,X_n(p)=v_n$ (for instance, one may take $U=M$ and choose $X_{k+1},\ldots,X_n$ to be constant vector fields equal to $v_{k+1},\ldots,v_n$, respectively). By shrinking $U$ if necessary, it follows that $X_1(q),\ldots,X_n(q)$ are linearly independent in $T_q M$ for every $q \in U$, i.e., $\{X_1,\ldots,X_n\}$ is a frame of $TM$ over $U$.

As a consequence, we can write 
$$
X = f_1 X_1 + \cdots + f_n X_n,
$$
where $f_1,\ldots,f_n$ are unique holomorphic functions on $U$. The relation $i_X \omega = 0$ implies that $X(q)$ lies in the span of $X_1(q),\ldots,X_k(q)$ for every $q \in U$, hence $f_{k+1} = \cdots = f_n = 0$.

By varying $p$, we obtain an open covering $\{U_\alpha\}_{\alpha \in \Lambda}$ of $M \setminus \sing(\omega)$, and for each $\alpha \in \Lambda$ unique holomorphic functions $f^\alpha_1,\ldots,f^\alpha_k$ defined on $U_\alpha$ such that 
$$
X = f^\alpha_1 X_1 + \cdots + f^\alpha_k X_k
$$
on $U_\alpha$. Moreover, by uniqueness, if $U_\beta \cap U_\gamma \neq \emptyset$, then $f^\beta_j = f^\gamma_j$ on $U_\beta \cap U_\gamma$ for all $j=1,\ldots,k$.

Thus, for each $j=1,\ldots,k$, we can define a function $f_j$ on $M \setminus \sing(\omega)$ by setting $f_j(p)=f^\alpha_j(p)$ whenever $p \in U_\alpha$. Clearly, the functions $f_1,\ldots,f_k$ are holomorphic and satisfy 
$$
X = f_1 X_1 + \cdots + f_k X_k
$$
on $M \setminus \sing(\omega)$. Since $\codim(\sing(\omega)) \geq 2$, it follows from Hartogs' theorem that $f_1,\ldots,f_k$ extend holomorphically to $M$, which completes the proof.
\end{proof}

If $X \in \mathcal{X}(M)$ is a vector field defined on a complex manifold $M$, denote by $T_X$ the subsheaf of $T_M$ generated by $X$, i.e. for each open set $U \subset M$,
\[
T_X(U)=\{f \cdot X_{|U}: f \in \mathcal{O}_M(U)\}.
\]
Given a distribution $\D$ on $M$, if $X$ is a vector field tangent to $\D$, it is clear that $T_X$ is also a subsheaf of $\TD$.

\begin{remark}\label{R:vec_fol_if_codim_2}
For $X \neq 0 \in \mathcal{X}(M)$, denote by $\G$ the one-dimensional foliation on $M$ defined by $X$, see Remark~\ref{R:vec_field_def_fol}. In general, we have that $T_X$ is a subsheaf of $\TG$. A simple verification shows that $T_X = \TG$ if and only if $\codim(\sing(X)) \geq 2$. 
\end{remark}

\begin{lemma}\label{L:lemma_vec_implies_generates}
Let $\D$ be a $k$-dimensional distribution on a complex manifold $M$, $X_1,\ldots,X_k \in \mathcal{X}(M)$ and $p \in M$ be such that
\begin{equation}\label{E:lemma_vec_implies_generates}
(\TD)_p=(T_{X_1})_p\oplus \cdots \oplus (T_{X_k})_p.
\end{equation}
Then $\codim(\sing(X_i),p) \geq 2$, for $i=1,\ldots,k$. In particular, if \eqref{E:lemma_vec_implies_generates} holds for every $p \in M$, and $\G_1,\ldots,\G_k$ are the one-dimensional foliations on $M$ defined by $X_1,\ldots,X_k$, respectively, we have $\TD=T_{{\G_1}} \oplus \cdots \oplus T_{{\G_k}}$.
\end{lemma}
\begin{proof}
Assume, on the contrary, that there are $p \in M$ satisfying \eqref{E:lemma_vec_implies_generates} and $i \in \{1,\ldots,k\}$ such that $\codim(\sing(X_i),p) \leq 1$. Clearly, it follows from \eqref{E:lemma_vec_implies_generates} that $\codim(\sing(X_i),p)=1$. By rearranging the vector fields if necessary, we can assume that $i=1$. As an element of $\mathcal{X}(M,p)$, we can write $X_1=fY$, where $Y \in \mathcal{X}(M,p)$ and $f \in \mathcal{O}_{M,p}$ satisfies $f(p)=0$.

We claim that $Y \not\in (\TD)_p$. In fact, if $Y \in (\TD)_p$, by considering $X_1,\ldots,X_k$ as elements of $\mathcal{X}(M,p)$, from \eqref{E:lemma_vec_implies_generates} there would exist $f_1,\ldots, f_k \in \mathcal{O}_{M,p}$ such that
\[
Y=f_1X_1+\cdots+f_kX_k.
\]
Since $X_1=fY$ we have that
\[
(f_1 f-1)Y+f_2 X_2+\cdots +f_k X_k=0.
\]
Multiplying both sides of the last equation by $f$, we get
\[
(f_1 f-1)X_1+f f_2 X_2+ \cdots+f f_k X_k=0.
\]
Then from \eqref{E:lemma_vec_implies_generates} we conclude that 
$f_1 f - 1 \equiv 0$, which is a contradiction since $f(p)=0$. 

Thus we have $Y \not\in (\mathcal{T}_{\mathcal{D}})_p$ while $X_1=fY \in (\mathcal{T}_{\mathcal{D}})_p$, and once again we obtain a contradiction, since $(\mathcal{T}_M/\mathcal{T}_{\mathcal{D}})_p 
\cong (\mathcal{T}_M)_p / (\mathcal{T}_{\mathcal{D}})_p$ is a torsion-free $\mathcal{O}_{M,p}$-module (see Remark~\ref{R:tang_sheaf_properties}). The particular case follows from Remark~\ref{R:vec_fol_if_codim_2}.
\end{proof}

\begin{remark}\label{R:form_in_polydisc}
If $M = U \subset \C^n$ is a Stein open set, in particular if it is a 
polydisc, then, since $H^1(U,\mathcal{O}_U^*)=0$, a distribution 
$\mathcal{D}$ on $U$ can be defined by a single $(n-k)$-form $\omega$ on $U$ satisfying $\codim(\sing(\omega)) \geq 2$. When we deal with germs of distributions defined at $0 \in \C^n$, such germs can be represented by an $(n-k)$-form $\omega$, with $\codim(\sing(\omega)) \geq 2$, defined on a sufficiently small 
polydisc $U$ containing $0$. For a similar reason, a germ of a one-dimensional 
foliation can be represented by a single vector field $X$ defined on a sufficiently small 
polydisc containing $0$, satisfying $\codim(\sing(X)) \geq 2$.
\end{remark}

\begin{prop}\label{P:known_fact}
Let $\D$ be a $k$-dimensional distribution on $U$, where $U$ is $(\C^n,0)$ or a polydisc of $\C^n$, defined by the ($n-k$)-form $\omega$ on $U$, see Remark~\ref{R:form_in_polydisc}. Assume that $\G_1,\ldots,\G_k$ are one-dimensional foliations on $U$, defined by $X_1,\ldots,X_k \in \mathcal{X}(U)$, respectively, with $\codim(\sing(X_i)) \geq 2,i=1,\ldots,k$. Then the following are equivalent:
\begin{enumerate}[(a)]
\item\label{a:P:known_fact} $\TD=T_{{\G_1}} \oplus \cdots \oplus T_{{\G_k}}$.
\item\label{b:P:known_fact} $\omega=i_{X_1} \cdots i_{X_k}\Theta$, for some $n$-form $\Theta$ on $U$.
\end{enumerate}
\end{prop}
\begin{proof}
We will first address the case where $U$ is a polydisc in $\C^n$. 
Assume first that \eqref{b:P:known_fact} holds. Given $p \in U$, let $Z$ be 
a germ of a vector field around $p$ tangent to $\mathcal{D}$, defined on a 
sufficiently small open set $V \subset U$. Since $\codim(\sing(\omega)) \geq 2$ 
and, in particular,
\[
\omega = i_{X_1} \cdots i_{X_k} \Theta
\]
on $V$, the condition $i_Z \omega = 0$ implies, by Lemma~\ref{R:existence_functions}, 
that there exist unique holomorphic functions $f_1,\ldots,f_k$ defined on $V$ 
such that
\[
Z = f_1 X_1 + \cdots + f_k X_k.
\]
We conclude that
\[
(\mathcal{T}_{\mathcal{D}})_p 
= (T_{X_1})_p \oplus \cdots \oplus (T_{X_k})_p.
\]
As $p \in U$ is arbitrary, we obtain that
\[
\mathcal{T}_{\mathcal{D}} = T_{X_1} \oplus \cdots \oplus T_{X_k}
\]
and hence \eqref{a:P:known_fact} holds from the particular case of 
Lemma~\ref{L:lemma_vec_implies_generates}.

Now, assume that \eqref{a:P:known_fact} holds. From Lemma~\ref{L:lemma_vec_implies_generates}, it follows that for every $p \in U$ the 
germs of $X_1,\ldots,X_k$ at $p$ form a basis for $(\mathcal{T}_{\mathcal{D}})_p$. 
Given an $n$-form $\Lambda$ on $U$ with $\codim(\sing(\Lambda)) \geq 2$, we have 
that 
\begin{equation}\label{E:Prop_inicial}
i_{X_1} \cdots i_{X_k}\Lambda = h\omega,
\end{equation}
for some holomorphic function $h$ on $U$. We claim that $h$ has no zeros in $U$, 
and we show that \eqref{b:P:known_fact} holds by setting $\Theta = \Lambda/h$. 

In fact, if $h$ has some zero in $U$, since $\codim(\sing(\omega)) \geq 2$ we can 
take $q \in U$ such that $h(q)=0$ and $\omega(q) \neq 0$. It follows from \eqref{E:Prop_inicial} that
\[
X_1(q) \wedge \cdots \wedge X_k(q) = 0,
\]
which contradicts $\omega(q) \neq 0$, since the germs of $X_1,\ldots,X_k$ at $q$ 
form a basis for $(\mathcal{T}_{\mathcal{D}})_q$.

Finally, consider that $U=(\C^n,0)$. We may assume that $\G_1,\ldots,\G_k,X_1,\ldots,X_k,\omega ,\Theta$ are all defined on an a polydisc $V$ centered at $0$, and that $\D$ is represented by a $k$-dimensional distribution $\D_V$ on $V$. 

If \eqref{b:P:known_fact} holds and $X \in \mathcal{X}(\C^n,0)$ satisfies $i_X \omega=0$, then, after 
possibly shrinking $V$, we may assume that $X$ is also defined on $V$. In this situation, an argument 
analogous to the case where $U$ is itself a polydisc shows that \eqref{a:P:known_fact} also holds. 

Conversely, suppose that \eqref{a:P:known_fact} holds. Since 
$\TD = (T_{\mathcal{D}_V})_0$ is a free $\mathcal{O}_{\C^n,0}$-module with basis 
$\{X_1,\ldots,X_k\} \subset \mathcal{X}(\C^n,0)$, and $T_{\mathcal{D}_V}$ is a coherent sheaf, we may 
(after shrinking $V$ if necessary) assume that for every $p \in V$, the vector fields 
$X_1,\ldots,X_k$, viewed as elements of $\mathcal{X}(V,p)$, form a basis of $(T_{\mathcal{D}_V})_p$. 
In other words,
\[
T_{\mathcal{D}_V} \;=\; T_{X_1} \oplus \cdots \oplus T_{X_k} 
\;=\; T_{\G_1} \oplus \cdots \oplus T_{\G_k}.
\]
Using again the same argument as in the case where $U$ is a polydisc, we deduce that 
\eqref{b:P:known_fact} holds, which completes the proof.
\end{proof}

\begin{prop}\label{P:known_fact_2}
Let $\D$ be a $k$-dimensional distribution on a complex manifold $M$. Then the following are equivalent:
\begin{enumerate}[(a)]
\item\label{a:P:known_fact_2} $\TD$ is locally free of rank $k$.
\item\label{b:P:known_fact_2} For every $p \in M$, there exist an open set $U$ containing $p$ and $X_1,\ldots,X_k \in \mathcal{X}(U)$ such that $\TD_{|U}=T_{X_1} \oplus \cdots \oplus T_{X_k}$.
\item\label{c:P:known_fact_2} For every $p \in M$, there exist an open set $U$ containing $p$ and one-dimensional foliations $\G_1,\ldots,\G_k$ on $U$ such that $\TD_{|U}=T_{{\G_1}} \oplus \cdots \oplus T_{{\G_k}}$.
\item\label{d:P:known_fact_2} For every $p \in M$, there exist $X_1,\ldots,X_k \in \mathcal{X}(M,p)$ and $\Theta \in \Omega^n(M,p)$ such that $\omega=i_{X_1} \cdots i_{X_k}\Theta$ satisfies $\codim(\sing(\omega)) \geq 2$ and it defines the germ of $\D$ at $p$.
\item\label{e:P:known_fact_2} For every $p \in M$, the $\mathcal{O}_{M,p}$-module $(\TD)_p$ is free with $k$ generators.
\end{enumerate}
\end{prop}
\begin{proof}
We start by showing that \eqref{a:P:known_fact_2} implies 
\eqref{b:P:known_fact_2}. If \eqref{a:P:known_fact_2} holds, by definition, for 
each $p \in M$ there exist an open set $U$ containing $p$ and an isomorphism 
of sheaves
\[
\varphi: 
\underbrace{{\mathcal{O}_{M}}_{|U} \oplus \cdots \oplus {\mathcal{O}_{M}}_{|U}}_{k \text{ copies}} 
\longrightarrow \mathcal{T}_{\mathcal{D}}{}_{|U}.
\]
Thus, we can obtain the vector fields $X_1,\ldots,X_k$ defined on $U$ by 
considering the images of the sections 
$(1,0,\ldots,0), \ldots, (0,\ldots,0,1)$ under $\varphi$. 

The proof that \eqref{b:P:known_fact_2} implies \eqref{c:P:known_fact_2} follows 
from the particular case of Lemma~\ref{L:lemma_vec_implies_generates}. If 
\eqref{c:P:known_fact_2} holds, then \eqref{d:P:known_fact_2} also holds as an 
immediate consequence of Proposition~\ref{P:known_fact}. 

If \eqref{d:P:known_fact_2} holds, note that in particular 
$\codim(\sing(X_i)) \geq 2$ for $i=1,\ldots,k$. Let 
$\mathcal{G}_1,\ldots,\mathcal{G}_k$ be the germs of the one-dimensional 
foliations defined by $X_1,\ldots,X_k$, respectively. By Proposition~\ref{P:known_fact}, we conclude that
\[
(\mathcal{T}_{\mathcal{D}})_p = 
(T_{\mathcal{G}_1})_p \oplus \cdots \oplus (T_{\mathcal{G}_k})_p,
\]
and then by Remark~\ref{R:vec_fol_if_codim_2} it follows that 
$(\mathcal{T}_{\mathcal{D}})_p$ is free with generators $X_1,\ldots,X_k$, which 
implies \eqref{e:P:known_fact_2}. 

Finally, if \eqref{e:P:known_fact_2} holds, then \eqref{a:P:known_fact_2} also 
holds, since $\mathcal{T}_{\mathcal{D}}$ is coherent.
\end{proof}

The following result provides a characterization of LDS forms (cf. \cite[Proposition~1.2.1]{MR1842027}).

\begin{prop}\label{P:Converse}
Let $\omega \in \Omega^{n-k}(M)$ be defined on a complex manifold $M$, with $n \geq 3$ and $1 \leq k < n$. Then $\omega$ is LDS if and only if for every $p \in M \setminus \sing(\omega)$ there exist an open neighborhood $U \subset M$ of $p$, vector fields $X_1,\ldots,X_{k} \in \mathcal{X}(U)$, and $\Theta \in \Omega^n(U)$ such that
\begin{equation}\label{E:prop_loc_loc_decom}
\omega=i_{X_1} \cdots i_{X_{k}} \Theta
\end{equation}
holds on $U$.
\end{prop}
\begin{proof}
Assume that $\omega$ is LDS, and let $\D$ be the regular distribution on $M \setminus \sing(\omega)$ defined by $\omega$. By \cite[Section 2]{GiraldoCollantes}, the tangent sheaf $\TD$ is locally free. We conclude from Proposition~\ref{P:known_fact_2} that $\omega$ can be written as in \eqref{E:prop_loc_loc_decom}.

Next, given $p \in M \setminus \sing(\omega)$, assume that an equality as in \eqref{E:prop_loc_loc_decom} holds on an open set $U \subset M \setminus \sing(\omega)$ containing $p$. Then the map 
\[
x \in U \longmapsto \ker(\omega(x))
\]
clearly defines a distribution of $k$-planes on $U$. Hence, the equivalence between items (i) and (v) of \cite[Proposition~1.2.1]{MR1842027} ensures that $\omega$ is LDS. However, since \cite{MR1842027} does not provide a proof of this proposition, we include a direct argument here.

Proceeding as in the proof of Lemma~\ref{R:existence_functions}, and shrinking $U$ if necessary, there exist $n-k$ holomorphic vector fields $X_{k+1}, \ldots, X_n$ on $U$ such that $\{X_1(q), \ldots, X_n(q)\}$ is a basis of $T_q M$ for every $q \in U$. Then there are (unique) holomorphic $1$-forms $\omega_1, \ldots,\omega_n$ defined on $U$ such that, $\forall q \in U$,
\begin{equation}\label{E:Converse}
i_{X_j}\omega_l(q)=
\begin{cases}
1, & \text{if } j=l,\\
0, & \text{if } j \neq l,
\end{cases}
\qquad j,l=1,\ldots,n.
\end{equation}
Since $\Theta_{|U}=f \omega_1 \wedge \cdots \wedge \omega_n$, for some nonzero holomorphic function $f$ on $U$, we obtain
\[
\omega_{|U}=\pm f\omega_{k+1} \wedge \cdots \wedge \omega_n,
\]
which shows that $\omega$ is LDS.
\end{proof}

In what follows, we present an alternative argument for the first part of the proof of Proposition~\ref{P:Converse}, following the same approach as in the second part. Assume that
\[
\omega_{|U}=\omega_1 \wedge \cdots \wedge \omega_{n-k},
\]
where $U \subset M$ is a neighborhood of a point $p \in M \setminus \sing(\omega)$, and $\omega_1,\ldots,\omega_{n-k}$ are holomorphic $1$-forms on $U$.

By shrinking $U$ if necessary, we may assume that there exist $k$ holomorphic $1$-forms $\omega_{n-k+1},\ldots,\omega_n$ defined on $U$ such that 
\[
\omega_1 \wedge \cdots \wedge \omega_n (q) \neq 0 \quad \text{for every } q \in U.
\]
There are then unique holomorphic vector fields $X_1,\ldots,X_n$ on $U$ satisfying \eqref{E:Converse}. Consequently,
\[
i_{X_{n-k+1}} \cdots i_{X_n} \Theta = \pm \omega_{|U},
\]
where $ \Theta := \omega_1 \wedge \cdots \wedge \omega_n \in \Omega^n(U) $.

An immediate consequence of Propositions~\ref{P:known_fact_2} and \ref{P:Converse} is the following:

\begin{cor}\label{C:Converse}
Let $M$ be a complex manifold of dimension $n \geq 3$. Assume that $X_1,\ldots,X_k \in \mathcal{X}(M) $ and $\Theta \in\Omega^n(M)$. Then
\[
\omega=i_{X_1} \cdots i_{X_k} \Theta
\]
is LDS. Consequently, $\omega$ defines a regular $k$-dimensional distribution on $M \setminus \sing(\omega)$. In particular, if $\codim(\sing(\omega)) \ge 2$, then $\omega$ defines a singular $k$-dimensional distribution on $M$ whose tangent sheaf is locally free.
\end{cor}

In the next two examples, we illustrate some well-known facts using the results developed above.

\begin{example}
An $(n-1)$-form $\omega$ on a complex manifold of dimension $n$ is LDS. Indeed, by Remark~\ref{R:iso_X_O}, for every $p \in M$, there exists an open neighborhood $U \subset M$ of $p$, a holomorphic vector field $X$, and a nonvanishing $n$-form $\Theta$, all defined on $U$, such that
\[
\omega = i_X \Theta
\]
holds on $U$.

Corollary~\ref{C:Converse} then shows that $\omega$ is LDS. In particular, if one chooses to describe a one-dimensional distribution by a collection of ($n-1$)-forms $\{\omega_j\}_{j \in J}$, then, in the definition of a singular holomorphic distribution, the condition that each $\omega_j$ be LDS is automatically satisfied.
\end{example}

\begin{example}\label{Ex:dist_1_inte}
We now verify that a one-dimensional distribution $\mathcal{D}$ on a complex manifold $M$ is always integrable, and that its tangent sheaf $\mathcal{T}_{\mathcal{D}}$ is locally free.

To this end, we use the original definition of a distribution given by a family $\{(U_j,\omega_j)\}_{j \in J}$, where $\mathcal{U}=\{U_j\}_{j \in J}$ is an open covering of $M$ and $\omega_j \in \Omega^{n-1}(U_j)$ satisfies $\codim(\sing(\omega_j)) \geq 2$. By the previous example, we know that each $\omega_j$ is LDS. Furthermore, in the definition of integrability, we have ($n+1$)-forms $d\eta_i \wedge \omega_j$ on the left side of \eqref{E:integrability}, so it is obviously true.

Finally, assume that
\[
\omega_j = i_{v_j}\Theta_j
\]
as in \eqref{E:one_form_vector}, where $v_j \in \mathcal{X}(U_j)$ and $\Theta_j \in \Omega^n(U_j)$. This ensures that condition \eqref{d:P:known_fact_2} in Proposition~\ref{P:known_fact_2} is satisfied, and therefore $\mathcal{T}_{\mathcal{D}}$ is locally free. This also follows from the well-known fact that every rank one reflexive sheaf on a complex manifold is locally free (see Remark~\ref{R:tang_sheaf_properties}).
\end{example}

\section{Some division properties}

Classically, given germs of holomorphic forms $\omega \in \Omega^1(\mathbb{C}^n,0)$ and $\alpha \in \Omega^k(\mathbb{C}^n,0)$, we say that $\alpha$ is divisible by $\omega$ if there exists $\beta \in \Omega^{k-1}(\mathbb{C}^n,0)$ such that $\alpha=\omega \wedge \beta$. If $\alpha$ is divisible by $\omega$, then $\omega \wedge \alpha=0$. One might ask when this condition is also sufficient, which leads to the following:

\textbf{division property by $k$-forms.} We say that $\omega \in \Omega^1(\mathbb{C}^n,0)$ satisfies the division property by $k$-forms, if for any $\alpha \in \Omega^k(\mathbb{C}^n,0)$ such that $\omega \wedge \alpha=0$ then there exists $\beta \in \Omega^{k-1}(\mathbb{C}^n,0)$ such that $\alpha=\omega \wedge \beta$ (cf. \cite{Moussu76} and \cite{MR2473858}).

The following result is known as the De Rham-Saito Division Lemma for $1$-forms. A more general version can be found in \cite{Saito76} and \cite{Malgrange1977}.

\begin{lemma}\label{L:De-S}
Let $\omega \in \Omega^1(\mathbb{C}^n,0)$ be a germ of a $1$-form such that $\codim( \sing (\omega)) \geq p+1$, where $p$ is a nonnegative integer. Then $\omega$ satisfies the division property by $k$-forms for $k \leq p$.
\end{lemma}

In a similar way, we say that 
$X \in \mathcal{X}(\mathbb{C}^n,0)$ satisfies the \emph{division property by $k$-forms} 
if, for any $\alpha \in \Omega^k(\mathbb{C}^n,0)$ such that $i_X \alpha = 0$ 
there exists $\beta \in \Omega^{k+1}(\mathbb{C}^n,0)$ such that 
$\alpha = i_X \beta$. 

Lemma~\ref{L:De-S} has a dual version in terms of vector fields. In fact, 
writing $X = \sum_{i=1}^n a_i \partial / \partial z_i$, set 
\[
\omega = \sum_{i=1}^n a_i  dz_i \in \Omega^1(\mathbb{C}^n,0).
\]
For $1 \leq k \leq n-1$, by making use of the Hodge star operator 
\[
* : \Omega^k(\mathbb{C}^n,0) \longrightarrow \Omega^{n-k}(\mathbb{C}^n,0),
\]
it can be proved that $\omega$ satisfies the division property by $k$-forms if and 
only if $X$ satisfies the division property by $(n-k)$-forms, see 
\cite{CamachoLins82} and \cite{MR2370063}. 

Then we have the following:

\begin{cor}\label{C:dual_division}
Let $X \in \mathcal{X}(\mathbb{C}^n,0)$ be a germ of a vector field such that $\codim (\sing (X)) \geq p+1$, where $p$ is a positive integer. Then $X$ satisfies the division property by $k$-forms for $n-p \leq k \leq n-1$. In particular, if $\codim ( \sing (X)) \geq 3$ then $X$ satisfies the division property by $k$-forms for $k=n-2$ and $k=n-1$.
\end{cor}

The next definition is simply a complement to the property of vector field division by $k$-forms.

\begin{dfn}\label{D:divides}
Let $\omega \in \Omega^k(\mathbb{C}^n,0) $ and $X \in \mathcal{X}(\mathbb{C}^n,0)$ be the germs of a $k$-form and a vector field, respectively. We say that $X$ divides $\omega$ if there is $\Theta \in \Omega^{k+1}(\mathbb{C}^n,0)$ such that $\omega=i_X  \Theta$.
\end{dfn}

In what follows, we extend the notion of division of differential forms by a vector field to the case of a distribution and a one-dimensional foliation tangent to it.

\begin{dfn}\label{D:def_division}
Let $\D$ be a $k$-dimensional distribution on a complex manifold $M$ of dimension $n \geq 3$, and $\G$ a one-dimensional foliation on $M$ tangent to $\D$. We say that $\G$ locally divides $\D$ at $p \in M$ if, for any $X \in \mathcal{X}(M,p)$ and any $\omega \in \Omega^{n-k}(M,p)$, with $\codim(\sing(X)) \geq 2$ and $\codim(\sing(\omega)) \geq 2$, defining the germs of $\G$ and $\D$ at $p$, respectively, we have
\[
\omega=i_X \eta,
\]
for some $\eta \in \Omega^{n-k+1}(M,p)$.
If $\G$ locally divides $\D$ at every $p \in M$, we say that $\G$ locally divides $\D$.
\end{dfn}

Since the elements of $\mathcal{X}(M,p)$ and $\Omega^{n-k}(M,p)$ that define the germs of $\mathcal{G}$ and $\mathcal{D}$ at $p$ and have no zeros of codimension one are determined up to multiplication by units in $\mathcal{O}_{M,p}$, it is easy to see that $\G$ locally divides $\D$ at $p$ if and only if there exist $X \in \mathcal{X}(M,p)$ and $\omega \in \Omega^{n-k}(M,p)$, with $\codim(\sing(X)) \geq 2$ and $\codim(\sing(\omega)) \geq 2$, defining the germs of $\G$ and $\D$ at $p$, respectively, such that
\[
\omega=i_X \eta,
\]
for some $\eta \in \Omega^{n-k+1}(M,p)$.

\begin{dfn}\label{D:def_division_vec}
Let $\D$ be a $k$-dimensional distribution on a complex manifold $M$ of dimension $n \geq 3$, $v$ a holomorphic vector field on $M$ (possibly with codimension one zeros) tangent to $\D$, and $\G$ the one-dimensional foliation on $M$ defined by $v$, see Remark~\ref{R:vec_field_def_fol}. We say that $v$ locally divides $\D$ at $p \in M$ if $\codim(\sing(v),p) \geq 2$ and $\G$ locally divides $\D$ at $p$. If $v$ locally divides $\D$ at every $p \in M$ we say that $v$ locally divides $\D$.
\end{dfn}

In view of Definition~\ref{D:def_division}, it is easy to see that $v$ locally divides $\D$ at $p \in M$ if and only for any $\omega \in \Omega^{n-k}(M,p)$ defining $\D$ at $p$ we have that $\omega=i_v\eta$, for some $\eta \in \Omega^{n-k+1}(M,p)$.

The following two lemmas will be used in the proof of Proposition~\ref{P:eq_subbundle}.

\begin{lemma}\label{R:dim_2_loc_free}
Let $\D$ be a two-dimensional distribution on a complex manifold $M$ of dimension $n \geq 3$, and $\G$ a one-dimensional foliation on $M$ tangent to $\D$. Given $p \in M$, the following assertions are equivalent:
\begin{enumerate}
\item\label{1:R:dim_2_loc_free} $\G$ locally divides $\D$ at $p$.
\item\label{2:R:dim_2_loc_free} There is a one-dimensional foliation $\h$ defined on an open neighborhood of $p$ and tangent to $\D$ such that $(\TD)_p=(\TG)_p \oplus (\Th)_p$.
\item\label{3:R:dim_2_loc_free} The module $(\TD/\TG)_p$ is $\mathcal{O}_{M,p}$-free.
\end{enumerate}
In particular, if $\G$ locally divides $\D$ then $\TD$ is locally free.
\end{lemma}
\begin{proof}
Suppose that $X \in \mathcal{X}(M,p)$ and $\omega \in \Omega^{n-2}(M,p)$ are germs of a vector field and an $(n-2)$-form, respectively, with $\codim(\sing(X)) \geq 2$ and $\codim(\sing(\omega)) \geq 2$, defining the germs of $\mathcal{G}$ and $\mathcal{D}$ at $p$.

If \eqref{1:R:dim_2_loc_free} holds, by definition there exists 
$\eta \in \Omega^{n-1}(M,p)$ such that $\omega = i_X \eta$. By Remark~\ref{R:iso_X_O}, we can write $\eta = i_Y \Theta$, for some 
$Y \in \mathcal{X}(M,p)$ and $\Theta \in \Omega^{n}(M,p)$, 
hence $\omega = i_X i_Y \Theta$. 

Since $\codim(\sing(\omega)) \geq 2$, we have that $\codim(\sing(Y)) \geq 2$. It follows that $Y$ induces a one-dimensional foliation $\mathcal{H}$ 
defined on an open neighborhood of $p$ and tangent to $\mathcal{D}$. 
By Proposition~\ref{P:known_fact}, from $\omega = i_X i_Y \Theta$ 
we conclude that 
\[
(\mathcal{T}_{\mathcal{D}})_p = (\mathcal{T}_{\mathcal{G}})_p \oplus (\mathcal{T}_{\mathcal{H}})_p.
\]
Then \eqref{2:R:dim_2_loc_free} follows.

Assume that \eqref{2:R:dim_2_loc_free} holds and that the germ of $\mathcal{H}$ 
at $p$ is defined by $Y \in \mathcal{X}(M,p)$, with $\codim(\sing(Y)) \geq 2$. By Proposition~\ref{P:known_fact}, from 
$(\mathcal{T}_{\mathcal{D}})_p=(\mathcal{T}_{\mathcal{G}})_p \oplus (\mathcal{T}_{\mathcal{H}})_p$ 
it follows that $\omega=i_X i_Y \Theta$, for some 
$\Theta \in \Omega^{n}(M,p)$. 
We conclude that $\mathcal{G}$ locally divides $\mathcal{D}$ at $p$, 
then \eqref{1:R:dim_2_loc_free} follows. 
We also have that
\[
(\mathcal{T}_{\mathcal{D}}/\mathcal{T}_{\mathcal{G}})_p 
\cong (\mathcal{T}_{\mathcal{D}})_p / (\mathcal{T}_{\mathcal{G}})_p 
\cong (\mathcal{T}_{\mathcal{H}})_p,
\]
then it follows that $(\mathcal{T}_{\mathcal{D}}/\mathcal{T}_{\mathcal{G}})_p$ 
is $\mathcal{O}_{M,p}$-free, since $(\mathcal{T}_{\mathcal{H}})_p$ 
is a rank one free $\mathcal{O}_{M,p}$-module generated by $Y \in (\mathcal{T}_{\mathcal{H}})_p$, 
obtaining \eqref{3:R:dim_2_loc_free}.

Finally, if \eqref{3:R:dim_2_loc_free} holds then the exact sequence
\[
0 \to (\mathcal{T}_{\mathcal{G}})_p \to (\mathcal{T}_{\mathcal{D}})_p 
\to (\mathcal{T}_{\mathcal{D}}/\mathcal{T}_{\mathcal{G}})_p \to 0
\]
splits, since $(\mathcal{T}_{\mathcal{D}}/\mathcal{T}_{\mathcal{G}})_p$ 
is $\mathcal{O}_{M,p}$-free. As $(\mathcal{T}_{\mathcal{D}}/\mathcal{T}_{\mathcal{G}})_p$ has rank one, 
we can identify it with $(T_Y)_p \subset (\mathcal{T}_{\mathcal{D}})_p$, 
where $Y$ is a holomorphic vector field tangent to $\mathcal{D}$, 
defined on an open set $U$ containing $p$. 
By shrinking $U$ if necessary, we can assume that $X$ is also represented on $U$. 
We have that
\[
(\mathcal{T}_{\mathcal{D}})_p = (\mathcal{T}_{\mathcal{G}})_p \oplus (T_Y)_p 
= (T_X)_p \oplus (T_Y)_p,
\]
and from Lemma~\ref{L:lemma_vec_implies_generates} we obtain that 
$\codim(\sing(Y,p)) \geq 2$. 
Hence $Y$ induces a one-dimensional foliation $\mathcal{H}$ tangent to $\mathcal{D}$ 
on a sufficiently small neighborhood of $p$, in such a way that 
$(\mathcal{T}_{\mathcal{D}})_p=(\mathcal{T}_{\mathcal{G}})_p \oplus (\mathcal{T}_{\mathcal{H}})_p$. 
Then \eqref{2:R:dim_2_loc_free} follows, which finishes the proof.

In particular, if $\mathcal{G}$ locally divides $\mathcal{D}$, 
using that \eqref{2:R:dim_2_loc_free} holds for every $p \in M$, we obtain that the $\mathcal{O}_{M,p}$-module $(\TD)_p$ is free of rank $2$. By Proposition~\ref{P:known_fact_2}, we conclude that $\mathcal{T}_{\mathcal{D}}$ is locally free.
\end{proof}

\begin{cor}\label{C:Locally_divides_general}
Let $\D$ be a two-dimensional distribution on a complex manifold $M$ of dimension $n \geq 3$. Assume that there exists a one-dimensional foliation $\G$ (resp. a holomorphic vector field $v$) on $M$ such that $\codim(\sing(\G)) \geq 3$ (resp. $\codim(\sing(v)) \geq 3$). Then $\G$ (resp. $v$) locally divides $\D$. In particular, $\TD$ is locally free.
\end{cor}
\begin{proof}
The distribution $\D$ is defined by a ($n-2$)-form on an open neighborhood of every point of $M$. By the particular case of Corollary~\ref{C:dual_division}, we have that $\G$ (resp. $v$) locally divides $\D$ at every point of $M$. Then $\G$ (resp. both $v$ and the one-dimensional foliation defined by $v$) locally divides $\D$. Lemma~\ref{R:dim_2_loc_free} implies that $\TD$ is locally free.
\end{proof}

Let $M$ be a complex manifold. It is well known that the category of vector bundles over $M$ and the category of locally free sheaves on $M$ are equivalent. This equivalence is given by associating to a vector bundle $E$ the sheaf $\mathcal{E}$ of holomorphic sections of $E$.

For the next two results, we will make no distinction between a locally free sheaf on a complex manifold $M$ and the associated vector bundle.

\begin{prop}\label{P:eq_subbundle}
Let $\D$ be a two-dimensional distribution on a complex manifold $M$ of dimension $n \geq 3$, and $\G$ a one-dimensional foliation on $M$ tangent to $\D$. Then $\G$ locally divides $\D$ if and only if $\TG$ is a subbundle of $\TD$.
\end{prop}
\begin{proof}
Assume that $\mathcal{G}$ locally divides $\mathcal{D}$. 
By Lemma~\ref{R:dim_2_loc_free}, it follows that $\mathcal{T}_{\mathcal{D}}$ 
is locally free and, for every $p \in M$, we have that 
$(\mathcal{T}_{\mathcal{D}}/\mathcal{T}_{\mathcal{G}})_p$ 
is an $\mathcal{O}_{M,p}$-free module. 
Since $\mathcal{T}_{\mathcal{D}}/\mathcal{T}_{\mathcal{G}}$ 
is a coherent sheaf, we obtain that it is locally free. 
We conclude that 
$\mathcal{T}_{\mathcal{G}}$ is a subbundle of $\mathcal{T}_{\mathcal{D}}$.

On the other hand, if $\mathcal{T}_{\mathcal{G}}$ is a subbundle of 
$\mathcal{T}_{\mathcal{D}}$, then we get that $\mathcal{T}_{\mathcal{D}}/\mathcal{T}_{\mathcal{G}}$ 
is locally free. 
Using Lemma~\ref{R:dim_2_loc_free} once again, we conclude that 
$\mathcal{G}$ locally divides $\mathcal{D}$ at every $p \in M$, 
i.e., $\mathcal{G}$ locally divides $\mathcal{D}$, 
which concludes the proof.
\end{proof}

It is well known that a rank two vector bundle over $\Pj^n$, $n \geq 2$, contains no proper subbundles, unless it splits. See, for example, \cite[Lemma~2.2.4]{OkSchSpi}. The analog for two-dimensional distributions and one-dimensional subfoliations would be the following result, also inspired by \cite[Lemma~4.15]{RaRuJo2022}.

\begin{cor}\label{C:divides}
Let $\D$ be a two-dimensional distribution on $\mathbb P^n$, $n \geq 3$. Assume that there exists a one-dimensional foliation $\G$ on $\mathbb P^n$ which locally divides $\D$. Then there is a one-dimensional foliation $\h$ on $\mathbb P^n$ tangent to $\D$ such that the tangent sheaf of $\D$ splits as $\TD=\TG \oplus \Th$.
\end{cor}
\begin{proof}
It follows from Proposition~\ref{P:eq_subbundle} that 
$\mathcal{T}_{\mathcal{G}}$ is a subbundle of the rank $2$ bundle 
$\mathcal{T}_{\mathcal{D}}$. 
From \cite[Lemma~2.2.4]{OkSchSpi}, we conclude that there is a splitting of vector bundles 
\[
\mathcal{T}_{\mathcal{D}}=\mathcal{T}_{\mathcal{G}} \oplus H,
\] 
for some rank one vector bundle $H$ over $\mathbb{P}^n$. 
If we denote the rank one locally free sheaf associated to $H$ by 
$\mathcal{O}_{\mathbb{P}^n}(h)$, with $h \in \mathbb{Z}$, 
we have the splitting 
\[
\mathcal{T}_{\mathcal{D}}=\mathcal{T}_{\mathcal{G}} \oplus \mathcal{O}_{\mathbb{P}^n}(h)
\]
as locally free sheaves. 
The factor $\mathcal{O}_{\mathbb{P}^n}(h)$ corresponds to the tangent sheaf of a 
one-dimensional foliation $\mathcal{H}$ on $\mathbb{P}^n$ tangent to $\mathcal{D}$ 
satisfying $\deg(\mathcal{H}) = 1 - h$. 
In particular,
\[
\mathcal{T}_{\mathcal{D}}=\mathcal{T}_{\mathcal{G}} \oplus \mathcal{T}_{\mathcal{H}}. \qedhere
\]
\end{proof}

\begin{remark}\label{R:decomponability_G_D}
Suppose that $\G$ and $\D$ satisfy the hypotheses of Corollary~\ref{C:divides}. It follows from \eqref{E:deg_aditive} that if $d=\deg(\D)$ and $e=\deg(\G)$, then
\[
\TD\cong \mathcal O_{\mathbb P^n}(1-e) \oplus \mathcal O_{\mathbb P^n}(1+e-d).
\]
\end{remark}

\begin{cor}\label{C:cod_3_divides}
Let $\D$ be a two-dimensional distribution on $\mathbb P^n$, $n \ge 3$, and $\G$ a one-dimensional foliation (resp. $v$ a holomorphic vector field) on $\Pj^n$ tangent to $\D$. Assume that $\codim( \sing (\G) )\geq 3$ (resp. $\codim( \sing (v) )\geq 3$). Then $\TD$ splits.
\end{cor}
\begin{proof}
It follows from Corollary~\ref{C:Locally_divides_general} that $\G$ (resp. both $v$ and the one-dimensional foliation defined by $v$) locally divides $\D$. The result follows from Corollary~\ref{C:divides}.
\end{proof}

\begin{cor}
Let $\D$ be a two-dimensional distribution on $\mathbb P^n$, $n \ge 3$, and $\G$ a one-dimensional foliation on $\Pj^n$ tangent to $\D$. If $\G$ locally divides $\D$, then $\deg(\G) \leq \deg(\D)$. In particular, if $ \codim \sing(\G) \geq 3$ then $\deg(\G) \leq \deg(\D)$.
\end{cor}
\begin{proof}
Since $\mathcal{G}$ locally divides $\mathcal{D}$, by Corollary~\ref{C:divides} 
there exists a one-dimensional foliation $\mathcal{H}$ on $\mathbb{P}^n$ 
tangent to $\mathcal{D}$ such that 
\[
\mathcal{T}_{\mathcal{D}}=\mathcal{T}_{\mathcal{G}} \oplus \mathcal{T}_{\mathcal{H}}.
\]
We obtain from \eqref{E:deg_aditive} that 
\[
\deg(\mathcal{D})=\deg(\mathcal{G})+\deg(\mathcal{H}),
\]
which implies $\deg(\mathcal{G}) \leq \deg(\mathcal{D})$. 
If $\codim \sing(\mathcal{G}) \geq 3$, by Corollary~\ref{C:Locally_divides_general} 
we have that $\mathcal{G}$ locally divides $\mathcal{D}$ and the result follows.
\end{proof}

\section{Proofs of Theorems A and B}

The next result will be used in the proof of Theorem~\ref{T:A}.

\begin{lemma}\label{L:divides_f_and_g}
For $n \geq 3$ and $1 \leq k < n$, let $\omega \in \Omega^{n-k}(\mathbb{C}^n,0)$, $X_1,\ldots,X_k \in \mathcal{X}(\mathbb{C}^n,0)$, and $\Theta \in \Omega^n(\mathbb{C}^n,0)$ be such that $\omega=i_{X_1} \cdots i_{X_k}\Theta$ satisfies $\codim(\sing(\omega)) \geq 2$. Let $X \in \mathcal{X}(\mathbb{C}^n,0)$ be given by $
X=f_1X_1+\ldots+f_k X_k$, with $f_1,\ldots,f_k \in \mathcal{O}_{\mathbb{C}^n,0}$. Then there are $Y_1,\ldots,Y_{k-1}\in \mathcal{X}(\mathbb{C}^n,0)$ and $\Lambda \in \Omega^n(\mathbb{C}^n,0)$ such that 
\begin{equation}\label{E:L:divides_f_and_g}
\omega=i_X i_{Y_1} \cdots i_{Y_{k-1}}\Lambda
\end{equation}
if and only if $f_i(0) \neq 0$ for some $i \in \{1,\ldots ,k\}$. Moreover, if $k>1$, we can choose $\Lambda=\Theta$ in \eqref{E:L:divides_f_and_g}. In particular, for $k=2$, we have that $X=f_1X_1+f_2 X_2$ divides $\omega$ if and only if $f_1(0) \neq 0$ or $f_2(0) \neq 0$.
\end{lemma}
\begin{proof}
Suppose that $f_i(0) \neq 0$ for some $i \in \{1,\ldots ,k\}$. As
\[
i_{X_1} \cdots i_{X_{i-1}} i_X i_{X_{i+1}} \cdots i_{X_k}\Theta=f_i \omega,
\]
we can write
\[
\omega=
i_X i_{X_1}\cdots i_{X_{i-1}} i_{X_{i+1}}\cdots i_{X_k}
\biggl(
\frac{(-1)^{i-1}}{f_i}\Theta
\biggr).
\]
and for $\Lambda= \frac{(-1)^{i-1}}{f_i}\Theta$, clearly we can pick $Y_1,\ldots,Y_{k-1}\in \mathcal{X}(\mathbb{C}^n,0)$ such that \eqref{E:L:divides_f_and_g} holds. If $k>1$, by replacing $Y_1$ by $\frac{(-1)^{i-1}}{f_i} Y_1$, we see that we can choose $\Lambda=\Theta$ in \eqref{E:L:divides_f_and_g}.

Now, assume by contradiction that $f_1(0)=\cdots=f_k(0)=0$ and there are $Y_1,\ldots,Y_{k-1}\in \mathcal{X}(\mathbb{C}^n,0)$ and $\Lambda \in \Omega^n(\mathbb{C}^n,0)$ such that \eqref{E:L:divides_f_and_g} holds. Define $Z_1=X,Z_2=Y_1,\ldots,Z_k=Y_{k-1}$. For each $j \in \{1,\ldots,k\}$, since $i_{Z_j}\omega=0$, by Lemma~\ref{R:existence_functions} there are unique $g_{1j},\ldots,g_{kj} \in \mathcal{O}_{\mathbb{C}^n,0}$ such that
\[
Z_j=g_{1j}X_1+\cdots+g_{kj}X_k.
\]
Note that $g_{11}=f_1,\ldots,g_{k1}=f_k$. Since $\codim(\sing(\omega)) \geq 2$, both $\Theta$ and $\Lambda$ are nonvanishing. Therefore, there is $h \in \mathcal{O}^*_{\C^n,0}$ such that $\Lambda=h \cdot \Theta$.

Since 
\[
\omega=i_{X_1} \cdots i_{X_k}\Theta=i_{Z_1} \cdots i_{Z_{k}}\Lambda,
\]
a straightforward computation shows that $\omega=h\cdot \det(M(z)) \cdot \omega$, where $M(z)$ is the $k \times k$ matrix 
\[
M(z)=[g_{ij}(z)]_{i=1,\ldots,k,j=1,\ldots,k}.
\]   

Thus $h \cdot \det(M(z))\equiv 1$. On the other hand, since $g_{11}(0)=\cdots=g_{k1}(0)=0$, we see that $\det(M(0))=0$. We obtain a contradiction and the result follows.

The particular case of the lemma follows immediately from the observation that for every $\eta \in \Omega^{n-1}(\mathbb{C}^n,0)$ there is $Y \in \mathcal{X}(\mathbb{C}^n,0)$ such that $\eta=i_Y\Theta$, see Remark~\ref{R:iso_X_O}.
\end{proof}

\begin{proof}[Proof of Theorem~\ref{T:A}]
Let $X_1,\ldots,X_k \in \mathcal{X}(\mathbb{C}^n,0)$ be such that
\begin{equation}\label{E:proof_the_A_4}
\omega=i_{X_1} \cdots i_{X_k} (dz_1 \wedge \cdots \wedge dz_n)
\end{equation}
satisfies $\codim(\sing(\omega)) \geq 2$ and it defines $\D$, see Proposition~\ref{P:known_fact_2}. Assume that $\G$ is defined by $X \in \mathcal{X}(\C^n,0)$, with $\codim(\sing(X)) \geq 2$. Since $i_X \omega=0$, by Lemma~\ref{R:existence_functions} there exist $f_1,\ldots,f_k \in \mathcal{O}_{\mathbb{C}^n,0}$ such that 
\begin{equation}\label{E:proof_the_A}
X=f_1 X_1 +\cdots+f_k X_k.
\end{equation}
There are two mutually exclusive cases:
\begin{itemize}
\item There is $i \in \{1,\ldots,k\}$ such that $f_i(0) \neq 0$.

In this case, by Lemma~\ref{L:divides_f_and_g} there are $Y_1,\ldots,Y_{k-1}\in \mathcal{X}(\mathbb{C}^n,0)$ and $\Theta \in \Omega^n(\C^n,0)$ such that
\begin{equation}\label{E:proof_the_A_2}
\omega=i_X i_{Y_1} \cdots i_{Y_{k-1}}\Theta.
\end{equation}
It follows that $\codim(\sing(Y_j)) \geq 2$ for each $j \in  \{1,\ldots,k\}$, since $\codim(\sing(\omega)) \geq 2$. Let $\h_1,\ldots,\h_{k-1}$ be the germs of the one-dimensional foliations at $0$ defined by $Y_1,\ldots,Y_{k-1}$, respectively. We conclude from Proposition~\ref{P:known_fact} that
\begin{equation}\label{E:proof_the_A_6}
\TD=\TG \oplus T_{\h_1}\oplus \cdots \oplus T_{\h_{k-1}},
\end{equation}
and we are in situation \eqref{1:T:A}.

\item $f_1(0)=\cdots=f_k(0)=0$.

In this case, by considering the linear part of both sides of \eqref{E:proof_the_A}, we see that
\begin{equation}\label{E:proof_the_A_3}
\hat{X}_1=g_1 W_1+\cdots+g_k W_k,
\end{equation}
where $\hat{X}_1$ is the linear part of $X$, $g_1,\ldots,g_k$ are homogeneous polynomials of degree $1$ in $\mathbb{C}^n$ and $W_1,\ldots,W_k \in \mathcal{X}(\mathbb{C}^n,0) $ are germs of constant vector fields. Note that 
\begin{equation}\label{E:proof_the_A_5}
W_1=X_1(0),\ldots,W_k=X_k(0).
\end{equation}
One can check that \eqref{E:proof_the_A_3} implies that the rank of 
$DX(0)=D\hat{X}_1(0)$ is at most $k$, and if it is $k$ then 
$W_1 \wedge \cdots \wedge W_k \neq 0$. By our hypothesis, we obtain that 
the linear rank of $\mathcal{G}$, which is equal to the linear rank of $X$, 
is $k$. Furthermore, from \eqref{E:proof_the_A_4} and \eqref{E:proof_the_A_5} we have
\[
\omega(0)=i_{W_1}\cdots i_{W_k}(dz_1 \wedge \cdots \wedge dz_n) \neq 0,
\]
since $W_1 \wedge \cdots \wedge W_k \neq 0$, which implies that the germ 
of distribution $\mathcal{D}$ is regular. Thus we are in situation \eqref{2:T:A}.
\end{itemize}

Finally, if \eqref{E:proof_the_A_6} holds then it follows from \eqref{E:proof_the_A_2} that $\sing(\G) \subset \sing(\D)$. On the other hand, if $\sing(\G) \subset \sing(\D)$, of course the germ of distribution $\D$ is not regular and we are not in the situation \eqref{2:T:A}, which concludes the proof.
\end{proof}

\begin{cor}\label{C:not_free}
For $n \geq 3$ and $1 \leq k < n$, let $\D$ be a $k$-dimensional distribution defined on a complex manifold $M$ of dimension $n$, and $\G$ be a one-dimensional foliation tangent to $\D$, defined on an open neighborhood of $p \in M$. Assume that $p$ is a singularity of both $\D$ and $\G$, and the linear rank of $\G$ at $p$ is at least $k$. Then:
\begin{enumerate}
\item\label{one:C:not_free} If $\TD$ is locally free at $p$, we have that $\G$ locally divides $\D$ at $p$.
\item\label{two:C:not_free} If $\sing(\G)$ is not contained in $\sing(\D)$, then $\TD$ is not locally free at $p$. In particular $\TD$ is not locally free.
\end{enumerate}
\end{cor}
\begin{proof}
Assume that the germs of $\D$ and $\G$ at $p$ are defined by $\omega \in \Omega^{n-k}(M,p)$ and $X \in \mathcal{X}(M,p)$, respectively, with $\codim(\sing(\omega)) \geq 2$ and $\codim(\sing(X)) \geq 2$. If $\TD$ is locally free at $p$, note that we are in situation \eqref{1:T:A} of Theorem~\ref{T:A}. By Proposition~\ref{P:known_fact}, it follows that an equality holds as in \eqref{E:proof_the_A_2}, for some $Y_1,\ldots,Y_{k-1}\in \mathcal{X}(M,p)$, and we conclude in particular that $\G$ locally divides $\D$ at $p$.

Next, assume by contradiction that $\sing(\G)$ is not contained in $\sing(\D)$ and $\TD$ is locally free at $p$. Then Theorem~\ref{T:A} applies. On the other hand, we are not in situation \eqref{1:T:A} of Theorem~\ref{T:A}, because $\sing(\G)$ is not contained in $\sing(\D)$, and also we are not in situation \eqref{2:T:A} of Theorem~\ref{T:A}, because $\D$ is not regular at $p$. We obtain a contradiction, which shows that \eqref{two:C:not_free} also holds.
\end{proof}

\begin{cor}\label{C:crite_div_loc_free}
For $n \geq 3$ and $1 < k < n$, let $\omega \in \Omega^{n-k}(\mathbb{C}^n,0)$ be such that $\omega(0)=0$, with $\codim(\sing(\omega)) \geq 2$, defining a germ of a $k$-dimensional singular distribution with locally free tangent sheaf, and $X \in \mathcal{X}(\mathbb{C}^n,0) $ satisfying $i_X\omega=0$. If $\rank(DX(0)) \geq k$ then $X$ divides $\omega$.
\end{cor}
\begin{proof}
If $X(0)\neq 0$, then it is clear that $X$ divides $\omega$. On the other 
hand, $X(0)=0$ and $\rank(DX(0)) \geq k$ mean that $X$ is a germ of a 
singular vector field tangent to $\omega$ and with linear rank at least $k$. 
Since $k \geq 2$, one can easily verify that this implies 
$\codim(\sing(X)) \geq 2$. Let $\mathcal{D}$ and $\mathcal{G}$ be the germs 
at $0$ of the $k$-dimensional distribution and the one-dimensional foliation 
defined by $\omega$ and $X$, respectively. By our hypothesis, $\mathcal{T}_{\mathcal{D}}$ 
is locally free, and as $\codim(\sing(X)) \geq 2$, we have that the linear 
rank of $\mathcal{G}$ is the linear rank of $X$. Since $0$ is a singularity 
of both $\mathcal{G}$ and $\mathcal{D}$, from part \eqref{one:C:not_free} of Corollary~\ref{C:not_free} we conclude that $X$ divides $\omega$.
\end{proof}

We see from its proof that Corollary~\ref{C:crite_div_loc_free} is also valid if $k=1$ and $\codim(\sing(X)) \geq 2$.

\begin{remark}\label{R:obs_divi_loc_free_dim_2}
Corollary~\ref{C:crite_div_loc_free} is particularly interesting in the case 
$k=2$. In fact, assuming that $\codim(\sing(\omega)) \geq 2$, in order for 
$X \in \mathcal{X}(\mathbb{C}^n,0)$ to divide 
$\omega \in \Omega^{n-2}(\mathbb{C}^n,0)$, it is necessary that $\omega$ 
define a germ of a two-dimensional distribution $\mathcal{D}$ with locally 
free tangent sheaf (see also Corollary~\ref{C:Converse}). The reason is that 
every $\eta \in \Omega^{n-1}(\mathbb{C}^n,0)$ is of the form 
$\eta = i_Y \Theta$, for some $Y \in \mathcal{X}(\mathbb{C}^n,0)$ and $\Theta \in \Omega^n(\mathbb{C}^n,0)$, see Remark~\ref{R:iso_X_O}. The 
particular case of Corollary~\ref{C:dual_division} guarantees that if 
$\codim(\sing(X)) \geq 3$, then $X$ satisfies the division property by 
$(n-2)$-forms; in particular, $X$ divides $\omega$. Corollary~\ref{C:crite_div_loc_free} provides an alternative condition 
implying that $X$ divides $\omega$, provided we assume a priori that 
$\omega$ defines a germ of a nonregular two-dimensional distribution with 
locally free tangent sheaf.
\end{remark}

Recall that the singular set of a coherent sheaf $\mathcal{T}$ defined on a complex manifold $M$ is
\[
\sing(\mathcal{T})=\{p \in M: \mathcal{T}_p \text{ is not } \mathcal{O}_{M,p}\text{-free}\},
\]
where $\mathcal{T}_p$ denotes the stalk of $\mathcal{T}$ at a point $p \in M$. We have that $\sing(\mathcal{T})$ is an analytic subset of $M$ of codimension at least $1$, see for example the corollary of \cite[Lemma~1.1.4]{OkSchSpi}.

Given a $k$-dimensional holomorphic distribution $\mathcal{D}$ on $M$ and a one-dimensional foliation $\mathcal{G}$ tangent to $\D$, set
\[
\mathcal{S}(\G,\D)=\{p \in M:\G \text{ does not locally divide } \D \text{ at } p\}.
\]

We observe that $\mathcal{S}(\G,\D)$ is a closed subset of $M$. Indeed, if $p \not\in \mathcal{S}(\G,\D)$, i.e., $\G$ locally divides $\D$ at $p$, then by definition there exist an open set $U$ containing $p$, a vector field $X \in \mathcal{X}(U)$, and a form $\omega \in \Omega^{n-2}(U)$, with $\codim(\sing(X)) \geq 2$ and $\codim(\sing(\omega)) \geq 2$, defining the restrictions of $\G$ and $\D$ to $U$, respectively, and a form $\eta \in \Omega^{n-1}(U)$ such that $\omega=i_X\eta$. Then $\G$ locally divides $\D$ at every point of $U$, which shows that $\mathcal{S}(\G,\D)$ is a closed subset of $M$.

Additionally, in the case where $\D$ has dimension $2$, we show that $\mathcal{S}(\G,\D)$ is an analytic subset of $M$. Indeed, we have from Lemma~\ref{R:dim_2_loc_free} that
\[
\mathcal{S}(\G,\D)=\sing(\TD/\TG).
\]
Furthermore, it follows that $\TD/\TG$ is torsion-free, because from Remark~\ref{R:tang_sheaf_properties} we have that $\TM/\TG$ is torsion-free and $\TD \subset \TM$. Since the singular set of a torsion-free coherent sheaf has codimension at least $2$, see for example the corollary of \cite[Lemma~1.1.8]{OkSchSpi}, we obtain that $\codim(\mathcal{S}(\G,\D)) \geq 2$.

In Theorem~\ref{P:divide_pure} below, we present a more precise description of $\mathcal{S}(\G,\D)$ when $\D$ has dimension $2$. The proof will rely on the following:

\begin{lemma}\label{L:divide_out_codim_at_least_3}
Let $\D$ be a two-dimensional distribution on a complex manifold $M$ of dimension $n \geq 3$, and $\G$ a one-dimensional foliation on $M$ tangent to $\D$. Suppose there exists an analytic subset $\mathcal{W} \subset M$ of codimension at least $3$ such that $\G$ locally divides $\D$ at every point $p \in M \setminus \mathcal{W}$. Then $\G$ locally divides $\D$.
\end{lemma}
\begin{proof}
Given $p \in \mathcal{W}$, let $U \subset M$ be an open set containing $p$, biholomorphic to a polydisc in $\C^n$, such that $H^1(U,\mathcal{O}_U)=0$ and there exists a nonvanishing $\Theta \in \Omega^n(U)$. Assume furthermore that $\D_{|U}$ and $\G_{|U}$ are defined by $\omega \in \Omega^{n-2}(U)$ and $X \in \mathcal{X}(U)$, respectively, with $\codim(\sing(\omega)) \geq 2$ and $\codim(\sing(X)) \geq 2$ (see Remark~\ref{R:form_in_polydisc}).

We will adapt a well-known argument that can be found, for instance, in the proof of \cite[Proposition~1]{MR3436562}. Since $H^1(U,\mathcal{O}_U)=0$ and $\codim(\mathcal{W}) \geq 3$, it follows from a theorem of H. Cartan that $H^1(V,\mathcal{O}_V)=0$, where $V:=U\setminus \mathcal{W}$. As $\G$ locally divides $\D$ at every point of $V$, we obtain an open covering $\mathcal{U}=\{U_i\}_{i \in I}$ of $V$ by open sets $U_i$ such that
\begin{equation}\label{E:divide_out_codim_at_least_3}
\omega=i_X \eta_i
\end{equation}
holds on $U_i$, for some $\eta_i \in \Omega^{n-1}(U_i)$. Let $Y_i \in \mathcal{X}(U_i)$ be such that
\begin{equation}\label{E:2:divide_out_codim_at_least_3}
\eta_i = i_{Y_i}\Theta,
\end{equation}
see Remark~\ref{R:iso_X_O}.

From \eqref{E:divide_out_codim_at_least_3} and \eqref{E:2:divide_out_codim_at_least_3}, we conclude that
\[
i_X i_{Y_i-Y_j} \Theta =0
\]
whenever $U_i \cap U_j \neq \emptyset$. Since $\codim(\sing(X)) \geq 2$, Lemma~\ref{R:existence_functions} implies that there exists $g_{ij} \in \mathcal{O}_M(U_i \cap U_j)$ such that $Y_i-Y_j=g_{ij} \cdot X$ on $U_i \cap U_j$. Note that $\{g_{ij}\}_{U_i \cap U_j \ne \emptyset}$ is an additive cocycle. As $H^1(V,\mathcal{O}_V)=0$, the cocycle is trivial, and, after refining the open covering if necessary, there exists a collection $\{h_i\}_{i \in I}$, with $h_i \in \mathcal{O}_M(U_i)$, such that $g_{ij}=h_i-h_j$ on $U_i \cap U_j$. Hence, there exists $Y \in \mathcal{X}(V)$ such that $\omega=i_X i_Y \Theta$ on $V$, and $Y_{|U_i}=Y_i-h_i \cdot X$. Since $\codim(\mathcal{W}) \geq 3$, it follows from Hartog's Theorem that $Y$ extends holomorphically to $U$, which implies that $\G$ locally divides $\D$ at $p$. As $p \in \mathcal{W}$ is arbitrary, the result follows.
\end{proof}

A closer look at the proof of Lemma~\ref{L:divide_out_codim_at_least_3} yields the following:

\begin{cor}\label{C:more_general_div_cod3}
Let $\omega \in \Omega^{n-2}(M)$ and $v \in \mathcal{X}(M)$ be defined on a complex manifold $M$ of dimension $n \geq 3$. Assume that $\codim(\sing (v)) \geq 2$. Define $U \subset M$ as the (open) set of points $p \in M$ for which there exists $\eta \in \Omega^{n-1}(M,p)$ such that
\begin{equation}\label{E:C:more_general_div_cod3}
\omega=i_v\eta
\end{equation}
holds on $(M,p)$. Suppose there exists an analytic subset $\mathcal{W} \subset M$ of codimension at least $3$ such that $M \setminus \mathcal{W} \subset U$. Then $U=M$.
\end{cor}

Corollary~\ref{C:more_general_div_cod3} can be viewed as a generalization of a result provided by Corollary~\ref{C:dual_division}. Indeed, given $X \in \mathcal{X}(\C^n,0)$ and $\omega \in \Omega^{n-2}(\C^n,0)$ satisfying $i_X \omega=0$, the particular case of Corollary~\ref{C:dual_division} asserts that if $\codim(\sing(X)) \geq 3$, then there exists $\eta \in \Omega^{n-1}(\C^n,0)$ such that $\omega=i_X \eta$. Once representatives of $X$ and $\omega$ are chosen in the same polydisc containing $0$, this conclusion also follows from Corollary~\ref{C:more_general_div_cod3} by taking $\mathcal{W}=\sing(X)$, since \eqref{E:C:more_general_div_cod3} holds at the regular points of $X$.

\begin{thm}\label{P:divide_pure}
Let $\D$ be a two-dimensional distribution on a complex manifold $M$ of dimension $n \geq 3$, and $\G$ a one-dimensional foliation on $M$ tangent to $\D$. Then $\mathcal{S}(\G,\D)$ is an analytic subset of $M$ containing $\sing(\TD)$. If $\mathcal{S}(\G,\D)$ is nonempty, it is the union of some irreducible components of $\sing(\G)$ of dimension $n-2$. In particular, whenever $\mathcal{S}(\G,\D)$ is nonempty, it is pure ($n-2$)-dimensional.
\end{thm}
\begin{proof}
Let $\sing(\G)_{n-2} \subset \sing(\G)$ denote the analytic subset of $M$ given by the union of the irreducible components of $\sing(\G)$ of dimension $n-2$. We start by showing that
\[
\sing(\TD) \subset\mathcal{S}(\G,\D) \subset \sing(\G)_{n-2}.
\]

If $p \not\in \mathcal{S}(\G,\D)$, then $\G$ locally divides $\D$ at $p$, and it follows from Lemma~\ref{R:dim_2_loc_free} that $(\TD)_p$ is $\mathcal{O}_{M,p}$-free. This shows that $\sing(\TD) \subset \mathcal{S}(\G,\D)$. On the other hand, if $p \not\in \sing(\G)$, we have in particular that $\codim(\sing(\G),p) \geq 3$. It then follows from the particular case of Corollary~\ref{C:dual_division} that $\G$ locally divides $\D$ at $p$, which shows that $\mathcal{S}(\G,\D) \subset \sing(\G)$. Since $\codim(\sing(\G)) \geq 2$, the same corollary implies that $\mathcal{S}(\G,\D) \subset \sing(\G)_{n-2}$.

Moreover, we already know that $\mathcal{S}(\G,\D)$ is an analytic subset of $M$. Lemma~\ref{L:divide_out_codim_at_least_3} clearly ensures that $\mathcal{S}(\G,\D)$ has no irreducible component of codimension at least $3$. Since $\mathcal{S}(\G,\D) \subset \sing(\G)_{n-2}$, the result follows.
\end{proof}

Although Definition~\ref{D:def_division} is local in nature, its global aspect becomes evident in the next two corollaries.

\begin{cor}\label{C:check_one-point}
Let $\D$ be a two-dimensional distribution on a complex manifold $M$ of dimension $n \geq 3$, and $\G$ a one-dimensional foliation on $M$ tangent to $\D$. Suppose that for every irreducible component $\mathcal{V}$ of $\sing(\G)$ of dimension $n-2$, there exists $p \in \mathcal{V}$ such that $\G$ locally divides $\D$ at $p$. Then $\G$ locally divides $\D$.
\end{cor}
\begin{proof}
Assume, on the contrary, that $\G$ does not locally divide $\D$, that is, $\mathcal{S}(\G,\D) \neq \emptyset$. Let $q \in \mathcal{S}(\G,\D)$. By Theorem~\ref{P:divide_pure}, there exists an irreducible component $\mathcal{V}$ of $\sing(\G)$ of dimension $n-2$ passing through $q$ such that $\mathcal{V} \subset \mathcal{S}(\G,\D)$.

By hypothesis, there exists $p \in \mathcal{V}$ such that $\G$ locally divides $\D$ at $p$, that is, $p \notin \mathcal{S}(\G,\D)$, which yields a contradiction.
\end{proof}

Let $\G$ and $\D$ be as in Theorem~\ref{P:divide_pure}. Given $p \in \mathcal{S}(\G,\D)$, Theorem~\ref{P:divide_pure} ensures that there exists at least one irreducible component $\mathcal{V}$ of $\sing(\G)$ of dimension $n-2$ containing $p$ such that $\mathcal{V} \subset \mathcal{S}(\G,\D)$. Note that, in principle, there may exist other irreducible components of $\sing(\G)$ of dimension $n-2$ passing through $p$ that are not contained in $\mathcal{S}(\G,\D)$.

As a consequence, given an irreducible component $\mathcal{V}$ of $\sing(\G)$ of dimension $n-2$ and a point $p \in \mathcal{V}$ that does not lie in any other such component, it follows that $\mathcal{V} \subset \mathcal{S}(\G,\D)$ if and only if $\G$ does not locally divide $\D$ at $p$. We thus obtain the following:

\begin{cor}\label{C:finite_check}
Let $\D$ be a two-dimensional distribution on a complex manifold $M$ of dimension $n \geq 3$, and $\G$ a one-dimensional foliation on $M$ tangent to $\D$. Let $\{\mathcal{V}_i\}_{i \in I}$ denote the family of pairwise distinct irreducible components of $\sing(\G)$ of dimension $n-2$. For arbitrary points $p_i \in \mathcal{V}_i$ with $p_i \not\in \mathcal{V}_j$ whenever $j \neq i$, we have
\[
\mathcal{S}(\G,\D)=\bigcup_{\G \text{ does not locally divide }\D \text{ at } p_i}  \mathcal{V}_i.
\]
In particular, if $M$ is compact, there exists a finite subset $\{p_1,\ldots,p_k\} \subset M$ such that $\G$ locally divides $\D$ if and only if $\G$ locally divides $\D$ at each point of $\{p_1,\ldots,p_k\}$.
\end{cor}

Note that the particular case of the previous corollary follows simply from the fact that the number of irreducible components of $\sing(\G)$ of dimension $n-2$ is finite when $M$ is compact.

Given \(\omega \in \Omega^{n-k}(\mathbb{C}^n,0)\) and 
\(X \in \mathcal{X}(\mathbb{C}^n,0)\), to distinguish from the usual 
notion of division, we say that \(X\) \emph{divides} \(\omega\) 
\emph{specially} if there exist 
\(X_1,\ldots,X_{k-1} \in \mathcal{X}(\mathbb{C}^n,0)\) 
and \(\Theta \in \Omega^{n}(\mathbb{C}^n,0)\) such that
\[
  \omega = i_X i_{X_1} \cdots i_{X_{k-1}} \Theta.
\]

Clearly, if \(X\) divides \(\omega\) specially, then \(X\) divides 
\(\omega\). Moreover, when \(k=2\), the notions of division 
and special division coincide; see 
Remark~\ref{R:obs_divi_loc_free_dim_2}.

As in Definition~\ref{D:def_division}, it can be verified that the definition of special division extends to germs of distributions and germs of one-dimensional foliations tangent to them. The analogue to Definition~\ref{D:def_division} would be the following: For a $k$-dimensional distribution $\D$ on a complex manifold $M$ of dimension $n \geq 3$, and $\G$ a one-dimensional foliation on $M$ tangent to $\D$, \(\G\) 
\emph{locally divides} \(\D\) \emph{specially} at \(p \in M\) if, for any $X \in \mathcal{X}(M,p)$ and any $\omega \in \Omega^{n-k}(M,p)$, with $\codim(\sing(X)) \geq 2$ and $\codim(\sing(\omega)) \geq 2$, defining the germs of $\G$ and $\D$ at $p$, respectively, there are $X_1,\ldots,X_{k-1}  \in \mathcal{X}(M,p)$ and $\Theta \in \Omega^{n}(M,p)$ such that
\[
\omega=i_X i_{X_1} \cdots i_{X_{k-1}}\Theta.
\]

We have the following:

\begin{prop}
Let $\D$ be a $k$-dimensional distribution on a complex manifold $M$ of dimension $n \geq 3$, and $\G$ a one-dimensional foliation on $M$ tangent to $\D$. Assume that $\TD$ is locally free. Let $\mathcal{S}(\G,\D)_e$ be the set of points $p \in M$ where $\G$ does not locally divide $\D$ specially at $p$. Then $\mathcal{S}(\G,\D)_e$ is an analytic subset of $M$, and if it is not empty then it is of dimension at least $n-k$.
\end{prop}
\begin{proof}
Since $\mathcal{T}_{\mathcal{D}}$ is locally free, by Proposition~\ref{P:known_fact_2}, for each $p \in M$ there exists an open set $U \subset M$ containing $p$ and vector fields $X_1,\ldots,X_k \in \mathcal{X}(U)$, together with $\Theta \in \Omega^n(U)$, such that
\[
\omega = i_{X_1} \cdots i_{X_k}\Theta
\]
satisfies $\codim(\sing(\omega)) \geq 2$ and defines $\mathcal{D}_{|U}$.

By shrinking $U$ if necessary, we can assume that $\G_{|U}$ is defined by $X \in \mathcal{X}(U)$, with $\codim(\sing(X)) \ge 2$. By Lemma~\ref{R:existence_functions}, there are $f_1,\ldots,f_k \in \mathcal{O}_M(U)$ such that $X=f_1 X_1 +\cdots+f_k X_k$. It follows from Lemma~\ref{L:divides_f_and_g} that
\[
\mathcal{S}(\G,\D)_e \cap U = \{q \in U: f_1(q)=\cdots=f_k(q)=0\},
\]
which is sufficient to finish the proof.
\end{proof}

Theorem~\ref{T:B} follows from the following:

\begin{lemma}\label{T:B_more_general}
Let $\D$ be a two-dimensional distribution on a complex manifold $M$ of dimension $n \geq 3$, with locally free tangent sheaf, and $\G$ a one-dimensional foliation on $M$, tangent to $\D$ and with linear rank at least $2$. Then the following are equivalent:
\begin{enumerate}
\item\label{1:T:B_more_general} Every irreducible component of $\sing(\G)$ of dimension $n-2$ intersects $\sing(\D)$.
\item\label{2:T:B_more_general} Every irreducible component of $\sing(\G)$ intersects $\sing(\D)$.
\item\label{3:T:B_more_general} $\sing(\G) \subset \sing(\D)$.
\item\label{4:T:B_more_general} $\G$ locally divides $\D$.
\end{enumerate}
\end{lemma}
\begin{proof}
It is clear that \eqref{4:T:B_more_general} implies \eqref{3:T:B_more_general}, \eqref{3:T:B_more_general} implies \eqref{2:T:B_more_general}, and \eqref{2:T:B_more_general} implies \eqref{1:T:B_more_general}. Assume that \eqref{1:T:B_more_general} holds. We will show that $\G$ locally divides $\D$, hence \eqref{4:T:B_more_general} holds and the result follows. 

In fact, assume by contradiction that there exists $q \in M$ such that $\G$ does not locally divide $\D$ at $q$. By Theorem~\ref{P:divide_pure}, there is an irreducible component $\mathcal{W}$ of $\sing(\G)$ of dimension $n-2$ such that $q \in \mathcal{W}$ and $\G$ fails to locally divide $\D$ at every point of $\mathcal{W}$. By our assumption, there is $p \in \mathcal{W} \cap \sing(\D)$, and since $\TD$ is locally free at $p$, it follows from part \eqref{one:C:not_free} of Corollary~\ref{C:not_free} that $\G$ locally divides $\D$ at $p$, yielding a contradiction.
\end{proof}

\begin{proof}[Proof of Theorem~\ref{T:B}]
Theorem~\ref{T:B} follows from Lemma~\ref{T:B_more_general} with 
\(M = \mathbb{P}^n\), \(n \geq 3\), since in this case condition 
\eqref{4:T:B} of Theorem~\ref{T:B} is equivalent to the fact that 
\(\G\) locally divides \(\D\). 

Indeed, if condition \eqref{4:T:B} of Theorem~\ref{T:B} holds, then 
in particular we have 
\((\TD)_p = (\TG)_p \oplus (\Th)_p\) for every \(p \in M\). 
By Lemma~\ref{R:dim_2_loc_free}, it follows that \(\G\) locally 
divides \(\D\) at every \(p \in M\); hence \(\G\) locally divides \(\D\). 

Conversely, if \(\G\) locally divides \(\D\), then condition 
\eqref{4:T:B} of Theorem~\ref{T:B} follows from 
Corollary~\ref{C:divides}.

For the particular case, if $\mathcal{W}$ is an irreducible component of dimension $n-2$ of $\sing(\G)$ and $\mathcal{T}$ is an irreducible component of dimension $\geq 2$ of $\sing(\D)$, as they are in particular projective varieties, from the respective dimensions we see that $\mathcal{W} \cap \mathcal{T} \neq \emptyset$, then the result follows.
\end{proof}

In Lemma~\ref{T:B_more_general}, the assumption that the tangent sheaf is locally free can be relaxed, at the cost of a corresponding weakening of the conclusion.

\begin{lemma}\label{T:B_more_general_without_loc_free}
Let $\D$ be a two-dimensional distribution on a complex manifold $M$ of dimension $n \geq 3$, and $\G$ a one-dimensional foliation on $M$, tangent to $\D$ and with linear rank at least $2$. Then the following are equivalent:
\begin{enumerate}
\item\label{1:T:B_more_general_without_loc_free} $\sing(\G) \subset \sing(\D)$.
\item\label{2:T:B_more_general_without_loc_free} $\G$ locally divides $\D$.
\end{enumerate}
\end{lemma}
\begin{proof}
It is clear that \eqref{2:T:B_more_general_without_loc_free} implies 
\eqref{1:T:B_more_general_without_loc_free}. Assume that 
\eqref{1:T:B_more_general_without_loc_free} holds, and let 
\(\mathcal{V}\) be an arbitrary irreducible component of 
\(\sing(\G)\) of dimension \(n-2\). 

Note that \(\sing(\TD)\) has codimension at least \(3\), because 
\(\TD\) is reflexive (see Remark~\ref{R:tang_sheaf_properties}), 
and the singular set of a coherent reflexive sheaf has codimension 
at least \(3\); see, for example, 
\cite[Lemma~1.1.10]{OkSchSpi}. 

Thus, we can choose a point 
\(p \in \mathcal{V} \subset \sing(\D)\) such that 
\(\TD\) is locally free at \(p\). It then follows from 
part~\eqref{one:C:not_free} of Corollary~\ref{C:not_free} 
that \(\G\) locally divides \(\D\) at \(p\). 
By Corollary~\ref{C:check_one-point}, we conclude that 
\(\G\) locally divides \(\D\).
\end{proof}

Using the same argument as in the proof of Theorem~\ref{T:B}, Lemma~\ref{T:B_more_general_without_loc_free} yields the following:

\begin{cor}\label{C:T:B_without_loc_free}
Let $\D$ be a two-dimensional distribution on $\Pj^n$, $n \geq 3$, and $\G$ a one-dimensional foliation on $\Pj^n$, tangent to $\D$ and with linear rank at least $2$. Then the following are equivalent:
\begin{enumerate}
\item\label{1:C:T:B_without_loc_free} $\sing(\G) \subset \sing(\D)$.
\item\label{2:C:T:B_without_loc_free} There is a one-dimensional foliation $\h$ on $\mathbb{P}^n$ such that $\TD=\TG \oplus \Th$.
\end{enumerate}
\end{cor}

Corollary~\ref{C:T:B_without_loc_free} has the following interpretation: Let $\D$ be a two-dimensional distribution and $\G$ a one-dimensional foliation tangent to $\D$, both defined on $\mathbb{P}^n$, $n \geq 3$. For there to exist another one-dimensional foliation $\mathcal{H}$ on $\mathbb{P}^n$ tangent to $\D$ such that $\TD=\TG \oplus \Th$, it is clearly necessary that $\sing(\G) \subset \sing(\D)$. Corollary~\ref{C:T:B_without_loc_free} tells us that this condition is also sufficient when the linear rank of $\G$ is at least $2$.

Let $\mathcal{D}$ be a two-dimensional distribution on $\mathbb{P}^n$, $n \geq 3$, with locally free tangent sheaf. Assume that there exists a one-dimensional foliation $\mathcal{G}$ on $\mathbb{P}^n$ tangent to $\mathcal{D}$ and with linear rank at least $3$. For every point $p \in \sing(\G)$, we must have $p \in \sing(\D)$; otherwise, we would be in situation \eqref{2:T:A} of Theorem~\ref{T:A}, where the linear rank of $\G$ at $p$ equals $2$, contradicting the assumption that it is at least $3$. By applying part~\eqref{one:C:not_free} of Corollary~\ref{C:not_free}, we deduce that $\G$ locally divides $\D$ at $p$. Since $p$ is arbitrary in $\sing(\G)$, it follows that $\G$ locally divides $\D$. We obtain from Corollary~\ref{C:divides} that $\TD$ splits. In fact, the hypothesis that $\TD$ is locally free can be dropped, according to a particular case of the following:

\begin{thm}\label{T:version_lr_general}
Let $\D$ be a $k$-dimensional distribution (not necessarily with locally free
tangent sheaf) on a complex manifold $M$ of dimension $n \geq 3$, and $\G$ a one-dimensional foliation on $M$ tangent to $\D$, where $1\leq k< n$. If $\G$ has linear rank at least $k+1$ then $\G$ locally divides $\D$. In particular, for $k=2$ and $M=\Pj^n$ we have that $\TD$ splits.
\end{thm}
\begin{proof}
It follows from Lemma~\ref{L:lin_rank_3} below that $\codim(\sing(\G)) \geq k+1$. By Corollary~\ref{C:dual_division}, any vector field defining $\mathcal{G}$ on an open neighborhood of each point $p \in M$, and having no zeros of codimension one, satisfies the division property by ($n-k$)-forms. The conclusion then follows, since $\mathcal{D}$ is defined on an open neighborhood of every point $p \in M$ by an $(n-k)$-form. The particular case follows by applying Corollary~\ref{C:divides}.
\end{proof}

\begin{lemma}\label{L:lin_rank_3}
For a germ of a vector field $v \in \mathcal{X}(\C^n,0)$, we have $\codim(\sing(v)) \geq \rank(Dv(0))$. In particular, if $\G$ is a one-dimensional foliation with linear rank at least $k$ on a complex manifold $M$, then $\codim(\sing(\G)) \geq k$.
\end{lemma}
\begin{proof}
If $v(0) \neq 0$, we have that $\codim(\sing(v))=n+1$ and the result follows, since $\rank(Dv(0))$ is at most $n$. If $v(0)=0$, we prove the following: Let $\mathcal{V} \neq \emptyset$ be a germ of an analytic set defined at $0 \in \mathbb{C}^n$. Assume that $\mathcal{V}$ is represented by
\[
\{z \in U:f_1(z)=\cdots=f_m(z)=0\},
\]
where $f_1,\ldots,f_m$ are holomorphic functions defined on an open neighborhood $U$ of the origin in $\mathbb{C}^n$. For the $ m \times n$ matrix
\[
M(z)=
\biggl[
\frac{\partial f_i}{\partial z_j}(z)
\biggr]_{i=1,\ldots,m,\;j=1,\ldots,n}.
\]
we will prove that $\codim(\mathcal{V}) \geq \rank(M(0))$. Then we conclude by setting $\mathcal{V}=\sing(v)$, where
\[
v(z)=f_1(z)\frac{\partial}{\partial z_1}+\cdots+f_n(z)\frac{\partial}{\partial z_n},
\]
since in this case $M(0)=Dv(0)$.

Denote by $\mathcal{I}_p \subset \mathcal{O}_{\mathbb{C}^n,p}$ the ideal of $\mathcal{V}$ at a point $p \in U \cap \mathcal{V}$. If the open set $U$ is sufficiently small, we can assume that there are finitely many holomorphic functions $g_1,\ldots,g_r$ defined on $U$, such that for each point $p \in U \cap \mathcal{V}$ the germs of $g_1,\ldots,g_r$ at $p$ generate $\mathcal{I}_p$. Let $\mathcal{V}_1$ be an irreducible component of $\mathcal{V}$ for which $\dim(\mathcal{V})=\dim(\mathcal{V}_1)$. If $p$ is a regular point of $\mathcal{V}_1$, we have
\[
\codim(\mathcal{V})=\codim(\mathcal{V}_1)=\rank(J(p)),
\]
where  
where
\[
J(z)=
\biggl[
\frac{\partial g_i}{\partial z_j}(z)
\biggr]_{i=1,\ldots,r,\;j=1,\ldots,n}.
\]

Since the germs of $g_1,\ldots,g_r$ at $p$ generate $\mathcal{I}_p$, there exist germs of holomorphic functions $h_{ij},i=1,\ldots,m,j=1,\ldots,r$, defined at $p$, such that
\begin{equation}\label{E:germs}
f_{i}=h_{i1}g_1+\cdots+h_{ir} g_r,i=1,\ldots,m.
\end{equation}

Taking derivative at both sides of \eqref{E:germs} with respect to $z_j,j=1,\ldots,n$, since $g_1(p)=\cdots=g_r(p)=0$, we have
\[
\frac{\partial f_i}{\partial z_j}(p)=h_{i1}(p)\frac{\partial g_1}{\partial z_j}(p)+\cdots+h_{ir}(p) \frac{\partial g_r}{\partial z_j}(p),i=1,\ldots,m,j=1,\ldots,n.
\]

The last set of equations tells us that the vector rows of the matrix $M(p)$ are in the row space of $J(p)$. Thus
\[
\codim(\mathcal{V})=\rank(J(p)) \ge \rank(M(p)).
\]

Finally, since the map $z \mapsto \rank(M(z))$ is lower semi-continuous, if $p$ is sufficiently close to $0$ it follows that $\rank(M(p)) \geq \rank(M(0))$, which concludes the proof.
\end{proof}

\begin{remark}
Unlike the particular case of Theorem~\ref{T:version_lr_general}, we cannot remove the hypothesis that $\TD$ is locally free from Theorem~\ref{T:B}. See, for example, \cite[Proposition~1]{Lizarbe17}.
\end{remark}

Finally, we obtain the following division criteria for vector fields and differential forms.

\begin{prop}\label{P:new_division_result}
Let $X \in \mathcal{X}(\mathbb{C}^n,0)$ be a germ of a vector field such that $\rank (DX(0)) \geq p+1$, where $p$ is a positive integer. Then $X$ satisfies the division property by $k$-forms for $n-p \leq k \leq n-1$. 
\end{prop}
\begin{proof}
By Lemma~\ref{L:lin_rank_3} we have that
\[\codim(\sing(X)) \geq \rank (DX(0)) \geq p+1.
\]
The result follows from Corollary~\ref{C:dual_division}.
\end{proof}

For $\omega \in  \Omega^1(\mathbb{C}^n,0)$, write $\omega=\sum_{i=1}^n A_i(z_1,\ldots,z_n)dz_i$. Define $X_\omega \in \mathcal{X}(\mathbb{C}^n,0)$ by
\[
X_\omega=\sum_{i=1}^n A_i(z_1,\ldots,z_n) \frac{\partial }{\partial z_i}.
\]

\begin{dfn}
Given $\omega \in  \Omega^1(\mathbb{C}^n,0)$, we define the rank of the linear part of $\omega$ as the rank of $D X_\omega(0)$.
\end{dfn}

We have defined $X_\omega$ as above for the sake of convenience, but note that the rank of the linear part of $\omega$ coincides with the rank of the linear part of $\hat{\omega}_1$, where $\omega=\sum_{i=0}^\infty \hat{\omega}_i$.

\begin{prop}\label{P:new_div_dual}
Let $\omega \in \Omega^1(\mathbb{C}^n,0)$ be a germ of a $1$-form such that the rank of the linear part of $\omega$ is at least $p+1$, where $p$ is a nonnegative integer. Then $\omega$ satisfies the division property by $k$-forms for $k \leq p$.
\end{prop}
\begin{proof}
By our assumption we have that $\rank(DX_\omega(0)) \geq p+1$. Since $\sing(\omega)=\sing(X_\omega)$, by Lemma~\ref{L:lin_rank_3} it follows that
\[\codim(\sing(\omega))=\codim(\sing(X_\omega)) \geq \rank (DX_\omega(0)) \geq p+1.
\]
The result follows from Lemma~\ref{L:De-S}.
\end{proof}

Lemma~\ref{L:De-S} admits the following converse: If $\omega \in \Omega^1(\C^n,0)$ satisfies the division property by $k$-forms for $k \leq p$, then it follows that $\codim( \sing (\omega)) \geq p+1$, see \cite[Theorem~17.4]{eis}. By duality, given a positive integer $p$, if $X  \in \mathcal{X}(\mathbb{C}^n,0)$ satisfies the division property by $k$-forms, for $n-p \leq k \leq n-1$, then $\codim(\sing(X)) \geq p+1$. So it is no surprise that Corollary~\ref{C:dual_division} is used in the proof of Proposition~\ref{P:new_division_result}, as well as Lemma~\ref{L:De-S} is used in the proof of Proposition~\ref{P:new_div_dual}. It may be easier to check Proposition~\ref{P:new_division_result} instead of Corollary~\ref{C:dual_division} in some situations, as well as Proposition~\ref{P:new_div_dual} instead of Lemma~\ref{L:De-S}.

\begin{example}
For \(n \geq 3\), let 
\[
  X = \sum_{i=1}^n A_i \frac{\partial}{\partial z_i} 
  \in \mathcal{X}(\C^n,0)
\]
and let \(\eta \in \Omega^{k}(\C^n,0)\) be such that \(i_X\eta = 0\).  
Set 
\[
  \omega = \sum_{i=1}^n A_i dz_i \in \Omega^1(\C^n,0)
\]
and let \(\alpha \in \Omega^{l}(\C^n,0)\) be such that 
\(\omega \wedge \alpha = 0\).

If \(X(0) \neq 0\), then \(\omega(0) \neq 0\), and it is clear that 
\(X\) divides \(\eta\) and \(\omega\) divides \(\alpha\). 
From now on, assume that \(X(0)=0\) and that the linear part of \(X\) 
is of the form
\[
  \hat{X}_1 = 
  z_1\frac{\partial}{\partial z_1} +
  z_2\frac{\partial}{\partial z_2} +
  \lambda z_3\frac{\partial}{\partial z_3} +
  \sum_{i=4}^n H_i \frac{\partial}{\partial z_i},
\]
where \(\lambda \in \C\) and \(H_4,\ldots,H_n\) are homogeneous 
polynomials of degree \(1\) in the variables 
\(z_1,\ldots,z_n\). 

If \(\lambda \neq 0\), then the rank of \(DX(0)\) is at least \(3\), 
and it follows from Proposition~\ref{P:new_division_result} 
that \(X\) divides \(\omega\) for \(k=n-2\) or \(k=n-1\). 
In this case, the rank of the linear part of \(\omega\) is also at 
least \(3\), and it follows from Proposition~\ref{P:new_div_dual} 
that \(\omega\) divides \(\alpha\) for \(l \leq 2\). 

On the other hand, if \(\lambda = 0\), we see that the rank of 
\(DX(0)\) is at least \(2\). 
If \(\eta(0) \neq 0\), then clearly \(X\) does not divide \(\omega\), 
because \(X(0)=0\). 
If \(\eta(0)=0\), in the particular case where \(k=n-2\) and \(\eta\) satisfies $\codim(\sing(\eta)) \geq 2$ and defines a germ of a two-dimensional distribution with locally free 
tangent sheaf, it follows from 
Corollary~\ref{C:crite_div_loc_free} that \(X\) divides \(\eta\).
\end{example}

\section{Proof of Theorem~\ref{T:structure}}

Let $v$ be a holomorphic vector field on $\mathbb P^3$. Recall that it can be represented by a class of homogeneous linear polynomial vector field on $\mathbb C^4$
\[
v=\sum_{i=0}^{3} v_i(x_0,\ldots,x_3) \frac{\partial}{\partial x_i}
\]
modulo 
\[
R=\sum_{i=0}^3 x_i \frac{\partial}{\partial x_i},
\]
the radial vector field.

Let $S$ be the $4 \times 4$ matrix with complex entries such that 
\begin{equation}\label{E:jordan_vector}
v=S \cdot \mathrm{X},
\end{equation}
where $\mathrm{X}=[x_0 \cdots x_3]^T$ is a column vector. Up to a linear automorphism of $\Pj^3$, we can assume that $S$ is in Jordan normal form. In particular, if $v \neq 0$ is nilpotent, it can be represented by
\begin{equation}\label{E:nil_P3}
v_1 = x_1\frac{\partial}{\partial x_0} + x_3\frac{\partial}{\partial x_2}
\quad \text{or} \quad
v_2 = x_1\frac{\partial}{\partial x_0} + x_2\frac{\partial}{\partial x_1}.
\end{equation}

Let $\eta = \sum_{j=0}^{\infty} \hat{\eta}_j \in \Omega^1(\mathbb{C}^3,0)$, with $\codim(\sing(\eta)) \geq 2$, defining a germ of a codimension one distribution $\D$ on $(\mathbb{C}^3,0)$. The algebraic multiplicity of $\eta$ (or $\D$) is defined as the smallest integer $k$ such that $\hat{\eta}_k \neq 0$.

The following proposition will be used in the proof of Lemma~\ref{L:divides_in_one_point} below, which in turn will be used in the proof of Theorem~\ref{T:structure}.

\begin{prop}\label{L:alg_mult_at_least_2}
Let $\F$ be a codimension one foliation on $\Pj^3$, of degree $\deg(\F) \geq 2$. Let $v \neq 0$ be a nilpotent vector field on $\Pj^3$ represented by $v_1$ or $v_2$ as in \eqref{E:nil_P3}. If $v$ is tangent to $\F$, then $p=(1:0:0:0)$ is a singular point of $\F$ with algebraic multiplicity at least $2$.
\end{prop}

\begin{lemma}\label{L:divides_in_one_point}
Let $\F$ be a codimension one foliation on $\Pj^3$, with locally free tangent sheaf, of degree $\deg(\F) \geq 2$. Let $v \neq 0$ be a nilpotent vector field on $\Pj^3$ represented by $v_1$ or $v_2$ as in \eqref{E:nil_P3}. If $v$ is tangent to $\F$, then $v$ locally divides $\F$ at $p=(1:0:0:0)$.
\end{lemma}

It is convenient for the proof of Proposition~\ref{L:alg_mult_at_least_2} to use tools related to directed graphs. We briefly recall their definition and introduce some related objects, which will be useful here and are not necessarily standard in the literature.

A directed graph (or digraph) is an ordered pair $G=(V,A)$ where: 
\begin{itemize}
\item $V$ (or $V(G)$) is a set whose elements are called \emph{vertices}.
\item $A$ (or $A(G)$) is a set of ordered pairs of vertices, called \emph{arcs}, with
\[
A \subset \{(u,w)\in V\times V \mid u \neq w\}.
\]
\end{itemize}

In the definition above, we do not allow the existence of \emph{loops} in $A$, that is, arcs of the form $e = (u,u)$ with $u \in V$. We emphasize that this is not universally adopted, and our choice is made for convenience.

In the representation of a directed graph $G$, we represent an arc $e = (u,w)$ by an arrow starting at $u$ and pointing to $w$. Given $u \in V$, define
\[
V_G^{+}(u)=\{w \in V : (u,w) \in A\}, \qquad 
V_G^{-}(u)=\{w \in V : (w,u) \in A\}.
\]

We denote by $\deg_G^{+}(u)$ and $\deg_G^{-}(u)$ the cardinalities of $V_G^{+}(u)$ and $V_G^{-}(u)$, respectively. A vertex $u \in V$ with $\deg_G^{-}(u)=0$ is called a \emph{source}. Similarly, a vertex with $\deg_G^{+}(u)=0$ is called a \emph{sink}. We define the index of a source (resp. a sink) $u$ as $\deg_G^{+}(u)$ (resp. $\deg_G^{-}(u)$).

A directed graph in which every vertex is either a source or a sink will be called a \emph{source--sink digraph} (abbreviated as \emph{ss-digraph}). An ss-digraph is a particular case of a type of directed graph known as \emph{bipartite}.

We take this opportunity to introduce, in the next example, two directed graphs that will be important in the proof of Proposition~\ref{L:alg_mult_at_least_2}.

\begin{example}\label{Ex:intro_two_subgraph}
Let $E_d$ be the $\C$-vector space of homogeneous $1$-forms in the variables $x_0,\ldots,x_3$ of degree $d+1$, where $d \geq 0$. It follows that $E_d$ has as a basis
\[
B_d=\{x^I dx_m:|I|=d+1,m=0,1,2,3\},
\]
where for $I=(i,j,k,l)$ we denote $x^I=x_{0}^i x_{1}^j x_{2}^k x_{3}^{l}$ and $|I|=i+j+k+l$.

Let $v=v_1$ be as in \eqref{E:nil_P3}. Then $L_v:E_d \to E_d$ is a $\C$-linear map, where $L_v \Omega$ denotes the Lie derivative of $\Omega \in E_d$ with respect to $v$.

We have
\begin{equation}\label{eq_lie}
\begin{cases}
L_v(x^Idx_m)&=(ix_0^{i-1}x_{1}^{j+1}x_{2}^kx_{3}^l+kx_0^{i}x_{1}^{j}x_{2}^{k-1}x_{3}^{l+1})dx_m+x^Idx_{m+1}, \\
L_v(x^Idx_m)&=(ix_0^{i-1}x_{1}^{j+1}x_{2}^kx_{3}^l+kx_0^{i}x_{1}^{j}x_{2}^{k-1}x_{3}^{l+1})dx_m,
\end{cases}
\end{equation}
where the first equation of \eqref{eq_lie} holds for $m=0$ and $m=2$, and the second equation holds for $m=1$ and $m=3$. 

Denote by $\hat{G}(d)$ the directed graph whose vertex set is $V=B_d$, and whose arcs are the pairs
\[
e=(x^I dx_m, x^J dx_l)
\]
such that $x^J dx_l$ appears in the expansion of $L_v(x^I dx_m)$ with respect to the basis $B_d$ of $E_d$, that is, the coefficient of $x^J dx_l$ in this expansion is nonzero. Note that the arcs can be determined through relation \eqref{eq_lie}.

On the other hand, if $v=v_2$ is as in \eqref{E:nil_P3}, we have
\begin{equation}\label{eq_lie_vi}
\begin{cases}
L_v(x^Idx_m)&=(ix_0^{i-1}x_{1}^{j+1}x_{2}^kx_{3}^l+jx_0^{i}x_{1}^{j-1}x_{2}^{k+1}x_{3}^{l})dx_m+x^Idx_{m+1}, \\
L_v(x^Idx_m)&=(ix_0^{i-1}x_{1}^{j+1}x_{2}^kx_{3}^l+jx_0^{i}x_{1}^{j-1}x_{2}^{k+1}x_{3}^{l})dx_m,
\end{cases}
\end{equation}
where the first equation of \eqref{eq_lie_vi} holds for $m=0$ and $m=1$, and the second equation holds for $m=2$ and $m=3$. We define $\tilde{G}(d)$ in a similar way to how we defined $\hat{G}(d)$.

With the exception of the final part of the proof of Proposition~\ref{L:alg_mult_at_least_2}, for simplicity we will write $\hat{G}$ and $\tilde{G}$ instead of $\hat{G}(d)$ and $\tilde{G}(d)$, respectively.

Note that \eqref{eq_lie} and \eqref{eq_lie_vi} ensure that $\hat{G}$ and $\tilde{G}$ do not have loops. From \eqref{eq_lie} and for $G=\hat{G}$, it follows that
\begin{equation}\label{E:max_min_vert_degre}
0 \leq \deg_G^{+}(x^I dx_m), \deg_G^{-}(x^I dx_m) \leq 3, \quad \text{for all } x^I dx_m \in B_d.
\end{equation}
From \eqref{eq_lie_vi} and for $G=\tilde{G}$, we have that \eqref{E:max_min_vert_degre} also holds.
\end{example}

A \emph{special path} in a directed graph $G$ of length $k \geq 1$ from $u_0$ to $u_k$ is a finite sequence
\begin{equation}\label{E:special_path}
(u_0,e_1,u_1,e_2,\ldots,e_k,u_k)
\end{equation}
such that
\begin{itemize}
\item $u_0,\ldots,u_k \in V(G)$ and $e_1,\ldots,e_k \in A(G)$.
\item For each $j=1,\ldots,k$, we have $e_j=(u_{j-1},u_{j})$ if $j$ is odd, and $e_j=(u_{j},u_{j-1})$ if $j$ is even.
\item $u_j \neq u_{j+2}$ for all $j=0,\ldots,k-2$.
\end{itemize}

Note that it follows from the definition of a directed graph that $u_j \neq u_{j+1}$ for $j=0,\ldots,k-1$. Sometimes, we represent a special path as in \eqref{E:special_path} by
\[
u_0 \to u_1 \gets u_2 \to u_3 \gets \cdots.
\]

As will become clear shortly, in the context of Example~\ref{Ex:intro_two_subgraph}, and assuming that $\Omega \in E_d$ satisfies $L_v \Omega=0$, the idea is that the vertices $u_j$ with $j$ even in the special path \eqref{E:special_path} encode the terms that must appear in the expansion of $\Omega$ with respect to the basis $B_d$, while those with $j$ odd encode the terms that appear in the expansion of $L_v \Omega$ with respect to the same basis. In this sense, the definition of a special path is motivated by Lemma~\ref{L:cons_two_elements_special} below.

Given $u \in V(G)$, we define the directed graph associated to $u$, denoted by $G_u$, as the (directed) subgraph of $G$ whose vertices and arcs are those belonging to a special path as in \eqref{E:special_path} with $u_0 = u$. Loosely speaking, $G_u$ is the subgraph of $G$ reachable from $u$ by special paths.

\begin{remark}\label{R:constructing_G_u}
When $G$ is finite, that is, $V(G)$ is finite, $G_u$ can be constructed algorithmically by determining the possible vertices $u_1$ such that $u_1 \in V_G^+(u)$, then the possible vertices $u_2 \neq u$ such that $u_2 \in V_G^-(u_1)$ (or equivalently $u_1 \in V_G^+(u_2)$), then the possible vertices $u_3 \neq u_1$ such that $u_3 \in V_G^+(u_2)$, and so on.
\end{remark}

\begin{prop}\label{P:G_w_equals_G_u}
Let $G$ be a directed graph and $u \in V(G)$ such that $G_u$ is an ss-digraph. Then $w \in V(G_u)$ is a source if and only if there exists a special path in $G_u$ of even length from $u$ to $w$. Moreover, if $w$ is a source in $G_u$, then $G_w = G_u$.
\end{prop}
\begin{proof}
Let $w \in V(G_u)$. By definition, there exists a special path in $G$
\[
\gamma_1=(u_0=u,e_1,u_1,e_2,\ldots,e_k,u_k=w)
\]
from $u$ to $w$. By the definition of $G_u$, the path $\gamma_1$ lies in $G_u$. Since $G_u$ is an ss-digraph, it follows that for any such path, $u_j$ is a source in $G_u$ if and only if $j$ is even. This proves the first assertion.

Next, assume that $w$ is a source in $G_u$. Then there exists a special path in $G$
\[
\gamma_2=(u_k=w,e_k,\ldots,e_1,u_0=u)
\]
from $w$ to $u$, obtained by reversing $\gamma_1$. In particular, $\gamma_2$ lies in $G_w$.

Let $s \in V(G_w)$ (resp. $e \in A(G_w)$). By definition, there exists a special path $\gamma_3$ starting at $w$ that contains $s$ (resp. $e$). Concatenating $\gamma_1$ with $\gamma_3$, we obtain a special path starting at $u$ that contains $s$ (resp. $e$). Hence $s \in V(G_u)$ (resp. $e \in A(G_u)$), and therefore $G_w \subset G_u$.

Applying the same argument using $\gamma_2$, we also obtain $G_u \subset G_w$. Hence $G_w = G_u$.
\end{proof}

\begin{prop}\label{P:S_S_positive}
Let $G$ be a directed graph and $u \in V(G)$ such that $V_G^+(u) \neq \emptyset$ and $G_u$ is an ss-digraph. Then every source and every sink in $G_u$ has positive index. Moreover, for any $w \in V(G_u)$, the index of $w$ in $G_u$ equals $\deg_G^+(w)$ if $w$ is a source in $G_u$, and equals $\deg_G^-(w)$ if $w$ is a sink in $G_u$.
\end{prop}
\begin{proof}
The condition $V_G^+(u) \neq \emptyset$ guarantees that $G_u$ is not the null graph, that is,
\[
V(G_u)\neq \emptyset.
\]
By definition, every vertex in $G_u$ belongs to a special path $\gamma$ in $G$ starting at $u$. Hence, $\gamma$ is itself a special path in $G_u$. It follows that every source and every sink in $G_u$ has positive index.

Suppose that $w \in V(G_u)$ is a source. Since $G_u$ is a subgraph of $G$, we have
\[
V_{G_u}^+(w) \subset V_G^+(w).
\]
We prove the reverse inclusion, thereby obtaining $V_{G_u}^+(w)=V_G^+(w)$, which suffices to conclude that the index of $w$ in $G_u$ equals $\deg_G^+(w)$.

Indeed, by the previous proposition, there exists a special path of even length $k$ in $G_u$ from $u$ to $w$:
\[
u_0=u \to u_1 \gets \cdots \gets u_k=w.
\]
In particular, $u_{k-1} \in V_{G_u}^+(w)$. Let $s \in V_G^+(w) \setminus \{u_{k-1}\}$. Then we obtain the following special path in $G$:
\[
u_0=u \to u_1 \gets \cdots \gets u_k=w \to s,
\]
so $s \in V_{G_u}^+(w)$. This shows that $V_G^+(w) \subset V_{G_u}^+(w)$.

Similarly, if $w$ is a sink in $G_u$, by considering a special path of odd length from $u$ to $w$, one shows that
\[
V_G^-(w)=V_{G_u}^-(w).
\]
Consequently, the index of $w$ in $G_u$ equals $\deg_G^-(w)$ in this case. This completes the proof.
\end{proof}

Two directed graphs $G=(V,A)$ and $G'=(V',A')$ are said to be \emph{isomorphic} if there exists a bijection
\[
\varphi : V \to V'
\]
such that, for all $u,v \in V$,
\[
(u,v) \in A \;\Longleftrightarrow\; (\varphi(u),\varphi(v)) \in A'.
\]
In this case, the map $\varphi$ is called an \emph{isomorphism} from $G$ to $G'$.

Given $u \in V(G)$ and $w \in V(G')$, we write
\[
(G_u,u) \cong (G'_w,w)
\]
if there exists an isomorphism
\[
\varphi : V(G_u) \to V(G'_w)
\]
mapping $u$ to $w$.

\begin{example}\label{Ex:ex_special_path_G_alpha}
Let $\hat{G}$ and $\tilde{G}$ be as in Example~\ref{Ex:intro_two_subgraph}. Assume that $d \geq 1$. In Figure~\ref{F:Figura_1}, we present a special path in $\hat{G}$ starting at $x_0^{d+1} dx_3 \in V(\hat{G})$. 

\begin{figure}[h]
\centering
\begin{tikzpicture}[
    >=stealth,
    every arrow/.style={->, thin, shorten >=0.5pt, shorten <=0.5pt},
]
\matrix (m) [matrix of nodes,
             row sep=0.5cm,
             column sep=0.5cm] 
{
|(A)| {$x_0^{d+1} dx_3$} &|(B)| {$x_0^{d} x_1  dx_3$} &|(C)| {$x_0^{d} x_1 dx_2$} & |(D)| {$x_0^{d-1} x_1^2 dx_2$}\\
};

\draw[every arrow] (A) -- (B);
\draw[every arrow] (C) -- (B);
\draw[every arrow] (C) -- (D);
\end{tikzpicture}
\caption{Representation of $\hat{G}_{x_0^{d+1} dx_3}=\hat{G}_{x_0^{d} x_1 dx_2}$}
\label{F:Figura_1}
\end{figure}

Moreover, one verifies that
\begin{align*}
\deg_{\hat{G}}^+(x_0^{d+1} dx_3) &= 1, & \deg_{\hat{G}}^-(x_0^{d} x_1 dx_3) &= 2, \\
\deg_{\hat{G}}^+(x_0^{d} x_1 dx_2) &= 2, & \deg_{\hat{G}}^-(x_0^{d-1} x_1^2 dx_2) &= 1,
\end{align*}
which shows that Figure~\ref{F:Figura_1} represents not only a special path starting at $x_0^{d+1} dx_3$, but in fact the entire directed graph $\hat{G}_{x_0^{d+1} dx_3}$. Note that $\hat{G}_{x_0^{d+1} dx_3}$ is an ss-digraph. Since $x_0^{d} x_1 dx_2$ is a source in $\hat{G}_{x_0^{d+1} dx_3}$, Proposition~\ref{P:G_w_equals_G_u} yields
\[
\hat{G}_{x_0^{d+1} dx_3}=\hat{G}_{x_0^{d} x_1 dx_2}.
\]

In Figure~\ref{F:Figura_2}, we display $\tilde{G}_{x_0^d x_2 dx_3}$, with $d \geq 2$.

\begin{figure}[h]
\centering
\begin{tikzpicture}[
    >=stealth,
    every arrow/.style={->, thin, shorten >=0.5pt, shorten <=0.5pt},
]
\matrix (m) [matrix of nodes,
             row sep=0.5cm,
             column sep=0.5cm] 
{
|(A)| {$x_0^{d} x_2 dx_3$} &|(B)| {$x_0^{d-1} x_1 x_2 dx_3$} &|(C)| {$x_0^{d-1} x_1^2 dx_3$} & |(D)| {$x_0^{d-2} x_1^3 dx_3$}\\
};

\draw[every arrow] (A) -- (B);
\draw[every arrow] (C) -- (B);
\draw[every arrow] (C) -- (D);
\end{tikzpicture}
\caption{Representation of $\tilde{G}_{x_0^{d} x_2 dx_3}$}
\label{F:Figura_2}
\end{figure}

Comparing Figures~\ref{F:Figura_1} and \ref{F:Figura_2}, if $d \geq 2$, it is clear that there exists an isomorphism
\[
\varphi : V(\hat{G}_{x_0^{d+1} dx_3}) \to V(\tilde{G}_{x_0^{d} x_2 dx_3})
\]
such that $\varphi(x_0^{d+1} dx_3)=x_0^{d} x_2 dx_3$, that is,
\[
(\hat{G}_{x_0^{d+1} dx_3},x_0^{d+1} dx_3) \cong (\tilde{G}_{x_0^{d} x_2 dx_3},x_0^{d} x_2 dx_3).
\]
\end{example}

Next, we define certain objects that will be particularly useful only in the case of ss-digraphs. However, in order for the definitions to appear less artificial, we proceed in a general setting and then make some remarks for the case of ss-digraphs.

Given a directed graph $G$, a finite subset $\{w_0,\ldots,w_{k-1}\} \subset V(G)$ is called a \emph{cycle} of length $k$ in $G$ if there exists a special path as in \eqref{E:special_path} with $u_0 = u_k$, such that the vertices $u_0,\ldots,u_{k-1}$ are pairwise distinct and
\[
\{w_0,\ldots,w_{k-1}\} = \{u_0,\ldots,u_{k-1}\}.
\]
Note that, by the definition of a special path, this requires $k \geq 3$.

We say that a special path as in \eqref{E:special_path} \emph{contains a cycle} if there exist indices $0 \leq i < j \leq k$ such that
\[
\{u_i,u_{i+1},\ldots,u_{j-1}\}
\]
is a cycle.

A special path in $G$ as in \eqref{E:special_path} is said to be:
\begin{itemize}
\item \emph{cyclic} if there exists $0 \leq j < k$ such that $u_k = u_j$, and the vertices $u_0,\ldots,u_{k-1}$ are pairwise distinct. In this case, the path contains the cycle $\{u_j,u_{j+1},\ldots,u_{k-1}\}$.

\item \emph{simple} if the vertices $u_0,\ldots,u_k$ are pairwise distinct.

\item \emph{source-extremal} if it is simple and $u_k$ is a source of index $1$. In this case, $k$ must be even.

\item \emph{regular} if $u_j$ is a source of index $2$ for every even $j$ with $0<j<k$, and $u_j$ is a sink of index~$2$ for every odd $j$ with $0<j<k$.

\item \emph{terminal} if it is regular and $u_k$ is a sink of index $1$. In this case, $k$ must be odd.
\end{itemize}

Note that every terminal path is simple. A vertex $u \in V(G)$ is said to be \emph{terminal} if there exists a terminal path starting at $u$. A vertex $u \in V(G)$ is said to be \emph{quasi-terminal} if there exists a regular path starting at $u$ containing a terminal vertex. In particular, every terminal vertex is quasi-terminal. A special path is said to be \emph{admissible} if it is either cyclic or source-extremal. Since every admissible path starting at $u$ is a special path, it appears in $G_u$.

\begin{example}\label{Ex:justifies_ex_article}
In Figure~\ref{F:Figura_3}, we display $\hat{G}_{x_0^{d} x_3 dx_1}$, with $d \geq 3$.

\begin{figure}[h]
\centering
\begin{tikzpicture}[
    >=stealth,
    every arrow/.style={->, thin, shorten >=0.5pt, shorten <=0.5pt},
]
\matrix (m) [matrix of nodes,
             row sep=0.5cm,
             column sep=0.5cm] 
{
& |(A)| {$x_0^d x_3dx_1$} & \\
& |(B)| {$x_0^{d-1} x_1 x_3dx_1$} & \\
|(C)| {$x_0^{d-1} x_1 x_3dx_0$} 
& & 
|(D)| {$x_0^{d-1} x_1 x_2dx_1$} \\
|(E)| {$x_0^{d-2} x_1^2 x_3dx_0$} 
& & 
|(F)| {$x_0^{d-2} x_1^2 x_2dx_1$} \\
& |(G)| {$x_0^{d-2} x_1^2 x_2dx_0$} & \\
& |(H)| {$x_0^{d-3} x_1^3 x_2dx_0$} & \\
};

\draw[every arrow] (A) -- (B);
\draw[every arrow] (C) -- (B);
\draw[every arrow] (D) -- (B);
\draw[every arrow] (C) -- (E);
\draw[every arrow] (D) -- (F);
\draw[every arrow] (G) -- (E);
\draw[every arrow] (G) -- (F);
\draw[every arrow] (G) -- 
node[pos=0.4, left, xshift=-2pt, yshift=1pt, font=\scriptsize] {$(d \geq 3)$} 
(H);
\end{tikzpicture}
\caption{Representation of $\hat{G}_{x_0^{d} x_3 dx_1}$}
\label{F:Figura_3}
\end{figure}

Note that $\hat{G}_{x_0^{d} x_3 dx_1}$ is an ss-digraph, and there is exactly one cycle in $\hat{G}_{x_0^{d} x_3 dx_1}$. The reader can easily verify that there are exactly two cyclic paths in $\hat{G}_{x_0^{d} x_3 dx_1}$ from $x_0^{d} x_3 dx_1$ to $x_0^{d-1} x_1 x_3 dx_1$. Since $x_0^{d} x_3 dx_1$ is the unique source of index $1$ in $\hat{G}_{x_0^{d} x_3 dx_1}$, no source-extremal path in $\hat{G}_{x_0^{d} x_3 dx_1}$ can start at $x_0^{d} x_3 dx_1$. As indicated in Figure~\ref{F:Figura_3}, we have
\[
x_0^{d-3}x_1^3x_2dx_0
\in
V_{\hat{G}_{x_0^{d}x_3dx_1}}^+
(x_0^{d-2}x_1^2x_2dx_0)
\]
only because \(d \geq 3\).

The vertex $x_0^{d-2} x_1^2 x_2 dx_0$ is a terminal source in $\hat{G}_{x_0^{d} x_3 dx_1}$, while $x_0^{d-1} x_1 x_3 dx_0$ and $x_0^{d-1} x_1 x_2 dx_1$ are quasi-terminal.
\end{example}

\begin{remark}\label{R:ss_digraph}
Let $G$ be an ss-digraph.
\begin{enumerate}
\item\label{R:ss_digraph_1} It is easy to see that any cycle in $G$ always has even length. It is well known that this holds more generally for bipartite digraphs. This implies that, in the definition of a cyclic path, the vertices $u_j$ and $u_k$ are either both sources or both sinks. Equivalently, the indices $j$ and $k$ have the same parity.
\item\label{R:ss_digraph_2} Terminal and quasi-terminal vertices are always sources, since by definition there always exist special paths starting at such vertices.
\end{enumerate}
\end{remark}

The next results are stated in the context of Example~\ref{Ex:intro_two_subgraph}. We will assume that $G$ is either $\hat{G}$ or $\tilde{G}$. As in the same example, if $G=\hat{G}$ (resp. $G=\tilde{G}$), we consider $v=v_1$ (resp. $v=v_2$) as in \eqref{E:nil_P3}.

Given elements $\alpha$ and $\beta$ in $B_d$, instead of saying that $\beta$ appears in the expansion of $L_v(\alpha)$ with respect to the basis $B_d$ of $E_d$, we will use the equivalent and more economical notation $\beta \in V_G^+(\alpha)$. For the sake of convenience, given $\Omega \in E_d$, we will abuse notation and denote by $V_G^+(\Omega)$ the set of elements $x^I dx_m \in B_d$ that appear in the expansion of $\Omega$ with respect to $B_d$. 

\begin{lemma}\label{L:cons_two_elements_special}
Let $G$ be either $\hat{G}$ or $\tilde{G}$. Suppose that $\Omega \in E_d$ satisfies $L_v\Omega=0$, $\alpha \in V_G^+(\Omega)$, and $\beta \in V_G^+(\alpha)$. Then there exists $\gamma \in V_G^+(\Omega)$ with $\gamma \neq \alpha$ such that $\beta \in V_G^+(\gamma)$.
\end{lemma}
\begin{proof}
Assume, on the contrary, that for every $\gamma \in V_G^+(\Omega)\setminus\{\alpha\}$ we have $\beta \notin V_G^+(\gamma)$. Let $a \neq 0$ be the coefficient of $\alpha$ in the expansion of $\Omega$ with respect to $B_d$, and write $\Omega = a\alpha + \eta$ for some $\eta \in E_d$. By our hypothesis, $\beta$ does not appear in the expansion of $L_v\eta$ with respect to $B_d$. However, $\beta \in V_G^+(\alpha)$ means that $\beta$ appears in the expansion of $L_v\alpha$ with respect to the same basis. Since $L_v:E_d \to E_d$ is $\mathbb{C}$-linear and $L_v\Omega=0$, we obtain a contradiction.
\end{proof}

\begin{prop}\label{L:dich_ensures_desirable_path}
Let $G$ be either $\hat{G}$ or $\tilde{G}$, and let $\alpha_0 \in B_d$ be such that $G_{\alpha_0}$ is an ss-digraph. Suppose that $\Omega \in E_d$ satisfies $L_v \Omega = 0$, $\alpha_0 \in V_G^+(\Omega)$, and $\alpha_1 \in V_G^+(\alpha_0)$. Then there exists an admissible path
\begin{equation}\label{E:desired_spe_path}
(u_0=\alpha_0,e_1,u_1=\alpha_1,e_2,u_2=\alpha_2,\ldots,e_k,u_k=\alpha_k)
\end{equation}
in $G_{\alpha_0}$ such that $\alpha_j \in V_G^+(\Omega)$ whenever $0 \le j \le k$ is even.
\end{prop}
\begin{proof}
Suppose that no source-extremal path in $G_{\alpha_0}$ satisfies the required condition. We show that there exists a cyclic path satisfying it.

Since $L_v \Omega=0$, $\alpha_0 \in V_G^+(\Omega)$, and $\alpha_1 \in V_G^+(\alpha_0)$, it follows from Lemma~\ref{L:cons_two_elements_special} that there exists $\alpha_2 \in V_G^+(\Omega)$, with $\alpha_2 \neq \alpha_0$, such that $\alpha_1 \in V_G^+(\alpha_2)$. Then
\[
(\alpha_0,e_1,\alpha_1,e_2,\alpha_2)
\]
is a special path in $G_{\alpha_0}$, where $e_1=(\alpha_0,\alpha_1)$ and $e_2=(\alpha_2,\alpha_1)$, and such that $\alpha_j \in V_G^+(\Omega)$ for $j=0,2$.

This path is not cyclic, since $\alpha_0 \neq \alpha_2$. Moreover, by assumption it is not source-extremal. Hence $\alpha_2$ is not a source of index~$1$ in $G_{\alpha_0}$. By Proposition~\ref{P:S_S_positive}, it follows that
\[
\deg_G^{+}(\alpha_2)>1.
\]

Thus, there exists $\alpha_3 \neq \alpha_1$ such that $\alpha_3 \in V_G^+(\alpha_2)$. Since $L_v(\Omega)=0$, $\alpha_2 \in V_G^+(\Omega)$, and $\alpha_3 \in V_G^+(\alpha_2)$, Lemma~\ref{L:cons_two_elements_special} ensures the existence of $\alpha_4 \in V_G^+(\Omega)$ such that
\[
\alpha_3 \in V_G^+(\alpha_4)
\quad\text{and}\quad
\alpha_4 \neq \alpha_2.
\]

Then
\[
(\alpha_0,e_1,\alpha_1,e_2,\alpha_2,e_3,\alpha_3,e_4,\alpha_4)
\]
is a special path in $G_{\alpha_0}$, where $e_3=(\alpha_2,\alpha_3)$ and $e_4=(\alpha_4,\alpha_3)$, and such that $\alpha_j \in V_G^+(\Omega)$ for $j=0,2,4$.

If $\alpha_4=\alpha_0$, then this path is cyclic. Otherwise, since it is not source-extremal, Proposition~\ref{P:S_S_positive} implies that
\[
\deg_G^{+}(\alpha_4)>1.
\]

Continuing inductively, we obtain a special path
\[
(\alpha_0,e_1,\alpha_1,e_2,\alpha_2,e_3,\alpha_3,e_4,\alpha_4,\ldots,e_{2l},\alpha_{2l}),
\]
with $\alpha_j \in V_G^+(\Omega)$ for every even $j \in \{0,1,\ldots,2l\}$.

If this path is not cyclic, then it is also not source-extremal. Hence, by Proposition~\ref{P:S_S_positive},
\[
\deg_G^{+}(\alpha_{2l})>1.
\]

Proceeding as before, we obtain vertices $\alpha_{2l+1} \in V_G^+(\alpha_{2l})$ and $\alpha_{2l+2} \in V_G^+(\Omega)$ such that
\[
\alpha_{2l+1} \in V_G^+(\alpha_{2l+2}),
\]
with $\alpha_{2l+1} \neq \alpha_{2l-1}$ and $\alpha_{2l+2} \neq \alpha_{2l}$.

Since $G_{\alpha_0}$ is finite, and every cycle in an ss-digraph has even length (see Remark~\ref{R:ss_digraph}~\eqref{R:ss_digraph_1}), eventually either
\[
\alpha_{2l+1}=\alpha_{2j+1}
\quad\text{for some } j<l-1,
\]
or
\[
\alpha_{2l+2}=\alpha_{2j}
\quad\text{for some } j<l.
\]

In the first case, we take $k=2l+1$ and $e_{2l+1}=(\alpha_{2l},\alpha_{2l+1})$. In the second case, we take $k=2l+2$, together with
\[
e_{2l+1}=(\alpha_{2l},\alpha_{2l+1})
\quad\text{and}\quad
e_{2l+2}=(\alpha_{2l+2},\alpha_{2l+1}).
\]

In either case, we obtain a cyclic path as in \eqref{E:desired_spe_path} satisfying
\[
\alpha_j \in V_G^+(\Omega)
\]
whenever $0\leq j\leq k$ is even. This completes the proof.
\end{proof}

\begin{cor}\label{C:desi_withou_cycles_index_1_source}
Let $G$ be either $G=\hat{G}$ or $G=\tilde{G}$, and let $\alpha_0 \in B_d$ be such that $V_G^+(\alpha_0) \neq \emptyset$ and $G_{\alpha_0}$ is an ss-digraph. Assume that there are no source-extremal paths in $G_{\alpha_0}$ starting at $\alpha_0$ and that $G_{\alpha_0}$ contains no cycles. Then, for any $\Omega \in E_d$ satisfying $L_v \Omega=0$, $\alpha_0$ does not appear in the decomposition of $\Omega$ with respect to $B_d$.
\end{cor}
\begin{proof}
First observe that the hypotheses ensure that there are no cyclic or source-extremal paths in $G_{\alpha_0}$ starting at $\alpha_0$, and therefore no admissible path in $G_{\alpha_0}$ starting at $\alpha_0$. Let $\alpha_1 \in V_G^+(\alpha_0)$. Assume that there exists $\Omega \in E_d$, with $L_v \Omega=0$, such that $\alpha_0$ appears in the decomposition of $\Omega$ with respect to $B_d$. By Proposition~\ref{L:dich_ensures_desirable_path}, there exists an admissible path in $G_{\alpha_0}$ as in \eqref{E:desired_spe_path}, which is a contradiction.
\end{proof}

\begin{prop}\label{P:origin_ter_or_quasi-terminal}
Let $G$ be either $\hat{G}$ or $\tilde{G}$, and let $\alpha_0 \in B_d$ be such that $G_{\alpha_0}$ is an ss-digraph. If $\alpha_0$ is terminal or quasi-terminal in $G_{\alpha_0}$, then for any $\Omega \in E_d$ satisfying $L_v \Omega=0$, $\alpha_0$ does not appear in the decomposition of $\Omega$ with respect to $B_d$.
\end{prop}
\begin{proof}
Assume, on the contrary, that there exists $\Omega \in E_d$ satisfying $L_v \Omega=0$, with $\alpha_0 \in V_G^+(\Omega)$.

Suppose first that $\alpha_0$ is terminal. By definition, there exists a regular path
\[
\alpha_0 \to \alpha_1 \gets \cdots \gets \alpha_{k-1} \to \alpha_k
\]
starting at $\alpha_0$, where $\alpha_k$ is a sink of index~$1$. Recall that, in particular, $k$ is odd.

Since $\alpha_1 \in V_G^+(\alpha_0)$, Lemma~\ref{L:cons_two_elements_special} yields the existence of $\gamma \in V_G^+(\Omega)$, with $\gamma \neq \alpha_0$, such that $\alpha_1 \in V_G^+(\gamma)$. If $k=1$, we obtain a contradiction, since in this case $\alpha_1$ is a sink of index~$1$.

Assume now that $k>1$. Since $\alpha_1$ is a sink of index~$2$ by the definition of a regular path, we must have $\gamma=\alpha_2$. Hence $\alpha_j \in V_G^+(\Omega)$ for $j=0,2$.

Continuing inductively, we conclude that $\alpha_j \in V_G^+(\Omega)$ for every even $j$ with $0 \leq j < k$. In particular,
\[
\alpha_{k-1} \in V_G^+(\Omega).
\]

Since $\alpha_k \in V_G^+(\alpha_{k-1})$, Lemma~\ref{L:cons_two_elements_special} again yields the existence of $\gamma \in V_G^+(\Omega)$, with $\gamma \neq \alpha_{k-1}$, such that
\[
\alpha_k \in V_G^+(\gamma),
\]
which contradicts the fact that $\alpha_k$ is a sink of index~$1$.

Suppose now that $\alpha_0$ is quasi-terminal. By definition, there exists a regular path
\[
\alpha_0 \to \alpha_1 \gets \alpha_2 \to \cdots
\]
starting at $\alpha_0$ and containing a terminal vertex $\beta$. Since terminal and quasi-terminal vertices are sources by Remark~\ref{R:ss_digraph}~\eqref{R:ss_digraph_2}, it follows from Proposition~\ref{P:G_w_equals_G_u} that $\beta=\alpha_m$ for some even $m$.

Proceeding as in the terminal case, we conclude that $\alpha_j \in V_G^+(\Omega)$ for every even $j$. In particular,
\[
\beta \in V_G^+(\Omega).
\]

By Proposition~\ref{P:G_w_equals_G_u}, we have
\[
G_{\alpha_0}=G_\beta.
\]

In particular, $G_\beta$ is an ss-digraph. Since $\beta \in V_G^+(\Omega)$, $L_v \Omega=0$, and $\beta$ is terminal in $G_\beta$, we obtain a contradiction exactly as in the terminal case. This completes the proof.
\end{proof}

\begin{example}\label{Ex:return_ex_special_path_G_alpha}
Returning to Example~\ref{Ex:ex_special_path_G_alpha}, where we displayed $\hat{G}_{x_0^{d+1} dx_3}$ in Figure~\ref{F:Figura_1}, we observe that there are no source-extremal paths starting at $x_0^{d+1} dx_3$, since $x_0^{d+1} dx_3$ is the unique source of index $1$ in $\hat{G}_{x_0^{d+1} dx_3}$. Moreover, $\hat{G}_{x_0^{d+1} dx_3}$ contains no cycles. Hence, if $d \geq 1$, by Corollary~\ref{C:desi_withou_cycles_index_1_source}, for every $\Omega \in E_d$ satisfying $L_v\Omega=0$, where $v=v_1$ is as in \eqref{E:nil_P3}, $x_0^{d+1} dx_3$ does not appear in the expansion of $\Omega$ with respect to $B_d$.

Note also that $x_0^{d+1} dx_3$ is a terminal vertex in $\hat{G}_{x_0^{d+1} dx_3}$. Therefore, Proposition~\ref{P:origin_ter_or_quasi-terminal} could also be applied to obtain the same conclusion.
\end{example}

\begin{cor}\label{C:cor_desri_ter_quasi}
Let $G$ be either $\hat{G}$ or $\tilde{G}$, and let $\alpha_0 \in B_d$ be such that $G_{\alpha_0}$ is an ss-digraph. Let $\alpha_1 \in V_G^+(\alpha_0)$. If every admissible path as in \eqref{E:special_path} with $u_0=\alpha_0$ and $u_1=\alpha_1$ contains a terminal or quasi-terminal vertex in $G_{\alpha_0}$, then, for any $\Omega \in E_d$ satisfying $L_v \Omega=0$, $\alpha_0$ does not appear in the decomposition of $\Omega$ with respect to $B_d$.
\end{cor}
\begin{proof}
Assume, on the contrary, that there exists $\Omega \in E_d$ satisfying $L_v \Omega=0$, with $\alpha_0 \in V_G^+(\Omega)$. By Proposition~\ref{L:dich_ensures_desirable_path}, there exists an admissible path
\begin{equation}\label{E:cor_desri_ter_quasi}
(\alpha_0,e_1,\alpha_1,e_2,\alpha_2,\ldots,e_k,\alpha_k)
\end{equation}
in $G_{\alpha_0}$ such that $\alpha_j \in V_G^+(\Omega)$ whenever $0\leq j\leq k$ is even.

By hypothesis, this admissible path contains a terminal or quasi-terminal vertex $\beta$ in $G_{\alpha_0}$. Since terminal and quasi-terminal vertices in $G_{\alpha_0}$ are sources by Remark~\ref{R:ss_digraph}~\eqref{R:ss_digraph_2}, Proposition~\ref{P:G_w_equals_G_u} implies that $\beta=\alpha_m$ for some even $m$. In particular,
\[
\beta \in V_G^+(\Omega).
\]

By Proposition~\ref{P:G_w_equals_G_u}, we have
\[
G_{\alpha_0}=G_\beta.
\]

Hence $G_\beta$ is an ss-digraph and $\beta$ is terminal or quasi-terminal in $G_\beta$. Proposition~\ref{P:origin_ter_or_quasi-terminal} then implies that $\beta$ does not appear in the decomposition of $\Omega$ with respect to $B_d$, contradicting the fact that
\[
\beta \in V_G^+(\Omega).
\]
This completes the proof.
\end{proof}

\begin{cor}\label{C:no_source_1_cycles_ter_quasi-ter}
Let $G$ be either $\hat{G}$ or $\tilde{G}$, and let $\alpha_0 \in B_d$ be such that $V_G^+(\alpha_0) \neq \emptyset$ and $G_{\alpha_0}$ is an ss-digraph. Assume that there are no source-extremal paths in $G_{\alpha_0}$ starting at $\alpha_0$ and every cycle in $G_{\alpha_0}$ has a terminal or quasi-terminal vertex in $G_{\alpha_0}$. Then, for any $\Omega \in E_d$ such that $L_v \Omega=0$, $\alpha_0$ does not appear in the decomposition of $\Omega$ with respect to $B_d$.
\end{cor}
\begin{proof}
Let $\alpha_1 \in V_G^+(\alpha_0)$. By hypothesis, every admissible path as in \eqref{E:special_path} with $u_0=\alpha_0$ and $u_1=\alpha_1$ cannot be source-extremal. Therefore, every such admissible path is cyclic.

Hence, every such path contains a cycle, and therefore, by hypothesis, contains a terminal or quasi-terminal vertex in $G_{\alpha_0}$. The conclusion now follows from Corollary~\ref{C:cor_desri_ter_quasi}.
\end{proof}

\begin{example}\label{Ex:sec_example_table}
Returning to Example~\ref{Ex:justifies_ex_article}, where we displayed $\hat{G}_{x_0^{d} x_3 dx_1}$ in Figure~\ref{F:Figura_3}, we saw that there are no source-extremal paths in $\hat{G}_{x_0^{d} x_3 dx_1}$ starting at $x_0^{d} x_3 dx_1$, and that the unique cycle in $\hat{G}_{x_0^{d} x_3 dx_1}$ contains both terminal and quasi-terminal vertices in $\hat{G}_{x_0^{d} x_3 dx_1}$. Hence, by Corollary~\ref{C:no_source_1_cycles_ter_quasi-ter}, if $d \geq 3$, for every $\Omega \in E_d$ satisfying $L_v\Omega=0$, where $v=v_1$ is as in \eqref{E:nil_P3}, $x_0^{d} x_3 dx_1$ does not appear in the expansion of $\Omega$ with respect to $B_d$.
\end{example}

\begin{example}\label{Ex:third_example_table}
In Figure~\ref{F:Figura_4}, we display $\tilde{G}_{x_0^{d} x_2 dx_2}$, with $d \geq 4$. There are no source-extremal paths in $\tilde{G}_{x_0^{d} x_2 dx_2}$ starting at $x_0^{d} x_2 dx_2$, since $x_0^d x_2 dx_2$ is the unique source of index $1$. Note that $\tilde{G}_{x_0^d x_2 dx_2}$ contains exactly seven cycles, and each of these cycles contains at least one of the quasi-terminal vertices
\[
x_0^{d-2}x_1^3 dx_1
\quad \text{and} \quad
x_0^{d-2}x_1^2 x_2 dx_0.
\]

Hence, by Corollary~\ref{C:no_source_1_cycles_ter_quasi-ter}, if $d \geq 4$, for every $\Omega \in E_d$ satisfying $L_v\Omega=0$, where $v=v_2$ is as in \eqref{E:nil_P3}, $x_0^{d} x_2 dx_2$ does not appear in the expansion of $\Omega$ with respect to $B_d$.

\begin{figure}[h]
\centering
\begin{tikzpicture}[
    >=stealth,
    every arrow/.style={->, thin, shorten >=0.5pt, shorten <=0.5pt},
]
\matrix (m) [matrix of nodes,
             row sep=0.5cm,
             column sep=0.5cm] 
{
& |(A)| {$x_0^{d} x_2 dx_2$} & &  &\\
&|(B)| {$x_0^{d-1} x_1 x_2 dx_2$} &  & |(C)| {$x_0^{d-1} x_2^2 dx_1$} & \\
|(D)| {$x_0^{d-1} x_1^2 dx_2$} & & |(E)| {$x_0^{d-1}x_1 x_2 dx_1$} & &|(F)| {$x_0^{d-1} x_2^2 dx_0$} \\
|(G)| {$x_0^{d-2} x_1^3 dx_2$} & & |(H)| {$x_0^{d-2} x_1^2 x_2 dx_1$} & &|(I)| {$x_0^{d-2} x_1 x_2^2 dx_0$} \\
&|(J)| {$x_0^{d-2} x_1^3 dx_1$} & & |(K)| {$x_0^{d-2} x_1^2 x_2 dx_0$}& \\
&|(L)| {$x_0^{d-3} x_1^4 dx_1$}  & & |(M)| {$x_0^{d-3} x_1^3 x_2 dx_0$}& \\
& & |(N)| {$x_0^{d-3} x_1^4 dx_0$}& & \\
& & |(O)| {$x_0^{d-4} x_1^5dx_0$}& & \\
};

\draw[every arrow] (A) -- (B);
\draw[every arrow] (D) -- (B);
\draw[every arrow] (E) -- (B);
\draw[every arrow] (E) -- (C);
\draw[every arrow] (F) -- (C);
\draw[every arrow] (D) -- (G);
\draw[every arrow] (E) -- (H);
\draw[every arrow] (F) -- (I);
\draw[every arrow] (J) -- (G);
\draw[every arrow] (J) -- (H);
\draw[every arrow] (N) -- 
node[pos=0.4, left, xshift=-2pt, yshift=-1pt, font=\scriptsize] {$(d \geq 4)$} 
(O);
\draw[every arrow] (K) -- (H);
\draw[every arrow] (K) -- (I);
\draw[every arrow] (J) -- 
node[pos=0.4, left, xshift=-2pt, yshift=-1pt, font=\scriptsize] {$(d \geq 3)$} 
(L);
\draw[every arrow] (K) -- 
node[pos=0.4, right, xshift=2pt, yshift=-1pt, font=\scriptsize] {$(d \geq 3)$} 
(M);
\draw[every arrow] (N) -- 
node[pos=0.4, left, xshift=-2pt, yshift=-5pt, font=\scriptsize] {$(d \geq 3)$} 
(L);
\draw[every arrow] (N) -- 
node[pos=0.4, right, xshift=0pt, yshift=-5pt, font=\scriptsize] {$(d \geq 3)$} 
(M);
\end{tikzpicture}
\caption{Representation of $\tilde{G}_{x_0^{d} x_2 dx_2}$}
\label{F:Figura_4}
\end{figure}
\end{example}

\begin{remark}\label{R:Remark_equiv_results}
Let $G,H \in \{\hat{G},\tilde{G}\}$ and $\alpha,\beta \in B_d$. Assume that
\[
(G_\alpha,\alpha) \cong (H_\beta,\beta).
\]
Suppose that $G_\alpha$ satisfies the hypotheses of either Corollary~\ref{C:desi_withou_cycles_index_1_source}, Proposition~\ref{P:origin_ter_or_quasi-terminal}, Corollary~\ref{C:cor_desri_ter_quasi}, or Corollary~\ref{C:no_source_1_cycles_ter_quasi-ter}. Then the same holds for $H_\beta$. Hence, for every $\Omega \in E_d$ satisfying $L_v\Omega=0$, $\beta$ does not appear in the expansion of $\Omega$ with respect to $B_d$, where $v=v_1$ (resp. $v=v_2$) as in \eqref{E:nil_P3} if $H=\hat{G}$ (resp. $H=\tilde{G}$).

However, note that we may have $G_\alpha=G_\beta$ without having
\[
(G_\alpha,\alpha) \cong (G_\beta,\beta).
\]
This occurs for the directed graph $\hat{G}_{x_0^{d+1} dx_3}$ considered in Example~\ref{Ex:ex_special_path_G_alpha}. Indeed, we saw there that
\[
\hat{G}_{x_0^{d+1} dx_3}=\hat{G}_{x_0^d x_1 dx_2}.
\]
Since $x_0^{d+1}dx_3$ and $x_0^d x_1dx_2$ are sources of indices $1$ and $2$ in $\hat{G}_{x_0^{d+1}dx_3}$ and $\hat{G}_{x_0^d x_1dx_2}$, respectively, we do not have
\[
(\hat{G}_{x_0^{d+1}dx_3},x_0^{d+1}dx_3)
\cong
(\hat{G}_{x_0^d x_1dx_2},x_0^d x_1dx_2).
\]

Note also that, by Example~\ref{Ex:return_ex_special_path_G_alpha}, there are no source-extremal paths in \(\hat{G}_{x_0^{d+1}dx_3}\) starting at \(x_0^{d+1}dx_3\). On the other hand, Figure~\ref{F:Figura_1} shows that there exists a source-extremal path in \(\hat{G}_{x_0^d x_1dx_2}\) starting at \(x_0^d x_1dx_2\) and ending at \(x_0^{d+1}dx_3\).
\end{remark}

\begin{remark}\label{R:possibly_other_results}
We observe that whenever Proposition~\ref{P:origin_ter_or_quasi-terminal} is applicable, Corollary~\ref{C:cor_desri_ter_quasi} is applicable as well. Indeed, if $\alpha_0$ is terminal or quasi-terminal in $G_{\alpha_0}$, then every admissible path starting at $\alpha_0$ necessarily contains a terminal or quasi-terminal vertex, namely $\alpha_0$ itself.

Similarly, Corollary~\ref{C:no_source_1_cycles_ter_quasi-ter}, which allows us to consider only cycles instead of cyclic paths, can always be replaced by Corollary~\ref{C:cor_desri_ter_quasi}.
\end{remark}

\begin{proof}[Proof of Proposition~\ref{L:alg_mult_at_least_2}]
Set $d=\deg(\F) \geq 2$. Let $E_d$, $B_d$, $\hat{G}=\hat{G}(d)$ and $\tilde{G}=\tilde{G}(d)$ be as in Example~\ref{Ex:intro_two_subgraph}. Recall that $\F$ is defined in homogeneous coordinates by some $\Omega \in E_d$. Thus, we can write in an unique way
\[
\Omega=\sum_{m=0}^3 \sum_{|I|=d+1} a(I;m)x^Idx_m,
\] 
for some complex numbers $a(I;m)$.

By writing the expression of $\Omega$ in the affine chart $(1:z:y:x)$, we see that $p=(1:0:0:0)$ is a singular point of $\F$ means that $a(I;m)=0$, for $I=(d+1,0,0,0)$ and $m=1,2,3$, and its algebraic multiplicity is at least $2$ if additionally $a(I;m)=0$, for $I=(i=d,j,k,l)$ and $m=1,2,3$. This amounts to verifying that $12$ of the coefficients $a(I;m)$ vanish. Define $C$ as the subset of $B_d$ consisting of the elements $x^I dx_m$ corresponding to each of these $12$ coefficients.

Since $v$ is nilpotent, it follows that $L_v\Omega=0$ \cite[Lemma~5.1]{RaRuJo2022}. If $d \geq 4$, the result will follow only from this relation.

We shall give a proof with the aid of Tables~\ref{Tab:1} and \ref{Tab:2} below.

\begin{table}[H]
\centering
\small
\setlength{\tabcolsep}{6pt}
\renewcommand{\arraystretch}{1.2}
\resizebox{\textwidth}{!}{
\begin{tabular}{>{\centering\arraybackslash}m{2.6cm}
                >{\centering\arraybackslash}m{3.2cm}
                >{\centering\arraybackslash}m{4.6cm}
                >{\centering\arraybackslash}m{3.2cm}}
\toprule
$\alpha = x^I dx_m$ 
& Number of cycles in $\hat{G}_\alpha$ 
& Reason for $a(I;m)=0$
& Comment \\
\midrule

$x_0^{d+1} dx_1$ & 0 
& $(\hat{G}_\alpha,\alpha) \cong (\hat{G}_{x_0^{d+1} dx_3}, x_0^{d+1} dx_3)$
&  \\

$x_0^{d+1} dx_2$ & 0 
& Corollary~\ref{C:desi_withou_cycles_index_1_source} or Proposition~\ref{P:origin_ter_or_quasi-terminal}
&  \\

$x_0^{d+1} dx_3$ & 0 
&  Corollary~\ref{C:desi_withou_cycles_index_1_source} or Proposition~\ref{P:origin_ter_or_quasi-terminal}
& See Example~\ref{Ex:return_ex_special_path_G_alpha} \\

$x_0^{d} x_1 dx_1$ & 0 
& $(\hat{G}_\alpha,\alpha) \cong (\hat{G}_{x_0^{d+1} dx_3}, x_0^{d+1} dx_3)$
& \\

$x_0^{d} x_1 dx_2$ & 0 
& Proposition~\ref{P:origin_ter_or_quasi-terminal}
& $\hat{G}_\alpha = \hat{G}_{x_0^{d+1} dx_3}$ \\

$x_0^{d} x_1 dx_3$ & 0 
& $(\hat{G}_\alpha,\alpha) \cong (\hat{G}_{x_0^{d+1} dx_3}, x_0^{d+1} dx_3)$
& \\

$x_0^{d} x_2 dx_1$ & 1 
& Proposition~\ref{P:origin_ter_or_quasi-terminal} or Corollary~\ref{C:no_source_1_cycles_ter_quasi-ter}
&  \\

$x_0^{d} x_2 dx_2$ & 0 
& Corollary~\ref{C:desi_withou_cycles_index_1_source} or Proposition~\ref{P:origin_ter_or_quasi-terminal}
&  \\

$x_0^{d} x_2 dx_3$ & 1 
& $(\hat{G}_\alpha,\alpha) \cong (\hat{G}_{x_0^{d} x_2 dx_1}, x_0^{d} x_2 dx_1)$
&  \\

$x_0^{d} x_3 dx_1$ & 1 
& Corollary~\ref{C:no_source_1_cycles_ter_quasi-ter}
&  See Example~\ref{Ex:sec_example_table}, $d \geq 3$ \\

$x_0^{d} x_3 dx_2$ & 1 
& $(\hat{G}_\alpha,\alpha) \cong (\hat{G}_{x_0^{d} x_2 dx_3}, x_0^{d} x_2 dx_3)$
& $\hat{G}_\alpha = \hat{G}_{x_0^{d} x_2 dx_3}$ \\

$x_0^{d} x_3 dx_3$ & 1 
& $(\hat{G}_\alpha,\alpha) \cong (\hat{G}_{x_0^{d} x_3 dx_1}, x_0^{d} x_3 dx_1)$
&  $ d \geq 3$ \\

\bottomrule
\end{tabular}
}
\caption{Justification that $a(I;m)=0$ when $x^I dx_m \in C$ and $v=v_1$.}
\label{Tab:1}
\end{table}

\begin{table}[H]
\centering
\small
\setlength{\tabcolsep}{6pt}
\renewcommand{\arraystretch}{1.2}
\resizebox{\textwidth}{!}{
\begin{tabular}{>{\centering\arraybackslash}m{2.6cm}
                >{\centering\arraybackslash}m{3.2cm}
                >{\centering\arraybackslash}m{4.6cm}
                >{\centering\arraybackslash}m{3.2cm}}
\toprule
$\alpha = x^I dx_m$ 
& Number of cycles in $\tilde{G}_\alpha$ 
& Reason for $a(I;m)=0$
& Comment \\
\midrule

$x_0^{d+1} dx_1$ & 0 
& Corollary~\ref{C:desi_withou_cycles_index_1_source} or Proposition~\ref{P:origin_ter_or_quasi-terminal}
&  \\

$x_0^{d+1} dx_2$ & 1 
& Corollary~\ref{C:no_source_1_cycles_ter_quasi-ter}
&  \\

$x_0^{d+1} dx_3$ & 0 
& Corollary~\ref{C:desi_withou_cycles_index_1_source} or Proposition~\ref{P:origin_ter_or_quasi-terminal}
&  \\

$x_0^{d} x_1 dx_1$ & 1 
& Proposition~\ref{P:origin_ter_or_quasi-terminal} or Corollary~\ref{C:cor_desri_ter_quasi}
& $\tilde{G}_\alpha = \tilde{G}_{x_0^{d+1} dx_2}$ \\

$x_0^{d} x_1 dx_2$ & 3 
& Corollary~\ref{C:no_source_1_cycles_ter_quasi-ter}
& $d \geq 3$ \\

$x_0^{d} x_1 dx_3$ & 0 
& $(\tilde{G}_\alpha,\alpha) \cong (\hat{G}_{x_0^{d+1} dx_2},x_0^{d+1} dx_2)$
&  \\

$x_0^{d} x_2 dx_1$ & 3 
& Corollary~\ref{C:no_source_1_cycles_ter_quasi-ter}
& $\tilde{G}_\alpha = \tilde{G}_{x_0^{d} x_1 dx_2}$ \\

$x_0^{d} x_2 dx_2$ & 7 
& Corollary~\ref{C:no_source_1_cycles_ter_quasi-ter}
& See Example~\ref{Ex:third_example_table}, $d \geq 4$ \\

$x_0^{d} x_2 dx_3$ & 0 
& $(\tilde{G}_\alpha,\alpha) \cong (\hat{G}_{x_0^{d+1} dx_3}, x_0^{d+1} dx_3)$
& See Example~\ref{Ex:ex_special_path_G_alpha} \\

$x_0^{d} x_3 dx_1$ & 0 
& $(\tilde{G}_\alpha,\alpha) \cong (\tilde{G}_{x_0^{d+1} dx_1}, x_0^{d+1} dx_1)$
&  \\

$x_0^{d} x_3 dx_2$ & 1 
& $(\tilde{G}_\alpha,\alpha) \cong (\tilde{G}_{x_0^{d+1} dx_2}, x_0^{d+1} dx_2)$
& $d \geq 3$ \\

$x_0^{d} x_3 dx_3$ & 0 
& $(\tilde{G}_\alpha,\alpha) \cong (\tilde{G}_{x_0^{d+1} dx_3}, x_0^{d+1} dx_3)$
&  \\

\bottomrule
\end{tabular}
}
\caption{Justification that $a(I;m)=0$ when $x^I dx_m \in C$ and $v=v_2$.}
\label{Tab:2}
\end{table}

Once the directed graphs $\hat{G}_\alpha$ and $\tilde{G}_\alpha$ are obtained for $\alpha \in C$ (see Remark~\ref{R:constructing_G_u}), we observe that all of them are ss-digraphs. As shown in Example~\ref{Ex:return_ex_special_path_G_alpha}, both Corollary~\ref{C:desi_withou_cycles_index_1_source} and Proposition~\ref{P:origin_ter_or_quasi-terminal} imply that
\[
a(d+1,0,0,0;3)=0
\]
when $v=v_1$, as indicated in row $3$ of Table~\ref{Tab:1}. Similarly, the reader may verify that the same results apply to show that
\[
a(d+1,0,0,0;2)=0
\]
for $v=v_1$, as indicated in row $2$ of Table~\ref{Tab:1}.

In Tables~\ref{Tab:1} and \ref{Tab:2}, in each case, the results that may be used to prove that $a(I;m)=0$ are not necessarily restricted to those listed in the tables; see Remark~\ref{R:possibly_other_results}. 

Moreover, whenever the reason for $a(I;m)=0$ is of the form
\[
(\hat{G}_\alpha,\alpha) \cong (H_\beta,\beta)
\quad \text{or} \quad
(\tilde{G}_\alpha,\alpha) \cong (H_\beta,\beta),
\]
where $H \in \{\hat{G},\tilde{G}\}$ and $\beta \in C$, we are implicitly using Remark~\ref{R:Remark_equiv_results} together with the fact that $H_\beta$ satisfies the hypotheses of one of the results mentioned in that remark. We chose to present the justification in this way, rather than simply describing the results used as in other rows, because, since we are not displaying all $24$ directed graphs, it provides additional information to the reader.

As observed in Example~\ref{Ex:ex_special_path_G_alpha}, the structure of $\hat{G}_{x_0^{d+1} dx_3}$ does not depend on $d\geq 2$ (in fact, it is clear there that the same holds for every $d\geq 1$). More precisely, we have
\[
(\hat{G}(d_1)_{x_0^{d_1+1} dx_3},x_0^{d_1+1} dx_3)
\cong
(\hat{G}(d_2)_{x_0^{d_2+1} dx_3},x_0^{d_2+1} dx_3)
\]
for all $d_1,d_2\geq 2$.

A similar phenomenon occurs for $\hat{G}_\alpha$ and $\tilde{G}_\alpha$ when $\alpha \in C$, except for a few cases indicated in Tables~\ref{Tab:1} and \ref{Tab:2}. More precisely, for
\[
\hat{G}_{x_0^d x_3 dx_1},\;
\hat{G}_{x_0^d x_3 dx_3},\;
\tilde{G}_{x_0^d x_1 dx_2},
\text{ and }
\tilde{G}_{x_0^d x_3 dx_2},
\]
one must assume $d\geq 3$, while for
\[
\tilde{G}_{x_0^d x_2 dx_2},
\]
one must assume $d\geq 4$ (see Examples~\ref{Ex:justifies_ex_article} and \ref{Ex:third_example_table} for $\hat{G}_{x_0^d x_3 dx_1}$ and $\tilde{G}_{x_0^d x_2 dx_2}$, respectively).

We thus conclude that the $12$ coefficients vanish for $v=v_1$ and $d\geq 3$, and for $v=v_2$ and $d\geq 4$. For the remaining cases, we determine an explicit form for $\Omega$. This is done by solving the system of linear equations
\begin{equation}\label{E:System_linear}
L_v \Omega=0, \qquad i_R \Omega=0, \qquad i_v \Omega=0.
\end{equation}
Note that, depending on $d$, this system has
\[
4\binom{d+4}{3}
=
\frac{2(d+4)(d+3)(d+2)}{3}
\]
variables, namely the coefficients of $\Omega \in E_d$.

From this point on, both here and in the proof of Theorem~\ref{T:structure}, whenever we derive an explicit expression for $\Omega$ by solving \eqref{E:System_linear}, we use the software \textsc{Maple}. We have uploaded to arXiv, as ancillary files, the \textsc{Maple} code used to perform these computations. This occurs on five occasions, including the three remaining cases
\[
d=2 \text{ and } v=v_1,\qquad
d=2 \text{ and } v=v_2,\qquad
d=3 \text{ and } v=v_2.
\]
After obtaining an explicit expression for $\Omega$, we verify that the $12$ coefficients vanish, thereby completing the proof of the proposition.
\end{proof}

\begin{proof}[Proof of Lemma~\ref{L:divides_in_one_point}]
Assume on the contrary that $v$ does not locally divide $\F$ at $p$. If $v=v_1$, in the affine chart $(1:z:y:x)$ it is represented by
\[
-V=-xz \frac{\partial}{\partial x}+(x-yz) \frac{\partial}{\partial y}-z^2 \frac{\partial}{\partial z}=x \frac{\partial}{\partial y}-zR_3,
\]
where $R_3:=x\partial/\partial x+y\partial/\partial y+z\partial/\partial z$.

For a polynomial $1$-form $\omega$ defining $\F$ in the same affine chart, since $\TF$ is locally free at $p$, it follows from Proposition~\ref{P:known_fact_2} that
\[
\omega=i_Z i_W (dx \wedge dy \wedge dz),
\]
where $Z,W \in \mathcal{X}(\mathbb{C}^3,0)$. By Lemma~\ref{R:existence_functions}, we have that $V=fZ+gW$, for $f,g \in \mathcal{O}_{\mathbb{C}^3,0}$. Since we are assuming that $v$ does not locally divide $\F$ at $p$, by the particular case of Lemma~\ref{L:divides_f_and_g} we get that $f(0)=g(0)=0$. Hence
\[
\sing (V)=\{f=g=0\}=\{x=z=0\},
\]
and we can write $f=xf_1+zf_2$ and $g=xg_1+zg_2$, for $f_1,f_2,g_1,g_2 \in \mathcal{O}_{\mathbb{C}^3,0}$. Then 
\begin{equation}\label{E:eq_for_V}
V=x X+z Y,
\end{equation}
where $X:=f_1 Z+g_1 W$ and $Y:=f_2 Z+g_2 W$. Writing $X=\sum_{i=0}^{\infty}\hat{X}_i$ and $Y=\sum_{i=0}^{\infty}\hat{Y}_i$, from \eqref{E:eq_for_V} we obtain that
\[
x \hat{X}_0 + z \hat{Y}_0=-x \frac{\partial}{\partial y},
\]
which implies that $\hat{Y}_0=0$ and $\hat{X}_0=-\partial/\partial y$. It also follows from \eqref{E:eq_for_V} that $x\hat{X}_1+z\hat{Y}_1=zR_3$, which implies that
\[
\hat{Y}_1=(x-ax)\frac{\partial}{\partial x}+(y-bx)\frac{\partial}{\partial y}+(z-cx)\frac{\partial}{\partial z},
\]
for some complex numbers $a,b,c$.

Since $i_{\hat{X}_0} i_{\hat{Y}_1} (dx \wedge dy \wedge dz) \neq 0$, we conclude that $\eta=i_X i_Y (dx \wedge dy \wedge dz)$ also defines $\F$ around $0 \in \C^3$ and the singularity $p$ of $\F$ has algebraic multiplicity $1$. In a very similar way, we conclude the same if $v=v_2$. On the other hand, we know from Proposition~\ref{L:alg_mult_at_least_2} that $p$ is a singular point of $\F$ with algebraic multiplicity at least $2$, obtaining a contradiction. This finishes the proof.
\end{proof}

For convenience, we state again Theorem~\ref{T:structure}.

\begin{thmC}
Let $\F$ be a codimension one holomorphic foliation on $\mathbb P^3$ with locally free tangent sheaf. If there exists a nonzero holomorphic vector field on $\mathbb P^3$ tangent to $\F$, then $\TF$ splits.
\end{thmC}
\begin{proof}
Let $v \neq 0$ be a holomorphic vector field on $\mathbb P^3$ tangent to $\F$. Recall that it can be represented by a class of homogeneous linear polynomial vector field $v$ on $\mathbb C^4$. Note that $\sing(v)$ is the image, under the canonical projection $\Pi: \mathbb{C}^4 \setminus \{0\} \to \mathbb{P}^3$, of the set
\[
\{ p \in \mathbb C^4 \setminus \{0\} : v(p) \wedge R(p)=0\}.
\]
Denote by $\G$ the one-dimensional foliation on $\Pj^3$ defined by $v$, see Remark~\ref{R:vec_field_def_fol}. Note that $\sing(\G)=\sing(v)$ if and only if $\codim(\sing(v)) \geq 2$.

Recall also that if $d=\deg(\F)$, then $\F$ is defined in homogeneous coordinates by 
\[
\Omega=\sum_{i=0}^3 A_i(x_0,\ldots,x_3) dx_i,
\]
where $A_0,\ldots,A_3$ are homogeneous polynomials of degree $d+1$.

If $\codim(\sing( v))=1$ or $\codim(\sing( v))=3$, then $\codim(\sing(\G))=3$, i.e., $\sing(\G)$ has isolated singularities. In fact, if $\codim(\sing( v))=1$, then, up to a linear automorphism of $\mathbb P^3$, we can assume that
\[
v=x_3 \frac{\partial}{ \partial x_3}.
\]
In this case, $\deg(\G)=0$ and $\sing(\G)=\{(0:0:0:1)\}$. On the other hand, if $\codim(\sing( v))=3$, then $\deg(\G)=1$, and the result follows simply because $\sing(\G)=\sing(v)$. Thus, we can apply Corollary~\ref{C:cod_3_divides} to conclude that $\G$ locally divides $\F$. Moreover, in the former case,
\[
\TF\cong \mathcal{O}_{\mathbb{P}^3}(1) \oplus \mathcal{O}_{\mathbb{P}^3}(1-d),
\]
while in the latter case,
\[
\TF\cong \mathcal{O}_{\mathbb{P}^3} \oplus \mathcal{O}_{\mathbb{P}^3}(2-d),
\]
see Remark~\ref{R:decomponability_G_D}. Note that the hypothesis that $\TF$ is locally free is not used in either case.

If $\codim(\sing(v)) = 2$, then $\sing(\G)=\sing(v)$. Writing
\[
v=S \cdot \mathrm{X}
\]
as in \eqref{E:jordan_vector}, we see that $S$ must have one of the following Jordan normal forms, where $a$ and $b$ are distinct nonzero complex numbers:
\begin{multicols}{2}
\begin{enumerate}[(I)]
\item\label{1:Mat_cases}
\(
\begin{pmatrix}
a & 0 & 0 & 0\\
0 & b & 0 & 0\\
0 & 0 & 0 & 0\\
0 & 0 & 0 & 0
\end{pmatrix}
\)
\item\label{2:Mat_cases}
\(
\begin{pmatrix}
a & 0 & 0 & 0\\
0 & a & 0 & 0\\
0 & 0 & 0 & 0\\
0 & 0 & 0 & 0
\end{pmatrix}
\)
\item\label{3:Mat_cases}
\(
\begin{pmatrix}
a & 1 & 0 & 0\\
0 & a & 0 & 0\\
0 & 0 & 0 & 0\\
0 & 0 & 0 & 0
\end{pmatrix}
\)
\item\label{4:Mat_cases}
\(
\begin{pmatrix}
a & 0 & 0 & 0\\
0 & a & 0 & 0\\
0 & 0 & b & 0\\
0 & 0 & 0 & 0
\end{pmatrix}
\)
\item\label{5:Mat_cases}
\(
\begin{pmatrix}
0 & 1 & 0 & 0\\
0 & 0 & 0 & 0\\
0 & 0 & 0 & 1\\
0 & 0 & 0 & 0
\end{pmatrix}
\)
\item\label{6:Mat_cases}
\(
\begin{pmatrix}
0 & 1 & 0 & 0\\
0 & 0 & 1 & 0\\
0 & 0 & 0 & 0\\
0 & 0 & 0 & 0
\end{pmatrix}
\)
\end{enumerate}
\end{multicols}

We begin by analyzing cases \eqref{1:Mat_cases}, \eqref{2:Mat_cases} and \eqref{3:Mat_cases} simultaneously. The vector field $v$ takes the following form:
\[
v=(ax_0 + \varepsilon x_1) \frac{\partial}{\partial x_0}+bx_1  \frac{\partial}{\partial x_1}.
\]
Here we allow $a=b$ (in \eqref{2:Mat_cases} and \eqref{3:Mat_cases}), $\varepsilon=0$ in \eqref{1:Mat_cases} and \eqref{2:Mat_cases} and $\varepsilon=1$ in \eqref{3:Mat_cases}. A simple verification shows that $\G$ has linear rank at least two in all three cases. 

Next, we will need $d \geq  1$, let us deal once and for all with the case $d=0$. It is well known that in a suitable coordinate system $\F$ can be defined in homogeneous coordinates by $\Omega=x_0 dx_1-x_1 d x_0$. Since
\[
\Omega=i_{\frac{\partial}{\partial x_2}}i_{\frac{\partial}{\partial x_3}} i_R (dx_0 \wedge \cdots \wedge dx_3),
\]
it follows from \eqref{E:split} that 
\[
\TF \cong \mathcal{O}_{\Pj^3}(1) \oplus \mathcal{O}_{\Pj^3}(1).
\]

Now assume that \(d \geq 1\). By Theorem~\ref{T:B}, it suffices to verify that every irreducible component of \(\sing(v)\) intersects \(\sing(\F)\). Moreover, we have
\[
\TF \cong \mathcal{O}_{\mathbb{P}^3} \oplus \mathcal{O}_{\mathbb{P}^3}(2-d).
\]

A straightforward verification shows that the relation $i_v \Omega=0$ implies that $A_0=x_1B$ and $A_1=-\frac{ax_0+\varepsilon x_1}{b}B$, for some homogeneous polynomial $B(x_0,\ldots,x_3)$ satisfying $\deg(B) \geq 1$. From the relation $i_R \Omega=0$ we get
\[
x_2A_2(0,0,x_2,x_3)+x_3A_3(0,0,x_2,x_3)=0,
\]
which means that $A_2(0,0,x_2,x_3)=x_3C(x_2,x_3)$ and $A_3(0,0,x_2,x_3)=-x_2C(x_2,x_3)$, for some homogeneous polynomial $C$ in the variables $x_2$ and $x_3$ and $\deg(C) \geq 1$. Therefore
\begin{equation}\label{E:int_sing_set}
\Omega(0,0,x_2,x_3)=C(x_2,x_3)(x_3dx_2-x_2 dx_3).
\end{equation}

The line 
\[
l_1=\{(x_0:\cdots:x_3):x_0=x_1=0\}
\] 
is contained in $\sing(v)$ in all three cases. Furthermore, in \eqref{1:Mat_cases} and \eqref{3:Mat_cases} it is the only one-dimensional irreducible component of $\sing(v)$. It follows from \eqref{E:int_sing_set} that $l_1 \cap \sing(\F) \neq \emptyset$ and the result follows. In \eqref{2:Mat_cases}, there is an extra irreducible component of dimension one contained in $\sing(v)$, given by the line 
\[
l_2=\{(x_0:\cdots:x_3):x_2=x_3=0\}.
\] 
As $v$ is also represented by 
\[
v-aR=-ax_2 \frac{\partial}{\partial x_2}-ax_3 \frac{\partial}{\partial x_3},
\]
similarly we show that
\begin{equation}\label{E:int_sing_set_II}
\Omega(x_0,x_1,0,0)=B(x_0,x_1,0,0)(x_1dx_0-x_0dx_1).
\end{equation}
It follows from \eqref{E:int_sing_set_II} that $l_2 \cap \sing(\F) \neq \emptyset$ and the result follows in case \eqref{2:Mat_cases}.

In \eqref{4:Mat_cases}, since $v$ is also represented by
\[
v-aR=(b-a)x_2 \frac{\partial}{\partial x_2}-ax_3 \frac{\partial}{\partial x_3},
\]
after a linear change of coordinates it follows that this case reduces to case \eqref{1:Mat_cases}.

It remains to analyze the two cases \eqref{5:Mat_cases} and \eqref{6:Mat_cases}, in which $v$ defines a nilpotent vector field on $\Pj^3$. If $v$ is given by \eqref{5:Mat_cases}, then
\[
\sing(v)=\{(x_0:\cdots:x_3):x_1=x_3=0\},
\] 
and all the singularities of $v$ have linear rank $1$. If $v$ is given by \eqref{6:Mat_cases}, then
\[
\sing(v)=\{(x_0:\cdots:x_3):x_1=x_2=0\},
\] 
and all the singularities of $v$ have linear rank $2$, with exception to the point $p=(1:0:0:0)$, which has linear rank $1$. Note that \eqref{5:Mat_cases} and \eqref{6:Mat_cases} correspond to \(v_1\) and \(v_2\) in \eqref{E:nil_P3}, respectively.

Thus, if $d \geq 2$ and $v$ is given by \eqref{5:Mat_cases} or \eqref{6:Mat_cases}, since in both cases $\sing(\G)$ consists of a single line containing $p$, from Lemma~\ref{L:divides_in_one_point} and Corollary~\ref{C:finite_check} we have that $\G$ locally divides $\F$ and $\TF$ splits. Besides, 
\[
\TF\cong \mathcal{O}_{\mathbb{P}^3} \oplus \mathcal{O}_{\mathbb{P}^3}(2-d).
\]

Finally, if $d=1$ and $v$ is given by \eqref{5:Mat_cases} or \eqref{6:Mat_cases}, we show that there exists a nonnilpotent holomorphic vector field on $\Pj^3$ tangent to $\F$, and the result follows from what we have already proved. Note that in this case
\[
\TF\cong \mathcal{O}_{\mathbb{P}^3} \oplus \mathcal{O}_{\mathbb{P}^3}(1).
\]

Assume that $d=1$ and $v$ is given by \eqref{5:Mat_cases}. We have used Maple to solve \eqref{E:System_linear}, then we have that $\F$ is defined by
\begin{align*}
\Omega=&(\alpha x_1 x_3 +\beta x_3^2) dx_0+(\alpha x_1 x_2 -2\alpha x_0 x_3 -\beta x_2 x_3-\gamma x_1 x_3-\delta x_3^2) dx_1+\\
&(-\alpha x_1^2-\beta x_1 x_3)dx_2+
(\alpha x_0 x_1+2\beta x_1 x_2-\beta x_0 x_3 +\gamma x_1^2+\delta x_1 x_3 )dx_3,
\end{align*}
with $\alpha,\beta,\gamma,\delta \in \C$. Since $\codim(\sing(\F)) \geq 2$, it follows that $\alpha \neq 0$ or $\beta \neq 0$. If $\alpha \neq 0$ we can assume, with no loss, that $\alpha=1$. Then a simple verification shows that
\[
W=(2x_0+\gamma x_1 +\beta x_2 +\delta x_3)\frac{\partial}{\partial x_0}+(x_1+\beta x_3)\frac{\partial}{\partial x_1}+x_2\frac{\partial}{\partial x_2}
\]
defines a holomorphic vector field on $\mathbb{P}^3$ tangent to $\F$, i.e., $i_W \Omega=0$. Note that $W$ is not nilpotent. On the other hand, if $\alpha=0$ and $\beta \neq 0$ we can assume, with no loss, that $\beta=1$. Then it can be easily verified that
\[
\delta x_1 \frac{\partial}{\partial x_0}+x_1 \frac{\partial}{\partial x_1}+(-\gamma x_1-x_2)\frac{\partial}{\partial x_2}
\]
defines a nonnilpotent vector field on $\mathbb{P}^3$ tangent to $\F$.

Finally, assume that $d=1$ and $v$ is given by \eqref{6:Mat_cases}. We have used Maple to solve \eqref{E:System_linear}, then we have that $\F$ is defined by
\begin{align*}
\Omega=&(\alpha x_2 x_3+\beta x_2^2) dx_0+(-\alpha x_1 x_3-\beta x_1 x_2) dx_1+\\
&(\gamma x_3^2+\alpha x_0 x_3+\delta x_2 x_3 -\beta x_0 x_2 +\beta x_1^2)dx_2+\\
&(-\gamma x_2 x_3-2\alpha x_0 x_2 +\alpha x_1^2 -\delta x_2^2)dx_3,
\end{align*}
with $\alpha,\beta,\gamma,\delta \in \C$. We claim that $\alpha=0$. In fact, since $\sing(v)=\{(x_1:\cdots:x_3):x_1=x_2=0\}$, a simple verification shows that if $\alpha \neq 0$ then
\[
\sing(v) \cap \sing(\F)=\{(1:0:0:0),(-\gamma:0:0:\alpha)\}.
\]
As $p_1:=(-\gamma:0:0:\alpha)$ is a singular point of $v$ with linear rank two, and $\TF$ is locally free at $p_1$, from part \eqref{two:C:not_free} of Corollary~\ref{C:not_free} we would obtain a contradiction. Since $\alpha=0$ and $\codim(\sing(\F)) \geq 2$, it follows that $\beta \neq 0$; thus we can assume, with no loss, that $\beta=1$. So it is easily seen that
\[
(\delta x_2+\gamma x_3)\frac{\partial }{\partial x_1}-x_1\frac{\partial }{\partial x_3}
\]
defines a holomorphic vector field on $\mathbb{P}^3$ tangent to $\F$, which is nonnilpotent if $\gamma \neq 0$. On the other hand, if $\gamma=0$, then one readily sees that
\[
\delta \frac{\partial }{\partial x_0}+\frac{\partial}{\partial x_3}
\]
is tangent to $\F$. This defines a holomorphic vector field on $\mathbb{P}^3$ with codimension one zeros.
\end{proof}

Let $v \neq 0$ be a holomorphic vector field on $\mathbb{P}^3$, and $\G$ be the one-dimensional foliation on $\mathbb{P}^3$ defined by $v$, see Remark~\ref{R:vec_field_def_fol}. It follows that $\deg(\G)=0$ or $\deg(\G)=1$. Furthermore, $\deg(\G)=0$ precisely when $\sing(v)$ has codimension one zeros. 

Regarding Theorem~\ref{T:structure}, assume that $v$ is tangent to $\F$. By analyzing its proof, it is clear that
\[
\TF \cong \mathcal{O}_{\Pj^3}(1) \oplus \mathcal{O}_{\Pj^3}(1-d)
\]
if $\deg(\G)=0$, while
\[
\TF \cong \mathcal{O}_{\Pj^3} \oplus \mathcal{O}_{\Pj^3}(2-d)
\]
if $\deg(\G)=1$, where $d=\deg(\F)$. One might wonder whether the vector field $v$ itself locally divides the foliation $\F$. This is not clear in the proof of Theorem~\ref{T:structure} when $\deg(\F)=1$ and $v$ is nilpotent and also when $\deg(\F)=0$. Of course, if this is the case, it is necessary that $\codim(\sing(v)) \geq 2$. If this condition is satisfied, the only exception where $v$ does not locally divide $\F$ is when $\deg(\F)=0$ and $\deg(\G)=1$. In fact, using Lemma~\ref{R:existence_functions} and the particular case of Lemma~\ref{L:divides_f_and_g}, we can obtain the following refinement of Theorem~\ref{T:structure}:

\begin{thm}
Let $\F$ be a codimension one holomorphic foliation on $\mathbb P^3$, with locally free tangent sheaf. Assume that there exists a nonzero holomorphic vector field $v$ on $\mathbb P^3$ tangent to $\F$. Denote by $\G$ the one-dimensional foliation on $\Pj^3$ defined by $v$, and let $\h$ be a one-dimensional foliation on $\Pj^3$ tangent to $\F$ of degree at most $1$. We have:
\begin{enumerate}
\item If $\deg(\F) \geq 3$ then $v$ locally divides $\F$ unless $\codim(\sing(v)) = 1$. Furthermore $\G$ locally divides $\F$ and $\h=\G$.
\item If $\deg(\F) = 2$ then $v$ locally divides $\F$ unless $\codim(\sing(v)) = 1$. Furthermore $\G$ locally divides $\F$ and $\deg(\h)=\deg(\G)$.
\item If $\deg(\F)=1$ then $v$ locally divides $\F$ unless $\codim(\sing(v)) = 1$. We can choose $\h$ satisfying $\deg(\h)=0$ or $\deg(\h)=1$, and for any such $\h$ we have that $\h$ locally divides $\F$. In particular $\G$ locally divides $\F$.
\item If $\deg(\F)=0$ then $v$ does not locally divide $\F$. We can choose $\h$ satisfying $\deg(\h)=0$ or $\deg(\h)=1$, and $\h$ locally divides $\F$ if and only if $\deg(\h)=0$.
\end{enumerate}
In particular $\TF$ splits and $\G$ locally divides $\F$ unless $\deg(\F)=0$ and $\deg(\G)=1$.
\end{thm}

In the proof of Theorem~\ref{T:structure}, the integrability of $\F$ was used only when $\deg(\F)=0$ or in the case in which the vector field tangent to $\F$ is nilpotent. If $\D$ is a degree $0$ distribution on $\mathbb{P}^3$, it follows that there exists a nonzero vector field on $\Pj^3$ tangent to $\D$. In fact, if $\Omega$ represents $\D$ in homogeneous coordinates in $\mathbb{C}^4$, using the coefficients of $\Omega$, it is easy to produce a homogeneous linear polynomial vector field $v$ on $\mathbb{C}^4$ satisfying $i_v \Omega=0$, in such a way that $v$ defines a nonzero vector field on $\Pj^3$ tangent to $\D$.

We conclude that a result analogous to Theorem~\ref{T:structure} for nonintegrable distributions cannot be true in the case of degree $0$, since the tangent sheaf of a nonintegrable distribution of degree $0$ on $\Pj^3$ does not split. Thus, we can state the following:

\begin{cor}\label{C:result_for-distrib}
Let $\D$ be a codimension one singular holomorphic distribution on $\mathbb P^3$, with locally free tangent sheaf. Assume that $\deg(\D) \geq 1$ and that there exists a nonnilpotent vector field on $\Pj^3$ tangent to $\D$. Then $\TD$ splits.
\end{cor}

Let $\mathscr{F}_k(d,n)$ denote the set of foliations of codimension $k$ and degree $d$ on $\mathbb{P}^n$, $n \geq 3$. It turns out that $\mathscr{F}_k(d,n)$ has a natural structure of quasi-projective variety, see for example \cite{Constant}. The space of foliations $\mathscr{F}_1(d,n)$ has $1$, $2$ and $6$ irreducible components, for $d=0,1,2$, respectively, see \cite{Jouanolou} for $d=0$ and $d=1$, and \cite{CerveauLins96} for $d=2$. In the codimension one case, we use the notation of \cite{RaRuJo2022}, where the logarithmic components $\Log(\mathbb P^n)(d_1,\ldots,d_k) \subset \mathscr{F}_1(d,n) $ also include the rational components.

When the degree of a codimension one foliation $\F$ on $\Pj^3$ is low, and by avoiding certain types of first integrals for $\F$, we can guarantee that there exists a nonzero holomorphic vector field on $\Pj^3$ tangent to such foliation $\F$. Consequently, we have the following corollary of Theorem~\ref{T:structure}:

\begin{cor}\label{C:low_cases}
Let $\F$ be a codimension one foliation on $\mathbb{P}^3$. Assume that one of the following conditions holds:
\begin{enumerate}
\item\label{1:C:low_cases} $\deg(\F)=0$ or $\deg(\F)=1$.
\item\label{2:C:low_cases} $\deg(\F)=2$ and $\F$ does not admit a rational first integral of the type $\frac{P(x_0,\ldots,x_3)}{H(x_0,\ldots,x_3)^3}$, where $P$ and $H$ are irreducible homogeneous polynomials with $\deg(P)=3$ and $\deg(H)=1$.
\item\label{3:C:low_cases} $\deg(\F)=3$, $\F$ does not admit a rational first integral and $\F$ is not in a logarithmic component.
\end{enumerate}
If additionally $\TF$ is locally free then $\TF$ splits.
\end{cor}
\begin{proof}
Set $h_0(\F)=\dim H^0(\mathbb{P}^3,\TF)$. Whenever there exists a nonzero vector field on $\mathbb{P}^3$ tangent to $\F$, which means that $h_0(\F) \geq 1$, the result follows from Theorem~\ref{T:structure}. Also, if $\F$ is a linear pull-back from $\mathbb{P}^2$ then $\TF$ splits, see \cite[Corollary~5.1]{CukiermanPereira}.

If \eqref{1:C:low_cases} holds and $\deg(\F)=0$ then $\TF$ splits, as we saw in the proof of Theorem~\ref{T:structure}. For $\deg(\F)=1$, we could apply \cite[Corollary~8.9]{10.1093/imrn/rny251}, but as we will use a similar idea in case \eqref{2:C:low_cases}, for the sake of completeness we will give a proof for this case. The irreducible components of $\mathscr{F}_1(1,3)$ are given by $\Log(\mathbb P^3)(1,2)$ and $\Log(\mathbb P^3)(1,1,1)$. Hence $\F$ has an invariant hyperplane, see for example \cite[Proposition~3.7]{RaRuJo2022}. 

Let $\omega$ be a polynomial $1$-form defining $\F$ on an affine coordinate system $ (x,y,z) \in \mathbb{C}^3 \subset \mathbb{P}^3$ such that the hyperplane at infinity corresponds to such a hyperplane invariant by $\F$. The decomposition of $\omega$ into homogeneous $1$-forms is given by
\[
\omega=\hat{\omega}_0+\hat{\omega}_1,\deg(\hat{\omega}_i)=i,i=0,1 \text{ and }\hat{\omega}_1 \neq 0.
\]
In this case, it is easy to produce a polynomial vector field $X \neq 0$ on $\mathbb{C}^3$ with degree at most $1$ satisfying $i_X\omega=0$, which gives rise to $v \in H^0(\mathbb{P}^3,\TF) -\{0\}$. 

If \eqref{2:C:low_cases} holds, assume first that $\F$ is in a logarithmic component. Then $\F$ is a pencil of quadrics or it has an invariant hyperplane, see again \cite[Proposition~3.7]{RaRuJo2022}. In the former case, the result follows from \cite[Theorem~D]{FaMaJe2023}, and in the latter case, let $\omega$ be a polynomial $1$-form defining $\F$ on an affine coordinate system as in \eqref{1:C:low_cases}. Since $\deg(\F)=2$, we have
\[
\omega=\hat{\omega}_0+\hat{\omega}_1+\hat{\omega}_2,\deg(\hat{\omega}_i)=i,i=0,1,2 \text{ and }\hat{\omega}_2 \neq 0.
\]

The integrability of $\omega$ is equivalent to $i_Y\omega=0$, where $Y$ is the polynomial vector field satisfying
\[
d\omega=i_Y(dx \wedge dy \wedge dz).
\]
If $Y \neq 0$, we have that $h_0(\F) \geq 1$. Otherwise, it follows that $\omega=df$, where $f(x,y,z)$ is a polynomial of degree $3$. There are three possibilities for $f$:
\begin{enumerate}[(i)]
\item $f$ is irreducible.
\item $f=ph$, $p$ is irreducible, $\deg(p)=2$ and $\deg(h)=1$.
\item $f=h_1 h_2 h_3$, $\deg(h_i)=1,i=1,2,3$, and $h_i$ and $h_j$ are relatively prime, for $i \neq j$. The last condition is due to $\deg(\F)=2$.
\end{enumerate}

Note that item i) is excluded by our hypothesis, otherwise it would give rise to a first integral of the form $F/x_3^m$, where $m$ is the degree of $f$ and 
\[
F(x_0,\ldots,x_3)
=
x_3^m
f\biggl(
\frac{x_0}{x_3},
\frac{x_1}{x_3},
\frac{x_2}{x_3}
\biggr).
\]
If ii) holds, after an affine transformation we can assume that $h=x$, then
\[
\omega=pdx+xdp.
\]
Since $p$ is irreducible, we have that $p_y:=\frac{\partial p}{\partial y} \neq 0$ or $p_z:=\frac{\partial p}{\partial z} \neq 0$. Then
\[
X=p_z \frac{\partial}{\partial y} - p_y \frac{\partial}{\partial z}
\]
provides an element in $H^0(\mathbb{P}^3,\TF) -\{0\}$. Similarly, we can produce an element in $H^0(\mathbb{P}^3,\TF) -\{0\}$ if $f$ is as in iii).

If $\deg(\F)=2$ and $\F$ is not in a logarithmic component, then $\F$ is a linear pull-back from $\mathbb{P}^2$ or it belongs to the exceptional component. In the latter case, using the notation of \cite{CerveauLins96}, the exceptional component is the closure of the orbit of $\overline{\mathcal{F}(\Gamma)}$ under $\Aut(\mathbb{P}^3)$. Since $h_0(\overline{\mathcal{F}(\Gamma)}) \geq 1$ we conclude that $h_0(\F) \geq 1$.

Finally, if \eqref{3:C:low_cases} holds, due to the structure theorem for degree three foliations on $\mathbb{P}^3$ of \cite{LorayPereiraTouzet17}, we have that either $\F$ is a linear pull-back from $\mathbb{P}^2$, or there is a degree one foliation by algebraic curves tangent to $\F$. This finishes the proof.
\end{proof}

We can conclude from Corollary~\ref{C:low_cases} that if $\F$ is a codimension one foliation on $\mathbb{P}^3$ with locally free tangent sheaf and not admitting an invariant hypersurface, then $\TF$ splits provided that $\deg(\F) \leq 3$. 

So far, the known examples of codimension one foliations on $\mathbb{P}^3$ whose tangent sheaf is locally free and does not split can be found in \cite[Theorem~E]{FaMaJe2023}: For every $k \ge 0$ there are rational foliations of type $(k+3,k+3)$ satisfying such property. Thus all such examples have a rational first integral, and for $k=0$ the respective foliation has degree $4$. Based on this and Corollary~\ref{C:low_cases}, we finish by presenting the following:

\begin{problem}
Let $\F$ be a codimension one foliation on $\mathbb{P}^3$ such that $\TF$ is locally free. If $\F$ does not have an invariant hypersurface, is it true that $\TF$ splits? More generally, if $\F$ does not admit a rational first integral, is it true that $\TF$ splits?
\end{problem}

\end{document}